\documentclass[10pt]{amsart}
\usepackage[margin=1.2in]{geometry}
\usepackage{colortbl}
\usepackage{color}
\usepackage{cite}
\newtheorem{theorem}{Theorem}[section]

\newtheorem{proposition}{Proposition}[section]

\theoremstyle{definition}

\theoremstyle{remark}
\newtheorem{remark}[theorem]{Remark}

\numberwithin{equation}{section}

\usepackage{graphicx}
\usepackage{subcaption}
\usepackage{multirow}
\usepackage{bbm}

\def\sinc{\mathop{\mbox{\rm sinc}}\nolimits}

\begin{document}

\title[High order conservative Semi-Lagrangian scheme for the BGK equation]{High order conservative Semi-Lagrangian scheme for the BGK model of the Boltzmann equation}

\author[S. Boscarino]{Sebastiano Boscarino}
\address{Sebastiano Boscarino\\
Department of Mathematics and Informatics  \\
University of Catania\\
95125 Catania, Italy} \email{boscarino@dmi.unict.it}

\author[S.-Y. Cho]{Seung-Yeon Cho}
\address{SeungYeon Cho\\
Department of Mathematics\\
Sungkyunkwan University\\
Suwon 440-746, Republic of Korea}
\email{chosy89@skku.edu}

\author[G. Russo]{Giovanni Russo}
\address{Giovanni Russo\\
Department of Mathematics and Informatics  \\
University of Catania\\
95125 Catania, Italy} \email{russo@dmi.unict.it}

\author[S.-B. Yun]{Seok-Bae Yun}
\address{Seok-Bae Yun\\
Department of Mathematics\\
Sungkyunkwan University\\
 Suwon 440-746, Republic of Korea}
\email{sbyun01@skku.edu}


\thanks{G. Russo is partially supported by ITN-ETN Horizon 2020 Project ModCompShock, Modeling and Computation on Shocks and Interfaces, Project Reference 642768. S. Boscarino is supported by the University of Catania (``Piano della Ricerca 2016/2018, Linea di intervento 2''). S.-B. Yun was supported by Samsung Science and Technology Foundation under Project Number SSTF-BA1801-02.}

\maketitle

\begin{abstract}
	In this paper, we present a conservative semi-Lagrangian finite-difference scheme for the BGK model. Classical semi-Lagrangian finite difference schemes, coupled with an L-stable treatment of the collision term, allow large time steps, for all the range of Knudsen number \cite{GRS,RS,SP}.  Unfortunately, however, such schemes are not conservative.  There are two main sources of lack of conservation. First, 
	when using classical continuous Maxwellian, conservation error is negligible only if velocity space is resolved with sufficiently large number of grid points. However, for a small number of grids in velocity space such error is not negligible, because the parameters of the Maxwellian do not coincide with the discrete moments. 
	Secondly, the non-linear reconstruction used to prevent oscillations destroys the translation invariance which is at the basis of the conservation properties of the scheme. As a consequence the schemes show a wrong shock speed in the limit of small Knudsen number. To treat the first problem and ensure machine precision conservation of  mass, momentum and energy with a relatively small number of velocity grid points, we replace the continuous Maxwellian with the discrete Maxwellian introduced in \cite{M}. The second problem is treated by implementing a conservative correction procedure based on the flux difference form as in \cite{RQT}.
	In this way we can construct a conservative semi-Lagrangian scheme which is Asymptotic Preserving (AP) for the underlying Euler limit, as the Knudsen number vanishes. The effectiveness of the proposed scheme is demonstrated by extensive numerical tests.
\end{abstract}

%\tableofcontents
\section{Introduction}
The dynamics of a non-ionized dilute gas at mesoscopic level is described by the celebrated Boltzmann equation \cite{Cercignani}.
The development of efficient numerical methods for its solution, however, constitutes a formidable challenge, due, among others, to the high dimensionality of the problem, the complicated structure of the collision operator, the need to preserve the collision invariants at a discrete level, and the stiffness issue arising when the  Knudsen number is very small.

In view of this situation, Bhatnaghar, Gross and Krook, in 1954, suggested a relaxation model of the Boltzmann equation, which now goes by the name of the BGK model \cite{BGK}. This approximation preserves several important qualitative features of the original Boltzmann equation, such as conservation of mass, momentum and energy, H-theorem and relaxation to equilibrium, and is now widely used as a simplified alternative to the Boltzmann equation because it is much less expensive to treat at a  numerical level.

Initial value problem for the BGK model on a periodic domain reads
\begin{align}\label{bgk}
\begin{split}
\frac{\partial{f}}{\partial{t}} + v \cdot \nabla_x{f} &= \frac{1}{\kappa}\left(\mathcal{M}(f)-f\right)\cr
f(x,v,0) &= f_0(x,v).
\end{split}
\end{align}
The velocity distribution function $f(x,v,t)$ represents the mass density of particles at point $(x,v)\in\mathbb{T}^d\times \mathbb{R}^d$ in phase space, at time $t>0$.
The Knudsen number $\kappa>0$ is defined as a ratio between the mean free path and a macroscopic characteristic length of the physical system.
The local Maxwellian $\mathcal{M}(f)$ is given by

\begin{align}\label{CM}
\mathcal{M}(f)(x,v,t):=\frac{\rho(x,t)}{\sqrt{\left(2 \pi T(x,t) \right)^d}} \exp\left(-\frac{|v-U(x,t)|^2}{2T} \right),	
\end{align}
where the macroscopic fields of local density $\rho(x,t) \in \mathcal{R}^+$, bulk velocity $U(x,t) \in \mathbb{R}^d$ and local temperature $T(x,t) \in \mathbb{R}^+$ are defined
through the following relation:
\begin{align}\label{Mom}
(\rho(x,t),\rho(x,t) U(x,t), E(x,t)^T = \langle f\phi(v)\rangle,
\end{align}
where
\[
\phi(v) = \left(1,v,\frac{1}{2}|v|^2\right)^T, \quad \textrm{and} \quad \langle g\rangle = \int_{\mathbb{R}^d} g(v)d v.
\]
%\begin{align}
%\begin{split}
%\rho(x,t)& = \int_{\mathbb{R}^d}f(x,v,t)dv,\cr
%\rho(x,t) U(x,t) &= \int_{\mathbb{R}^d}f(x,v,t)vdv,\cr
%E(x,t) &= \frac{1}{2}\int_{\mathbb{R}^d}f(x,v,t)|v|^2dv
%\end{split}
%\end{align}
%for
The physical quantity $E(x,t)$ is the total energy density per unit volume, and it is related to the temperature $T(x,t)$ by the following relation:
\[
E(x,t)=\frac{d}{2}\rho(x,t)T(x,t)+\frac{1}{2}\rho(x,t)|U(x,t)|^2.
\]
The BGK model (\ref{bgk}) satisfies the main properties of the Boltzmann equation such as: conservation of mass, momentum, and energy:
\begin{align}\label{cancellation}
\langle \mathcal{M}(f) \phi(v)\rangle = \langle f \phi(v)\rangle,
\end{align}
as well as entropy dissipation:
\[
\quad \, \int_{\mathbb{R}^{d}} (\mathcal{M}(f) - f) \ln f dv \le 0.
\]
%The relaxation operator satisfies the following cancellation property:
%\begin{align}\label{cancellation}
%\int_{\mathbb{R}^d}\big\{\mathcal{M}(f)-f\big\}(1,v,|v|^2)dv=0,
%\end{align}
%leading to the conservation of mass, momentum, energy
%\[
%\frac{d}{dt}\int_{\mathbb{R}^d\times\mathbb{R}^d}f (1,v,|v|^2)dxdv=0.
%\]
%The BGK model also satisfies the H-theorem:
%\[
%\frac{d}{dt}\int_{\mathbb{R}^d\times\mathbb{R}^d}f\ln fdxdv\leq 0,
%\]
%which follows from the dissipative estimate:
%\begin{align*}
%\int_{\mathbb{R}^d}\big\{\mathcal{M}(f)-f\big\}\ln fdv\leq 0.
%\end{align*}
Note that the equilibrium state clearly is the local Maxwellian determined by $f$. Indeed the collision operator vanished for $f = \mathcal{M} (f)$. Therefore, the BGK model
gives the correct Euler limit as $\kappa\rightarrow0$, i.e., the moments of solution to (\ref{bgk}), in the limit of vanishing Knudsen number, satisfy the macroscopic
compressible Euler equations for a monoatomic gas \cite{BGL,BGP}:
\begin{align}\label{LimE}
\begin{split}
\partial_t\rho  + \nabla \cdot (\rho U) &= 0,\cr
\partial_t(\rho U) + \nabla \cdot (\rho U \otimes U + pI) &= 0,\cr
\partial_t E +  \nabla \cdot ((E + p)U) &= 0
\end{split}
\end{align}
with pressure $p$ given by the constitutive relation to close the system (\ref{LimE}) $p = \rho T$.

Navier-Stokes equations can be derived by the Chapman-Enskog equation (see for example \cite{Chapman-Cowling}), by inserting a formal expansion of the distribution function $f$ in terms of the Knudsen number. To zero-th order one obtains compressible Euler's equations, while to first order in $\kappa$ one derives the Navier-Stokes equations associated to the BGK model.

We mention that such Navier-Stokes limit is slightly inconsistent with the one obtained from the Boltzmann equation,
in that the Prandtl number ${\rm Pr} = c_p\mu/k$ ($c_p$ is the specific heat at constant pressure, $\mu$ is the viscosity and $k$ the thermal conductivity) derived from the BGK model is numerically different from the value computed using the Boltzmann equation. Several techniques have been proposed to overcome this drawback, the most widely adopted being the  so-called Ellipsoidal BGK (ES-BGK), see \cite{ABLP,ALPP,Holway}. 
A semi-Lagrangian method for the ES-BGK model has recently been proposed and analyzed in \cite{RY}.

%%%%%%%%%%%%%%%%%%%%%%%%%%%%%%%%%%%%%%%%%%%%%%%%%%%%%%%%%%%%%%%%%%%%%%%%%%%%%%%%%%%%%%%%%%%%%%%%%%%%%%%%%%%%%%%%%%%%%%%%%%%%%%%%%%%%%%%%%%%%%%%
The aim of this paper is to develop high order conservative semi-Lagrangian (SL) finite-difference schemes for the BGK model of the Boltzmann equation.

Conservative semi-Lagrangian methods have recently attracted a lot of attention, especially in the context oc the Vlasov-Poisson model (see \cite{CMS-Vlasov-2010,FSB-Vlasov-2001}). 

General procedures have been developed for the construction of conservative SL schemes, as in \cite{QS-SL-2011}, however such procedures are often restricted to treat one dimensional problems.

SL schemes for BGK models have recently received increasing interest \cite{GRS,RS,RSY,RY,SP} since the SL treatment avoids the classical CFL stability restriction. Furthermore, the implicit treatment of the collision term, which can be easily computed, allows the methods to capture the underlying fluid dynamic limit. 

Unfortunately, however, classical SL schemes do not necessarily conserve the total mass, momentum and energy, and the error may become more relevant as  the Knudsen number gets smaller \cite{GRS}. 

We identify the cause of lack of conservation in the use of continuous Maxwellian in the collision term, and in the non-linear weights adopted in the high order non-oscillatory reconstruction, and propose a remedy based on the use of a discrete Maxwellian (as in \cite{M}) and on a conservative correction to fully restore the conservation properties of the schemes, such as the one adopted in \cite{RQT} in the case of the Vlasov-Poisson equation.

%In general this procedure implies a stability restriction on the time step.

% The outline is not required, but we show an example here.
The paper is organized as follows. 
Section 2 is devoted to first order schemes. It is shown that  the conservation error depends sensitively on the number of velocity grids, and the cause is identified in the use of a continuous Maxwellian in a discrete scheme. We prove that the SL schemes can be made conservative within round-off errors by adopting a discrete Maxwellian in place of the classical continuous one.

Section 3 considers high order SL schemes, which exhibit lack of conservation even with the use of the discrete Maxwellian in the collision term. A conservative correction is then adopted, which restores exact conservation of the methods (within round-off). 
Section 4 is devoted to linear stability analysis, to explain the stability limitations introduced by the conservative correction. 
In Section 5 we present several numerical tests, which confirm the expected accuracy and conservation properties of the proposed schemes, and provide numerical evidence of the AP property of the scheme towards the underlying fluid dynamic limit as the Knudsen number vanishes. 
At the end of we draw some conclusions.
The paper deals with 1D case. The extension to the multi-dimensional case is briefly mentioned in the conclusion.

\section{First Order Semi-Lagrangian schemes}\label{RevMot}

We start from  the basic first order semi-Lagrangian scheme \cite{RSY}, and gradually build up to derive our conservative high order semi-Lagrangian scheme (see Section \ref{high order section}).

\subsection{First order SL scheme}\label{Step I}
We start from the characteristic formulation of  (\ref{bgk}) :
\begin{equation}\label{characteristic}
\begin{split}
&\frac{df}{dt}=\frac{1}{\kappa}(\mathcal{M}(f)-f),\quad \frac{dx}{dt}=v,
\end{split}
\end{equation}
\noindent subject to the initial data: $f(x,v,0)=f_0(x,v)$.

We consider one dimensional problem in space and velocity, and we divide the spatial and velocity domain into uniform grids with mesh spacing $\Delta x$ and $\Delta v$, respectively. We also use uniform time step $\Delta t$. Given an computational domain,
$ [x_{min},x_{max}] \times [v_{min},v_{max}] \times [0,T^f] $, we denote the grid points by
\begin{align*}
\begin{array}{ll}
x_i=x_{min}+(i-\frac{1}{2})\Delta x,  &(i=1,...,N_x)\cr
v_j=v_{min} + j\Delta v,  &(j=0,...,N_v)\cr
t^n=n\Delta t, &(n=0,...,N_t),
\end{array}
\end{align*}
where $N_x$, $N_v+1$ and $N_t$ are the number of grid nodes in space, velocity and  time, respectively, so that $x_{max}  = x_{min} + N_x\Delta x$, $v_{max}=v_{min} + N_v\Delta v $ and $T^f=N_t \Delta t$.

Let $f_{i,j}^{n}$ denote a discrete approximation of $f(x_i,v_j,t^n)$ and $\phi(v_j)= \left(1,v_j,\frac{v_j^2}{2}\right)^T.$ Applying first order semi-Lagrangian implicit Euler (IE-SL) scheme to (\ref{characteristic}), we get
\begin{equation}\label{IE}
f_{i,j}^{n+1}=\tilde{f}_{ij}^n+\frac{\Delta{t}}{\kappa}\left(\mathcal{M}(f^{n+1}_{i,j})-f_{i,j}^{n+1}\right),
\end{equation}
where $\tilde{f}_{ij}^n$ is an approximation of $f(x_i-v_j\Delta{t},v_j,n\Delta{t})$ obtained by a suitable interpolation from $\{f^n_{ij}\}$. Note that
linear reconstruction will be sufficient for first order SL scheme, while a higher order non-oscillatory reconstruction is necessary for high order accuracy.
The Maxwellian $\mathcal{M}(f_{i,j}^{n+1})$ is given by
\begin{align}\label{1CM}
\mathcal{M}(f_{i,j}^{n+1}) =\frac{\rho_i^{n+1}}{\sqrt{2 \pi T_i^{n+1} }}\exp\left(-\frac{|v_j-U_i^{n+1}|^2}{2T_i^{n+1}}\right),
\end{align}
where discrete macroscopic moments are constructed from $f^{n+1}$ as follows:
\begin{align*}
\begin{pmatrix}
\rho_{i}^{n+1}\\
\rho_{i}^{n+1} U_{i}^{n+1}\\
E_{i}^{n+1}
\end{pmatrix}
&=\sum_{j=0}^{N_v}f_{i,j}^{n+1}
\phi(v_j)\Delta{v},
\end{align*}
which is equivalent to using midpoint rule in the computation of the momentums, Eq. \eqref{Mom}.

We now employ a technique that enables one to explicitly solve the implicit scheme \eqref{IE}.
For this, we multiply both sides in (\ref{IE}) by $\phi(v_j)$, sum over $j$, and use the property that
the moments of $\mathcal{M}(f_{i,j}^{n+1})-f^{n+1}_{i,j}$ up to second order vanish, to obtain
\begin{align*}
\sum_j^{N_v+1} f_{i,j}^{n+1}\phi(v_j) \Delta v= \sum_j^{N_v+1} \tilde{f}_{ij}^n\phi(v_j)\Delta v.
\end{align*}
This gives
\begin{align*}
\begin{pmatrix}
{\rho}_{i}^{n+1}\\
{U}_{i}^{n+1}\\
{E}_{i}^{n+1}
\end{pmatrix}
&=\begin{pmatrix}\tilde{\rho}_{i}^{n}\\
\tilde{U}_{i}^{n}\\
\tilde{E}_{i}^{n}\end{pmatrix},
\end{align*}
with
\begin{align*}
\begin{pmatrix}
\tilde{\rho}_{i}^{n}\\
\tilde{\rho}_{i}^{n}\tilde{U}_{i}^{n}\\
\tilde{E}_{i}^{n}
\end{pmatrix}
&=\sum_{j=0}^{N_v}\tilde{f}_{ij}^n
\phi(v_j)\Delta{v}, \quad \tilde{E}_{i}^{n}=\frac{1}{2}\tilde{\rho}_{i}^{n}|\tilde{U}_{i}^{n}|^2+\frac{1}{2}\tilde{\rho}_{i}^{n}\tilde{T}_{i}^{n}.
\end{align*}
Therefore,  we can legitimately replace $\mathcal{M}(f_{i,j}^{n+1})$ with $\mathcal{M}(\tilde{f}_{i,j}^{n})$, so that the scheme becomes
\begin{equation*}
f_{i,j}^{n+1}=\tilde{f}_{ij}^n+\frac{\Delta{t}}{\kappa}\left(\mathcal{M}(\tilde{f}^{n}_{ij})-f_{i,j}^{n+1}\right),
\end{equation*}
which gives
\begin{align}\label{numSol}
f_{i,j}^{n+1} = \frac{\kappa \tilde{f}_{ij}^n + \Delta t \mathcal{M}(\tilde{f}_{ij}^n)}{\kappa + \Delta t}.
\end{align}
This approach has been fruitfully used, for example, in \cite{GRS,PP, RSY, RS}. Note that scheme \eqref{numSol} allows us to use large CFL$>1$ numbers.% in order to obtain larger time steps.

Summarizing, we have the following procedure (see Fig.\ref{fig1First}):
\begin{enumerate}%{Procedure}
	\item Use linear interpolation to obtain $\tilde{f}_{ij}^n$ from $\{f_{i,j}^n\}$.
	\item Compute $\mathcal{M}(\tilde{f}_{ij}^n)$ from $\{\tilde{f}_{ij}^n\}$ by using the approximate macroscopic moments, i.e. $(\tilde{\rho}_{i}^{n},\tilde{U}_{i}^{n},\tilde{E}_{i}^{n})^T$.
	\item Compute numerical solution using \eqref{numSol}.
\end{enumerate}

\begin{figure}[h]\label{fig1}
	\centering
	\includegraphics[width=0.7\linewidth]{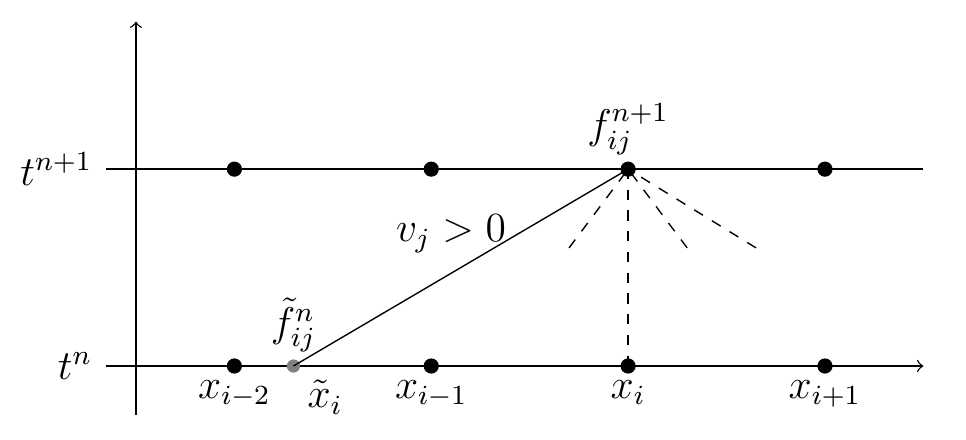}			
	\caption{Representation of the implicit first order scheme.}	\label{fig1First}
\end{figure}

We apply the scheme to the propagation of a single shock, where we can compare the numerical solution to the exact one, and therefore accurately check the conservation properties of the scheme.

\subsection{Test 1}
\label{sec:test1}
We apply IE-SL scheme (\ref{numSol}) to equation (\ref{bgk}) with
$f_0$ given by the Maxwellian w.r.t macroscopic quantities
\begin{equation*}
\displaystyle(\rho_0,u_0,p_0) = \left\{\displaystyle\begin{array}{lr}
\left(\frac{(\gamma+1)M^2}{(\gamma-1)M^2+2},\frac{2\sqrt{\gamma}(M^2-1)}{(\gamma+1)M},1+\frac{2\gamma(M^2-1)}{(\gamma+1)}\right), & \text{for } x\le 0.5\\[3mm]
(1,0,1), & \text{for } x>0.5.\\
\end{array}\right.
\end{equation*}
We take the Knudsen number $\kappa=10^{-6}$, the polytropic constant $\gamma = 3$ (corresponding to a polytropic gas with one degree of freedom per gas molecule) and Mach number $M = 2 $.

To prevent the solution from reaching the boundary, final time is taken $T_f=0.4$. We used free flow boundary conditions and performed the computation on $ (x, v) \in [0,5] \times [-20,20]$.

The results are summarized in Table \ref{tab1}, where the conservation errors are reported for various values of $N_x$ and $N_v$.

\begin{table}[hbt]
	\centering
	{\begin{tabular}{|ccccccc|}
			\hline
			\multicolumn{1}{ |c| }{} & \multicolumn{3}{ |c| }{IE-SL-Linear-CM, $N_x=100$} & \multicolumn{3}{ |c| }{IE-SL-Linear-CM, $N_x=200$} \\ \hline
			\multicolumn{1}{ |c }{$N_v$}& \multicolumn{1}{ |c  }{Mass} & \multicolumn{1}{ |c }{Momentum}& \multicolumn{1}{ |c| }{Energy} & \multicolumn{1}{ |c  }{Mass} & \multicolumn{1}{ |c }{Momentum}& \multicolumn{1}{ |c| }{Energy} \\ \hline	
			\multicolumn{1}{ |c|  }{$30$}& 3.63e-04 & 0.0012 & \multicolumn{1}{ c| }{0.0021} & 9.10e-04 & 0.0030 & 0.0051 \\
			\multicolumn{1}{ |c|  }{$40$}& 5.54e-08 & 3.26e-07 & \multicolumn{1}{ c| }{6.03e-07} & 1.15e-07 & 6.43e-07 & 1.25e-06 \\
			\multicolumn{1}{ |c|  }{$50$}& 8.55e-13 & 7.81e-12 & \multicolumn{1}{ c| }{1.43e-11} & 1.78e-12 & 1.54e-11 & 2.97e-11 \\
			\multicolumn{1}{ |c|  }{$60$}& 3.55e-14 & 4.96e-14 & \multicolumn{1}{ c| }{3.89e-14} & 7.45e-14 & 8.24e-14 & 7.23e-14 \\
			\multicolumn{1}{ |c|  }{$90$}& 3.24e-14 & 4.82e-14 & \multicolumn{1}{ c| }{3.77e-14} & 7.16e-14 & 7.32e-14 & 7.45e-14 \\
			\hline
	\end{tabular}}
	\captionof{table}{$CFL=4$, $\kappa=10^{-6}$. First order scheme, conservation error of discrete moments in relative $L_1$ norm for single shock with velocity domain $[-20,20]$. } \label{tab1}
\end{table}

% \begin{remark}\label{step1 rmk}
From the results we can make the following observations: 

\begin{enumerate}
	\item Table 1 shows that the first order IE-SL scheme with enough points in velocity space maintains conservation within machine precision, independently of the number of grid point in space;
	\item the same scheme with smaller number of points in velocity produces non-negligible conservation errors.
\end{enumerate}
This numerical evidence suggests that the convection part is conservative, while errors in conservation are a consequence of the numerical approximation of the relaxation term. 
The lack of conservation is indeed due to the use of a continuous Maxwellian on a discrete scheme in velocity: the parameters of the continuous Maxwellian do not coincide with the discrete moments, they are just approximated by them with spectral accuracy when the integrals are replaced by a summation. The spectral accuracy of the quadrature explains, for example, the dramatic drop of the conservation error when the number of points in velocity is increased from 40 to 50.

%This, in turn, is because, in the classical expression of a Maxwellian (\ref{1CM}), the relation between the
%parameters $\rho$, $U$, and $T$ and the function $M$ are obtained by integration, while numerically they are obtained by discrete summation. The quadrature formulas are spectrally accurate, therefore when the Maxwellian is fully resolved one obtains high accuracy.

%\textcolor{red}{This suggests that there is a strong dependence on $N_v$ when we use continuous Maxwellian (CM).
%The continuous Maxwellian is related to the velocity distribution function through velocity moments. Therefore, the conservation error inevitably depends on the
%numerical error of the quadrature formula. }
%\textcolor{red}{In Conclusion, the conservation law at machine precision is secured only with large number of velocity grids, which severely reduces the efficiency of the scheme.}

%\end{remark}

%\textcolor{red}{We see in Step II that the problem observed in (1) can be largely improved by employing the discrete Maxwellian in place of the continuous Maxwellian.
%To take care of the problem observed in (2),  we implement the conservative correction based on the flux difference form.}%using third-fifth order \emph{generalized} WENO reconstruction (G-WENO) \cite{CFR} in Step III.}\newline
%\end{remark}
\subsection{Classical SL scheme with the discrete Maxwellian}
\label{DiscMax}

In this section, we replace the continuous Maxwellian with the discrete Maxwellian to resolve the problem of strong dependence of the conservation error on the number of velocity grids.

\subsubsection{Discreate Maxwellian}
We start by briefly describing the discrete Maxwellian  introduced in \cite{M}. In that work, the author proved that a discrete entropy minimization problem has a uniqueness solution called the discrete Maxwellian($d\mathcal{M}$). Moreover, together with an assumption that if a set of discrete velocity points $\{v_j\}$ is rank $d+2$, there exist a unique vector $a(x,t) \in \mathbb{R}^{d+2}$ such that the following exponential characterization holds:
\begin{align}\label{DisM}
d\mathcal{M}(x,v_j,t):= \exp \left(a(x,t) \cdot \phi(v_j) \right),
\end{align}
if and only if there exists a set of $\{g_j>0\}_j$ such that
$$ \sum_j f(x,v_j,t) \phi(v_j)(\Delta v)^d =\sum_j g_j \phi(v_j)(\Delta v)^d $$
where $f(x,v_j,t)$ are given. The vector $a(x,t)$ is determined by solving the following non-linear system:
\begin{equation}\label{dm prb}
\sum_{j=0}^{N_v}  f(x,v_j,t) \phi(v_j)\Delta v=\sum_ {j=0}^{N_v} \exp \big(a(x,t) \cdot \phi(v_j) \big)  \phi(v_j)\Delta v.
\end{equation}
In practice, employing a Newton algorithm, we find $a(x,t)$ such that
\begin{equation}\label{dm newton}
\max_{1\leq \ell\leq 3}\left|\sum_{j=0}^{N_v} \left( f(x,v_j,t)-d\mathcal{M}(x,v_j,t) \right)  \phi_\ell(v_j)\Delta v\right|<tol
\end{equation}
for arbitrary small tolerance(tol). Throughout this paper, we take $tol$ to be the order of $10^{-14}$. Here, we denote the $\ell$th component of $\phi(v_j)$ by $\phi_\ell(v_j)$, $\ell=1,2,3$. With the use of discrete Maxewllian in \eqref{numSol}, 
\begin{align}\label{numSol dm}
f_{i,j}^{n+1} = \frac{\kappa \tilde{f}_{ij}^n + \Delta t d\mathcal{\mathcal{M}}(\tilde{f}_{ij}^n)}{\kappa + \Delta t},
\end{align}
and it is possible to prove the following estimate on the conservation error (see Appendix \ref{SM proof of conservation error}):
\begin{align*}
&\max_{1\leq \ell\leq 3}\left|\sum_{i = 1}^{N_x}\sum_{j=0}^{N_v}\left(f_{i,j}^{Nt} - f_{i,j}^0\right)\phi_\ell(v_j) \Delta v\Delta x\right|  \leq  \frac{N_t \Delta{t} }{\kappa + \Delta t }(x_{max}-x_{min}) tol.
\end{align*}

On the other hand, we recall that, in discrete velocity models, we need to take the velocity domain sufficiently large to secure correct profile of macroscopic moments, especially  when there is a large space variation of mean velocity $U$ and temperature $T$.
%Therefore, we fix a sufficiently large velocity domain for the accurate representation in time evolution (see \cite{M}.)

%Moreover, we need to secure appropriate velocity domain. This is due to the spatial dependence of the macroscopic quantities such afs velocity and temperature. The domain for velocity space is fixed during the computiation, therefore, it needs to be sufficiently large for the accurate representation of $d\mathcal{M}$ during time evolution, otherwise we cannot capture correct profile of macroscopic moments, especially if there is a large space variation of mean velocity $U$ and temperature $T$.

Therefore, it is necessary to balance the size of the velocity domain needed for the accurate computation of macrosocpic fields, and the efficient choice of smallest possible number of grids to guarantee the efficient performance of the scheme. (See Test 2 in Section 5.)

Such issues of the optimal choice of the grid points in velocity space %or, velocity grid adaptation in space,
is not considered here and will be left for future investigation.

\section{High order schemes and conservative correction}\label{high order section}

%High order schemes in space and time can be constructed using high order non oscillatory reconstructions coupled with Runge-Kutta or BDF methods for time advancement, as illustrated in \cite{GRS}.
Several techniques can be adopted to obtain  high order accuracy and to ensure the shock
capturing properties  near the fluid regime, avoiding
spurious oscillations. 
Here we consider two of the schemes adopted in \cite{GRS}, namely third order schemes obtained by 
combining high order methods in time (RK3 and BDF3) with a high order non-oscillatory spatial interpolation technique that we call generalized WENO (G-WENO)\cite{CFR} to obtain  high order accuracy and to ensure the shock capturing properties of the proposed schemes near the fluid regime, avoiding spurious oscillations (see the Section 5).

We repeat the same moving shock test using third order schemes, and the results are summarized in Table  
\ref{TabRK3} for SL schemes using Runge-Kutta time advancement (RK3-W35), and in Table \ref{TabBDF3} for the BDF-based SL schemes (BDF3-W35). 
Fully resolved high order schemes both in space and velocity produce finite conservation error, which is 
much larger than the conservation error  of the first order scheme, shown in Table \ref{tab1}. % We used a so restrictive CFL because  due to negative Temperature even for large Nx and small CFL.

%Note that for classical RK3+W35+CM and RK3+W35+CM SL schemes the high order accuracy in time is coupled with high order interpolation
%techniques in space.
% 
%Note from table 2 and 3, however, that we are getting non-negligible errors no matter how much we refine our velocity grids.
This indicates that there are cases where high order schemes may show even bigger conservation errors compared to those obtained by the first order scheme.

The main qualitative difference between first order and high order methods is that the former uses a fixed stencil for the linear interpolation at the foot of the characteristics, while high order non-oscillatory reconstructions such as G-WENO use a weighted sum of reconstructions on different stencils, the weight depending on the local regularity properties of the function to be reconstructed. 
As a result,  in the first order SL scheme the interpolation
weights are the same for all intervals, whereas in high order SL schemes, due to the nonlinearity of the non-oscillatory reconstruction, 
the interpolation weights are
not the same for all intervals, thus destroying the translation invariance which is at the basis of the conservation property of the schemes.

\begin{table}
	%	\centering
	%	{\begin{tabular}{|cccc|}
	%			\hline
	%			\multicolumn{1}{ |c| }{$L_1$  relative error } & \multicolumn{3}{ |c| }{Classical RK3+W35+CM} \\ \hline
	%			\multicolumn{1}{ |c }{$(N_x,N_v), CFL=2$}& \multicolumn{1}{ |c  }{Mass} & \multicolumn{1}{ |c }{Momentum}& \multicolumn{1}{ |c| }{Energy} \\ \hline	
	%			\multicolumn{1}{ |c|  }{$(100,42)$}& 1.28e-03 & 1.25e-02 & 1.41e-02 \\
	%			\multicolumn{1}{ |c|  }{$(100,50)$}& 1.06e-03 & 1.31e-02 & 1.47e-02 \\
	%			\multicolumn{1}{ |c|  }{$(100,60)$}& 1.43e-03 & 1.26e-02 & 1.49e-02 \\
	%			\multicolumn{1}{ |c|  }{$(100,90)$}& 1.35e-03 & 1.28e-02 & 1.48e-02 \\ \hline
	%			\multicolumn{1}{ |c|  }{$(200,42)$}& 1.54e-03 & 1.30e-02 & 1.45e-02 \\
	%			\multicolumn{1}{ |c|  }{$(200,50)$}& 1.30e-03 & 1.35e-02 & 1.51e-02 \\
	%			\multicolumn{1}{ |c|  }{$(200,60)$}& 1.68e-03 & 1.30e-02 & 1.53e-02 \\
	%			\multicolumn{1}{ |c|  }{$(200,90)$}& 1.60e-03 & 1.32e-02 & 1.53e-02 \\ \hline
	%			\multicolumn{1}{ |c|  }{$(400,42)$}& 1.68e-03 & 1.32e-02 & 1.47e-02 \\
	%			\multicolumn{1}{ |c|  }{$(400,50)$}& 1.42e-03 & 1.36e-02 & 1.53e-02 \\
	%			\multicolumn{1}{ |c|  }{$(400,60)$}& 1.80e-03 & 1.32e-02 & 1.55e-02 \\
	%			\multicolumn{1}{ |c|  }{$(400,90)$}& 1.73e-03 & 1.34e-02 & 1.54e-02 \\ \hline
	%			\multicolumn{1}{ |c|  }{$(800,60)$}& 1.86e-03 & 1.33e-02 & 1.55e-02 \\
	%			\multicolumn{1}{ |c|  }{$(800,90)$}& 1.80e-03 & 1.34e-02 & 1.55e-02 \\
	%			\hline
	%	\end{tabular}}
	\centering
	{\begin{tabular}{|ccccccc|}
			\hline
			\multicolumn{1}{ |c| }{ } & \multicolumn{3}{ c| }{Classical RK3-W35-CM} & \multicolumn{3}{ c| }{Classical RK3-W35-DM} \\ \hline
			\multicolumn{1}{ |c }{$(N_x,N_v)$}& \multicolumn{1}{ |c  }{Mass} & \multicolumn{1}{ |c }{Momentum}& \multicolumn{1}{ |c }{Energy} & \multicolumn{1}{ |c  }{Mass} & \multicolumn{1}{ |c }{Momentum}& \multicolumn{1}{ |c| }{Energy} \\ \hline	
			\multicolumn{1}{ |c|  }{$(100,42)$}& 1.28e-03 & 1.25e-02 & \multicolumn{1}{ c| }{1.40e-01} &1.22e-03
			&1.29e-02
			&1.47e-02\\
			\multicolumn{1}{ |c|  }{$(100,50)$}& 1.06e-03 & 1.31e-02 & \multicolumn{1}{ c| }{1.47e-02} & 1.06e-03
			&1.36e-02
			&1.47e-02
			\\
			\multicolumn{1}{ |c|  }{$(100,60)$}& 1.43e-03 & 1.26e-02 & \multicolumn{1}{ c| }{1.49e-02} & 1.43e-03
			&1.26e-02
			&1.49e-02
			\\
			\multicolumn{1}{ |c|  }{$(100,90)$}& 1.35e-03 & 1.28e-02 & \multicolumn{1}{ c| }{1.48e-02} &1.35e-03
			&1.28e-02
			&1.48e-02
			\\ \hline
			\multicolumn{1}{ |c|  }{$(200,42)$}& 1.54e-03 & 1.30e-02 & \multicolumn{1}{ c| }{1.45e-02} &1.48e-03
			&1.33e-02
			&1.52e-02
			\\
			\multicolumn{1}{ |c|  }{$(200,50)$}& 1.30e-03 & 1.35e-02 & \multicolumn{1}{ c| }{1.51e-02} &1.30e-03
			&1.35e-02
			&1.51e-02
			\\
			\multicolumn{1}{ |c|  }{$(200,60)$}& 1.68e-03 & 1.30e-02 & \multicolumn{1}{ c| }{1.53e-02} &1.68e-03
			&1.30e-02
			&1.53e-02
			\\
			\multicolumn{1}{ |c|  }{$(200,90)$}& 1.60e-03 & 1.32e-02 & \multicolumn{1}{ c| }{1.53e-02} &1.60e-03
			&1.32e-02
			&1.53e-02
			\\ \hline
			\multicolumn{1}{ |c|  }{$(400,42)$}& 1.68e-03 & 1.32e-02 & \multicolumn{1}{ c| }{1.47e-02} &1.61e-03
			&1.35e-02
			&1.54e-02
			\\
			\multicolumn{1}{ |c|  }{$(400,50)$}& 1.42e-03 & 1.36e-02 & \multicolumn{1}{ c| }{1.53e-02} &1.42e-03
			&1.36e-02
			&1.53e-02\\
			\multicolumn{1}{ |c|  }{$(400,60)$}& 1.80e-03 & 1.32e-02 & \multicolumn{1}{ c| }{1.55e-02} &1.80e-03
			&1.32e-02
			&1.55e-02
			\\
			\multicolumn{1}{ |c|  }{$(400,90)$}& 1.73e-03 & 1.34e-02 & \multicolumn{1}{ c| }{1.54e-02} &1.73e-03
			&1.34e-02
			&1.54e-02
			\\ \hline
			\multicolumn{1}{ |c|  }{$(800,60)$}& 1.86e-03 & 1.33e-02 & \multicolumn{1}{ c| }{1.55e-02} &1.86e-03
			&1.33e-02
			&1.55e-02
			\\
			\multicolumn{1}{ |c|  }{$(800,90)$}& 1.80e-03 & 1.34e-02 & \multicolumn{1}{ c| }{1.55e-02} &1.80e-03
			&1.34e-02
			&1.55e-02
			\\
			\hline
	\end{tabular}}
	\captionof{table}{$CFL=2$, $\kappa=10^{-6}$. High order schemes, conservation error of discrete moments in relative $L_1$ norm for single shock problem with velocity domain $[-20,20]$. %\textcolor{blue}{Please, add the results obtained with the discrete Maxwellian, to show that for high order scheme this is not enough to solve the problem. It is sufficient to add columns to tables \ref{TabRK3} and
		%	\ref{TabBDF3}
		%}
		%	\textcolor{magenta}{Cho: Could you please check the changed tables?}
	} \label{TabRK3}
\end{table}

\begin{table}
	%\centering
	%{\begin{tabular}{|cccc|}
	%		\hline
	%		\multicolumn{1}{ |c| }{$L_1$  relative error } & \multicolumn{3}{ |c| }{Classical BDF3+W35+CM} \\ \hline
	%		\multicolumn{1}{ |c }{$(N_x,N_v), CFL=2$}& \multicolumn{1}{ |c  }{Mass} & \multicolumn{1}{ |c }{Momentum}& \multicolumn{1}{ |c| }{Energy} \\ \hline	
	%		\multicolumn{1}{ |c|  }{$(100,42)$}& 1.73e-03 & 1.03e-02 & 1.39e-02 \\
	%		\multicolumn{1}{ |c|  }{$(100,50)$}& 1.58e-03 & 1.04e-02 & 1.38e-02 \\
	%		\multicolumn{1}{ |c|  }{$(100,60)$}& 1.73e-03 & 1.00e-02 & 1.38e-02 \\
	%		\multicolumn{1}{ |c|  }{$(100,90)$}& 1.75e-03 & 1.03e-02 & 1.40e-02 \\\hline
	%		\multicolumn{1}{ |c|  }{$(200,42)$}& 2.02e-03 & 1.10e-02 & 1.46e-02 \\
	%		\multicolumn{1}{ |c|  }{$(200,50)$}& 1.88e-03 & 1.11e-02 & 1.45e-02 \\
	%		\multicolumn{1}{ |c|  }{$(200,60)$}& 2.01e-03 & 1.07e-02 & 1.44e-02 \\
	%		\multicolumn{1}{ |c|  }{$(200,90)$}& 2.03e-03 & 1.10e-02 & 1.46e-02 \\
	%		\multicolumn{1}{ |c|  }{$(400,42)$}& 2.18e-03 & 1.14e-02 & 1.49e-02 \\ \hline
	%		\multicolumn{1}{ |c|  }{$(400,50)$}& 2.05e-03 & 1.15e-02 & 1.48e-02 \\
	%		\multicolumn{1}{ |c|  }{$(400,60)$}& 2.16e-03 & 1.11e-02 & 1.47e-02 \\
	%		\multicolumn{1}{ |c|  }{$(400,90)$}& 2.19e-03 & 1.14e-02 & 1.49e-02 \\ \hline
	%		\multicolumn{1}{ |c|  }{$(800,60)$}& 2.24e-03 & 1.13e-02 & 1.49e-02 \\
	%		\multicolumn{1}{ |c|  }{$(800,90)$}& 2.27e-03 & 1.16e-02 & 1.51e-02 \\
	%		\hline
	%\end{tabular}}
	\centering
	{\begin{tabular}{|ccccccc|}
			\hline
			\multicolumn{1}{ |c| }{ } & \multicolumn{3}{ |c| }{Classical BDF3-W35-CM} & \multicolumn{3}{ |c| }{Classical BDF3-W35-DM} \\ \hline
			\multicolumn{1}{ |c }{$(N_x,N_v)$}& \multicolumn{1}{ |c  }{Mass} & \multicolumn{1}{ |c }{Momentum}& \multicolumn{1}{ |c }{Energy} & \multicolumn{1}{ |c  }{Mass} & \multicolumn{1}{ |c }{Momentum}& \multicolumn{1}{ |c| }{Energy}\\ \hline	
			\multicolumn{1}{ |c|  }{$(100,42)$}& 1.73e-03 & 1.03e-02 & \multicolumn{1}{ c| }{1.39e-02} &1.72e-03
			&1.03e-02
			&1.39e-02
			\\
			\multicolumn{1}{ |c|  }{$(100,50)$}& 1.58e-03 & 1.04e-02 & \multicolumn{1}{ c| }{1.38e-02} &1.58e-03
			&1.04e-02
			&1.38e-02
			\\
			\multicolumn{1}{ |c|  }{$(100,60)$}& 1.73e-03 & 1.00e-02 & \multicolumn{1}{ c| }{1.38e-02} &1.73e-03
			&1.00e-02
			&1.38e-02
			\\
			\multicolumn{1}{ |c|  }{$(100,90)$}& 1.75e-03 & 1.03e-02 & \multicolumn{1}{ c| }{1.40e-02} & 1.75e-03 & 1.03e-02 & 1.40e-02\\ \hline
			\multicolumn{1}{ |c|  }{$(200,42)$}& 2.02e-03 & 1.10e-02 & \multicolumn{1}{ c| }{1.46e-02} &2.01e-03
			&1.10e-02
			&1.46e-02
			\\
			\multicolumn{1}{ |c|  }{$(200,50)$}& 1.88e-03 & 1.11e-02 & \multicolumn{1}{ c| }{1.45e-02} &1.88e-03
			&1.11e-02
			&1.45e-02
			\\
			\multicolumn{1}{ |c|  }{$(200,60)$}& 2.01e-03 & 1.07e-02 & \multicolumn{1}{ c| }{1.44e-02} & 2.01e-03 & 1.07e-02 & 1.44e-02\\
			\multicolumn{1}{ |c|  }{$(200,90)$}& 2.03e-03 & 1.10e-02 & \multicolumn{1}{ c| }{1.46e-02} & 2.03e-03 & 1.10e-02 & 1.46e-02\\ \hline
			\multicolumn{1}{ |c|  }{$(400,42)$}& 2.18e-03 & 1.14e-02 & \multicolumn{1}{ c| }{1.49e-02} &2.18e-03
			&1.14e-02
			&1.49e-02
			\\ 
			\multicolumn{1}{ |c|  }{$(400,50)$}& 2.05e-03 & 1.15e-02 & \multicolumn{1}{ c| }{1.48e-02} & 2.05e-03 & 1.15e-02 & 1.48e-02\\
			\multicolumn{1}{ |c|  }{$(400,60)$}& 2.16e-03 & 1.11e-02 & \multicolumn{1}{ c| }{1.47e-02} & 2.16e-03 & 1.11e-02 & 1.47e-02\\
			\multicolumn{1}{ |c|  }{$(400,90)$}& 2.19e-03 & 1.14e-02 & \multicolumn{1}{ c| }{1.49e-02} & 2.19e-03 & 1.14e-02 & 1.49e-02\\ \hline
			\multicolumn{1}{ |c|  }{$(800,60)$}& 2.24e-03 & 1.13e-02 & \multicolumn{1}{ c| }{1.49e-02} & 2.24e-03 & 1.13e-02 & 1.49e-02\\
			\multicolumn{1}{ |c|  }{$(800,90)$}& 2.27e-03 & 1.16e-02 & \multicolumn{1}{ c| }{1.51e-02} & 2.27e-03 & 1.16e-02 & 1.51e-02\\
			\hline
	\end{tabular}}
	\captionof{table}{$ CFL=2$, $\kappa=10^{-6}$. High order schemes, conservation error of discrete moments in relative $L_1$ norm for single shock problem with velocity domain $[-20,20]$. }  \label{TabBDF3}
\end{table}

\subsection{Conservative correction and discrete Maxwellian}
%As is observed in Table x.x the conservative property of the scheme constructed in step II cannot be preserved in the high order extentions.
%To resolve this problem, we employ a conservative modification

In subsection \ref{DiscMax}, we achieved a machine precision conservation error for first order scheme by implementing the discrete Maxwellian in place of the continuous one.
This remedy, however, is not sufficient in high order implementations, as was indicated in Table \ref{TabRK3}, \ref{TabBDF3}.
%where we observe that the conservation errors do not vanish in classical high order schemes even when we increase the number of velocity grid points.
%This suggests that there is an conservation error which can arise regardless of velocity discretization.
%\textcolor{red}{Recall that, in the first order scheme, we benefit from the telescoping cancellation as in (\ref{telescoping}), so that no additional errors other than that arise
%from the discrete relaxation operator arise and accumulate. In the case of the high-order reconstruction of the characteristic foot, we cannot always expect such nice property due, for example, to the presence of nonlinear weights.}
%The error generated in the reconstruction at characteristic foot
%using high order interpolations is the reason for such phenomena. In our case, the non-linear weights  employed in the G-WENO
%reconstruction prevents the cancellation of errors through telescoping process as in \eqref{telescoping}.

%In \cite{GRS}, for high order classical schemes, high order G-WENO reconstructions \cite{CFR} are used. Since non-linear weights are computed and take different values for each characteristic foot, it is difficult the telescoping summation of numerical solutions to be cancellend as in the linear interpolation case \eqref{telescoping} in which reconstruction weights are same for each $j$.

To overcome this, we modify the scheme \eqref{numSol dm} using the conservative correction procedure based on a flux difference form \cite{PPR,RQT} to derive our main scheme.

For clarity of exposition, we start by describing the procedure in the case of first order schemes, although its real benefit appears in its application to high order methods.

The conservative method can be viewed as a predictor-corrector method. It is based on a SL non-conservative prediction, and a conservative correction. 

%Such conservative correction secure cancellation through telescoping summation.
%Our first order conservative SL scheme (C-SL) reads:
%in order to make it to be cancelled in telescoping summation. This property can be achieved by treating convection term in a different way. As already used in kinetic context \cite{RQT} and in gas dynamics \cite{PPR}, we consider a conservative correction procedure based on a flux difference form.
%Such procedure guarantees the telescoping cancellation and hence conservation for the SL scheme, while introducing

%This form, together with the discrete Maxwellian, enables us to derive conservative high order schemes in Section 4.
%
%Table xx. suggests that, in the high order approximation
%additional error can arise from the descritization of the convection terms, therefore, the remedy of replacing the continuous maxwellian with the discrete maxwellian is not
%sufficient to gurantee the conservation property.

% In view of this, we modify further the scheme derived in Step II in order to make the SL scheme conservative in this step.
%As already used in kinetic context \cite{RQT} and in gas dynamics \cite{PPR}, we consider a conservative correction procedure based on a flux difference form.
%Such procedure guarantees conservation for the SL scheme, while introducing time step constraints for stability requirement (for a theoretical investigation, see \cite{RQT}).
%The scheme starts from the standard non-conservative procedure
%with backward characteristics tracing (\ref{IE}), then the conservative
%correction is performed by a flux-difference formulation.

\begin{figure}[h!]
	\centering
	\includegraphics[width=0.6\linewidth]{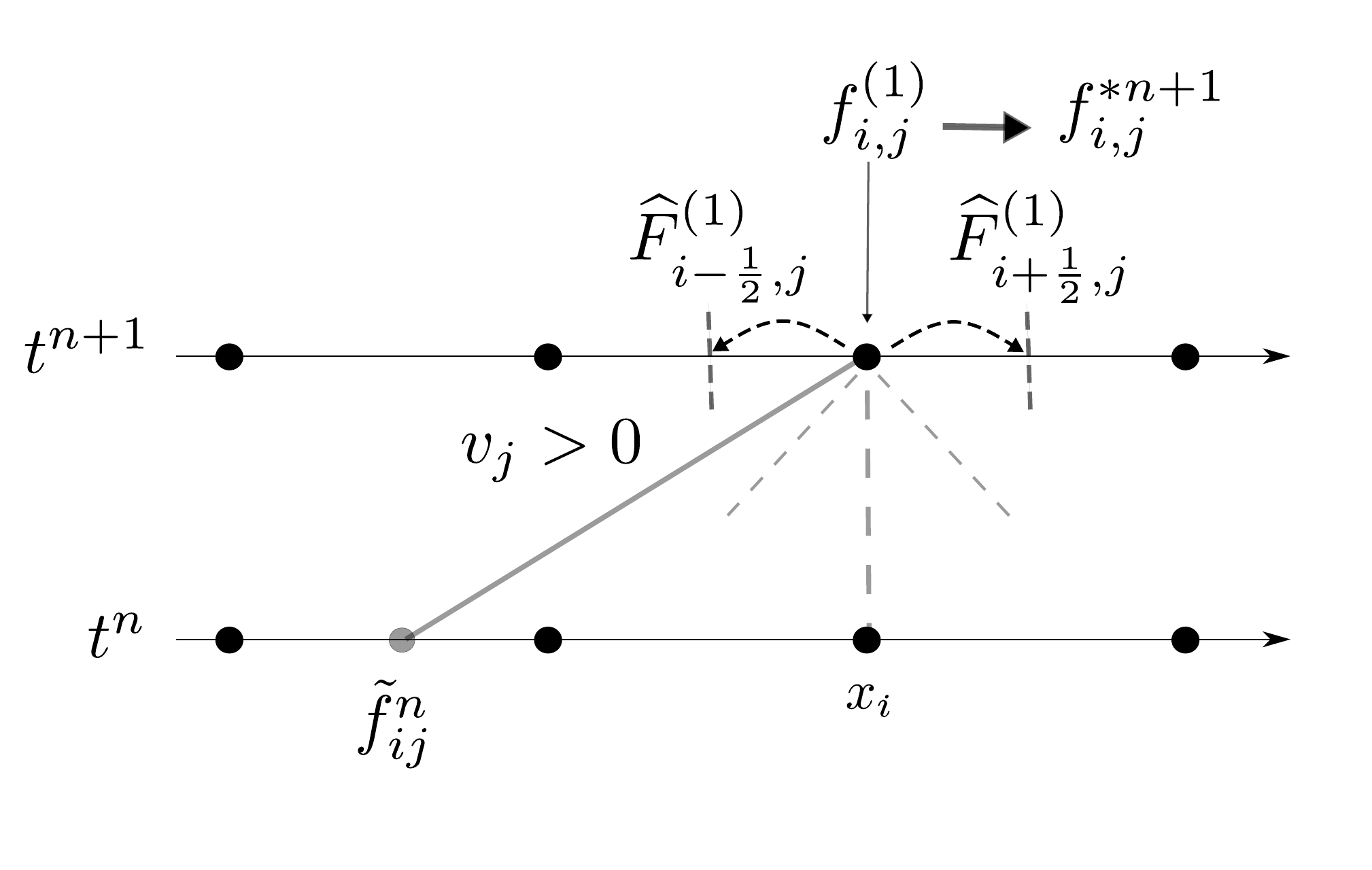}			\caption{Representation of first order scheme with conservative correction.}\label{fig2CSL}	
\end{figure}

%Recall the follwoing conservative SL scheme (C-SL):
With reference to Figure \ref{fig2CSL}, the first order scheme with conservative correction works as follows:
\begin{enumerate}
	\item using (\ref{numSol dm}), predict $f_{i,j}^{(1)}$  from $\{f_{i,j}^n\}$ at time $t^{n+1}$;\\
	\item reconstruct $\widehat{F}_{i+\frac{1}{2},j}^{(1)}$ and $\widehat{F}_{i-\frac{1}{2},j}^{(1)}$ from $\{v_j f_{i,j}^{(1)}\}$, by using a suitable high order reconstruction (see Sect.\ \ref{Recon}); \\
	\item compute the convective term ${f^*}_{i,j}^{n+1}$ by the conservative scheme
	\[
	{f^*}_{i,j}^{n+1}
	= f_{i,j}^n -\frac{\Delta{t}}{\Delta{x}}
	(\widehat{F}_{i+\frac{1}{2},j}^{(1)}- \widehat{F}_{i-\frac{1}{2},j}^{(1)});
	\]
	\item compute the discrete Maxwellian	${d\mathcal{M}^*}_{i,j}^{n+1}$  from ${f^*}_{i,j}^{n+1}$;\\
	\item update the solution  $f_{i,j}^{n+1}$ using
	\begin{equation}\label{Conservative Stable DM}
	f_{i,j}^{n+1}={f^*}_{i,j}^{n+1}+\frac{\Delta{t}}{\kappa}({d\mathcal{M}^*}_{i,j}^{n+1}-f_{i,j}^{n+1})
	\end{equation}
\end{enumerate}

Here 
$\widehat{F}$ is an accurate reconstruction of the flux  $v f$ in the sense of conservative finite difference \cite{Shu98}. 
%Note that we precompute all values of $f$ at $t_{n+1}$ using (\ref{numSol dm}) in order to reconstruct the cell boundary values of flux so that the method still remains in a semi-implicit form. 
We only present the formulation in 1D. Extension to more dimensions can be obtained performing  a dimension by dimension 1D reconstruction of the fluxes. 

\begin{remark}\label{APprop}
	The conservative correction imposes severe stability restriction
	on the CFL number for the C-SL schemes (for a theoretical investigation see \cite{RQT}). An accurate analysis for high order Runge-Kutta or BDF C-SL schemes will be given in Sect.5.
\end{remark}

\subsection{Spatial discretization}\label{Recon}
We restrict ourselves to 1D case and adopt a uniform grid $\Delta x:= x_{i+1}- x_i$.

\subsubsection*{Flux computation at the foot of the characteristics}
We use the {Generalized} WENO reconstruction (G-WENO) introduced in \cite{CFR} for non-oscillatory high-order reconstruction of $\tilde{f}^n_{ij} $. The main advantage of such a reconstruction is its use of polynomial weights, which provide a general framework to implement WENO interpolation on any points in a cell. See Appendix \ref{G-Weno} for details.

In our C-SL scheme, we need an accurate approximation of the convection term: $v\partial_xf$. For this, we set $F(f):=vf$,
and look for a function $\widehat{F}$ such that
\begin{align}\label{cell average}
F_i = \frac{1}{\Delta x} \int_{x_{i-\frac{1}{2}}}^{x_{i+\frac{1}{2}}} \widehat{F}dx.
\end{align}
where $F_i = F(f(x_i,v,t))$. Then we can compute the convection term using the following relation:
\begin{align}\label{Flux}
\partial_x F_i = \frac{1}{\Delta x} \left(\widehat{F}(x_{i+\frac{1}{2}}) -  \widehat{F}(x_{i-\frac{1}{2}})\right).
\end{align}
To compute $\widehat{F}(x_{i\pm\frac{1}{2}})$, we use the classical WENO reconstruction in \cite{Shu98} to guarantee non-oscillatory high-order approximation of $\widehat{F}_{i\pm\frac{1}{2}}$. In this reconstruction, we actually find a piecewise polynomial function that interpolates $\{F_i\}_{i=1 \dots N_x}$. Since those polynomials contain discontinuity at cell boundaries $x_{i\pm \frac{1}{2}}$, it is necessary to pick the correct direction where information comes from. For this reason, upwinding is introduced by flux splitting:
$$F=F^+ + F^-$$
where
\begin{align*}
F^+(f) =
\begin{cases}
vf, \quad v>0\\
0, \quad \text{otherwise}	
\end{cases}
,\quad
F^-(f) =
\begin{cases}
0, \quad v> 0\\
vf, \quad \text{otherwise}	
\end{cases},
\end{align*}
so that (\ref{cell average}) can be rewritten as $F_i=F_i^+ + F_i^-$, where 
\[
F^\pm(x) = \frac{1}{\Delta x} \int_{x-\Delta x/2}^{x+\Delta x/2}  \widehat{F}^\pm(\xi)\,d\xi.
\]
The half fluxes $\widehat{F}^\pm(x)$ are obtained by piecewise polynomial reconstruction:
\[
\widehat{F}^\pm(x) = \sum_i\chi_i(x)\widehat{F}^\pm_i(x)
\]
where $\chi_i(x)$ denotes the characteristic function of interval $[x_{i-1/2},x_{i+1/2}]$.

Then, by standard WENO process \cite{Shu98}, we reconstruct $\widehat{F}_i^\pm(x)$  from $\{F_i^\pm\}$. Finally, our numerical flux is obtained as follows: 
\[
\widehat{F}_{i+\frac{1}{2}}=\widehat{F}_{i}^+(x_{i+1/2}) + \widehat{F}_{i+1}^-(x_{i+1/2}).\]

\subsection{Time discretization}
%In this section, we provide high order time discretization for (\ref{bgk}).
%In Section 2, we used a simple first order method, implicit Euler method.
High order discretization in time can be obtained by Runge-Kutta methods (RK) or backward differentiation formulas (BDF) \cite{HWN}.
For the sake of simplicity, we again consider the one-dimensional problem in space and velocity with uniform grid in time.
\subsubsection*{Runge-Kutta methods}
%A Runge-Kutta schemes can be represented through the Butcher's table
%\begin{align}\label{Butcher}
%\begin{array}{c|c}
%c & A\\
%\hline
%& b^T
%\end{array}
%\end{align}
%with the $s \times s$ matrix $A=(a_{kl})$, and by the coefficients vectors $c= (1,...,c_s)^T$ and
%$b=(b1,...,b_{s})^T$ , which are derived by imposing accuracy and stability constraints \cite{HWN}.

%Note that when $\kappa$ is small, the ODE system (\ref{characteristic}) becomes \emph{stiff}. Stiff equations require implicit RK methods and concepts of stability, i.e. $A$-stability, \cite{HW}. In this section, we focus on a class of implicit Runge-Kutta methods, \emph{diagonally implicit Runge-Kutta} (DIRK) methods. DIRK methods are characterized by the low triangular $s \times s$ matrix $A=(a_{kl})$ such that $a_{ij} = 0$ if $j>i$ for $i = 1, ..., s$.

Our system (\ref{characteristic}) becomes \emph{stiff} as
$\kappa \rightarrow \infty$. To overcome this difficulty, we need stable schemes. In view of this, \emph{L-stable diagonally implicit Runge-Kutta} (DIRK) methods
provide a balanced
performance between stability and efficiency \cite{HW}.

DIRK methods can be represented using the Butcher's table
\begin{align}\label{Butcher}
\begin{array}{c|c}
c & A\\
\hline \\[-3mm]
& b^T
\end{array}
\end{align}
where $A=(a_{kl})$ is a $s\times s$ lower triangle matrix and $c= (c_1,...,c_s)^T$ and
$b=(b_1,...,b_{s})^T$ are coefficients vectors \cite{HWN}.

In order to guarantee $L$-stability, here we make use of {\em stiffly accurate} schemes (SA), i.e.\ schemes for 
which the last row of matrix $A$ is equal to the vector of weights: $a_{s,j} = b_j, j=1,\ldots,s$. This will ensure that 
the absolute stability function vanishes at infinity. As a consequence, an $A$ -stable scheme which is SA is also $L$-stable 
\cite{HW}.

%	 derived by imposing accuracy and stability constraints

%Stiff equations require implicit RK methods and concepts of stability, i.e. $A$-stability, \cite{HW}. In this section, we focus on a class of implicit Runge-Kutta methods, \emph{diagonally implicit Runge-Kutta} (DIRK) methods. DIRK methods are characterized by the low triangular $s \times s$ matrix $A=(a_{kl})$ such that $a_{ij} = 0$ if $j>i$ for $i = 1, ..., s$.

%In \cite{GRS,RS,SP}, the relaxation operator in (\ref{bgk}) was dealt with an $L$-stable DIRK schemes.
%We point out that, a sufficient condition that guarantees that an $A$-stable implicit RK scheme is $L$-stable is that the the following condition is satisfied: $A^Te_s = b^T$, i.e. the numerical solution is equal to the last stage value. Such
%schemes are called \emph{stiffly accurate},(see  for details \cite{HWN}).
% and this means that $f^{(s)}(x,v) = f^{n+1}(x,v)$. Note that a implicit RK scheme that satisfy such condition  is said \emph{stiffly accurate}. %ction, we obtain DIBDF3 with WENO35 and CWENO35 reconstruction

\begin{figure}[t]
	\centering
	\includegraphics[width=0.6\linewidth]{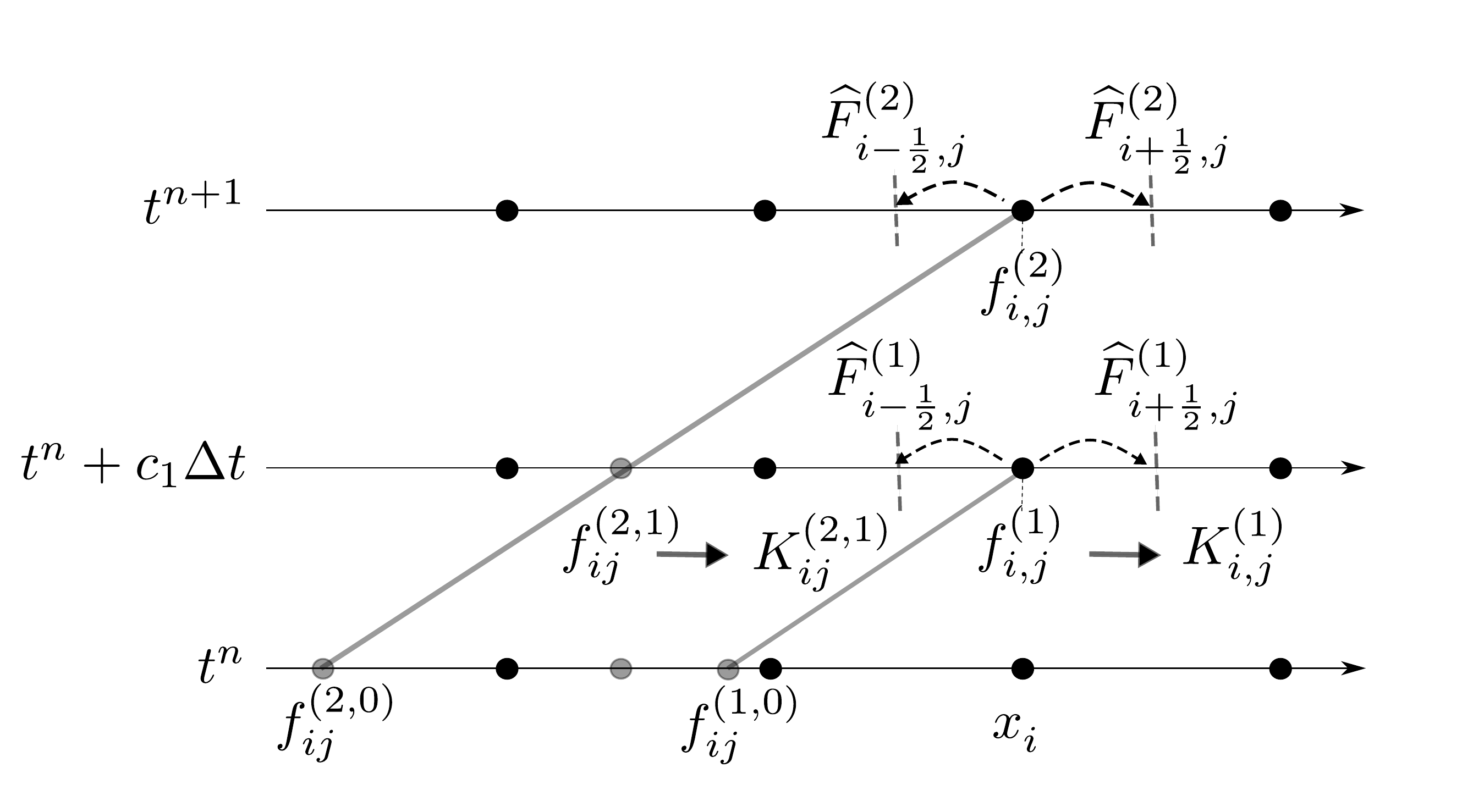}			\caption{Representation of DIRK2 scheme with conservative correction.}\label{DIRK2}	
\end{figure}

Now, we illustrate our L-stable DIRK schemes to approximate the characteristic system \eqref{characteristic} coupled with the conservative correction and the discrete Maxwellian.
%Here we illustrate the application of the DIRK schemes to the characteristic formulation of the BGK equation (first equation in (\ref{characteristic})) coupled with the conservative correction and the discrete Maxwellian.
%Note that high order methods in time are coupled with high order interpolation techniques in space as GWENO \cite{CGR}.
In the following,  $f_{ij}^{(k,\ell)}$, $\ell = 0, ..., s,$ denotes the $\ell$-th stage value computed along the $k$-th characteristic corresponding to each $x_i$ and $v_j$, see Figures \ref{DIRK2}. For example, in the case of $\ell = 0$, $f_{ij}^{(k,0)}$ is the approximation of  $f(x_i-c_k\Delta t v_j ,v_j, t^n)$ reconstructed from $\{ f^n_{i,j}\}$.
%We point out that using a semi-Lagrangian DIRK scheme with $s$ stages, the $\ell$-th stage value computing along the $k$th-characteristic and corresponding to each $x_i$ and $v_j$, will be denoted by a  $f_{ij}^{(k,\ell)}$ with $\ell = 0, ..., s$, (see Figures \ref{DIRK2} and \ref{DIRK3}). For $\ell = 0$, $f_{ij}^{(k,0)} := f(x_i-c_k\Delta t v_j ,v_j, t^n)$, with $f$ a  suitable reconstruction of $f$ obtained from $\{ f^n_{i,j}\}$.
We also define the RK flux $K_{ij}^{(k,\ell)} $ by
\begin{align*}
K_{ij}^{(k,\ell)} &= \frac{1}{\kappa}(d\mathcal{M}_{ij}^{(k,\ell)}-f_{ij}^{(k,\ell)}), \quad \ell = 1, ...,s.
\end{align*}

%\textcolor{blue}{We do not show the figure for RK3}
%\begin{figure}[t]
%	\centering
%	\includegraphics[width=1\linewidth]{Figures/RK3}			\caption{Representation of DIRK3 scheme with conservative correction.}	\label{DIRK3}
%\end{figure}

\subsubsection{Algorithm DIRK}\label{A_DIRK}

For $k=1,\dots,s$.

\begin{itemize}
	\item \emph{Non-conservative step}
	\begin{enumerate}
		\item Compute $f_{ij}^{(k,0)}$ in $x_{ij}^{(k, 0)}:= x_{i}-c_k v_j\Delta t$ along the $k$-th characteristic by interpolation from $\{f_{i,j}^{n}\}$.
		
		\item %For $\ell = k$,
		Compute:
		\[
		f^{(k)}_{i,j} = f_{ij}^{(k,0)} + \Delta t \sum_{\ell = 1}^{k -1} a_{k \ell} K_{ij}^{(k, \ell)} + \frac{\Delta t}{\kappa}a_{kk}\left( d\mathcal{M}^{(k)}_{i,j} - f_{i,j}^{(k)}\right)
		\]
		where $d\mathcal{M}^{(k)}_{i,j}$ is computed imposing, within some tolerance, that
		\begin{align}\label{dMRK}
		\sum_{j} \phi_j \,d\mathcal{M}_{i,j}^{(k)} \Delta v= \sum_{j}\phi_j \,(f^{(k, 0)}_{ij} + \Delta t \sum_{\ell = 1}^{k -1} a_{k \ell} K_{ij}^{(k, \ell)})\Delta v, \quad \phi_j = 1, v_j, v_j^2/2.
		\end{align}
		
		\item Compute:
		\begin{align*}
		K_{i,j}^{(k)} &= \frac{1}{\kappa}\left(d\mathcal{M}_{i,j}^{(k)}-f_{i,j}^{(k)}\right).
		\end{align*}
		
		\item Compute the RK flux $K_{ij}^{(\ell,k)}$ in $x_{ij}^{(\ell, k)}:= x_{i}-(c_\ell-c_k)v_j\Delta t$ with $\ell = k+1, \cdots, s$ along the $\ell$-th charactheristc by interpolation from $\{K_{i,j}^{(k)}\}$.% from (\ref{f_k})
		\item Reconstruct $\widehat{F}^{(k)}_{i+1/2,j}$ from $\{ v_j f_{i,j}^{(k)}\}$ using G-WENO reconstruction \cite{Shu98}  within a \emph{finite difference formulation} (fd).
		
	\end{enumerate}
	
	\item	\emph{Conservative correction step}
	\begin{enumerate}
		\item Compute the conservative convection: $$f^{*}_{i,j} = f_{i,j}^{n} + \frac{\Delta t}{\Delta x} \sum_{\ell = 1}^{s} b_{\ell} \left(\widehat{F}^{(\ell)}_{i+1/2,j} - \widehat{F}^{(\ell)}_{i-1/2,j} \right).$$
		\item Compute conservative solution:
		$$
		f^{n+1}_{i,j} = f^{*}_{i,j} + \Delta t \sum_{\ell = 1}^{s -1} b_{\ell} K_{i,j}^{(\ell)} + \frac{\Delta t}{\kappa}b_{s}\left( d\mathcal{M}^{(*)}_{i,j} - f_{i,j}^{n+1}\right)
		$$
		%	with
		%\begin{align}\label{dM}
		%\sum_{j} \phi_j \,dM_{i,j}^{(**)} = \sum_{j}\phi_j \,(f^{(**)}_{i,j} + \Delta t \sum_{\ell = 1}^{k -1} a_{k \ell} K^{(k, \ell)}), \quad \phi_j = 1, v_j, v_j^2.
		%\end{align}
		where ${d\mathcal{M}}_{i,j}^{(*)}$ is computed imposing, within some tolerance, that
		\begin{align}
		\sum_{j} \phi_j \,{d\mathcal{M}}_{i,j}^{(*)} \Delta v = \sum_{j}\phi_j \,f^{*}_{i,j}\Delta v, \quad \phi_j = 1, v_j, v_j^2/2.
		\end{align}
		
	\end{enumerate}
\end{itemize}
In the previous expression, the terms containing $K_{i,j}^{(\ell)}$ vanish because the intermediate Maxwellians have the same moments of the stage values. 
%
%After the above iteration is completed, we set  $f^{n+1}_{i,j} = f^{(s)}_{i,j}$ because the DIRK scheme is stiffly accurate (SA), and move on to the next step. 

\subsubsection*{BDF methods}
%In this section we present another family of high order methods for the time discretization,
Another time discretization we use for the stable approximation of stiff problems (\ref{characteristic}) is  the backward differentiation formula (BDF) (see \cite{HW}) whose general form is given by
\[
BDF:y^{n+1}= \sum_{k=1}^{s}a_{k}\,y^{n+1-k} + \beta_s \,\Delta t \, g(y^{n+1},t_{n+1})\\
\]
with $\beta_s \neq 0$.
%for stable approximation for stiff problem (\ref{})
For our work, we use BDF2 and BDF3:
\begin{align}\label{BDFschemes}
\begin{array}{l}
\displaystyle	BDF2: y^{n+1}= \frac{4}{3}y^n - \frac{1}{3}y^{n-1} + \frac{2}{3} \Delta t \, g(y^{n+1},t_{n+1}),\\[3mm]
BDF3:
\displaystyle	y^{n+1}= \frac{18}{11}y^n - \frac{9}{11}y^{n-1} + \frac{2}{11}y^{n-2} + \frac{6}{11} \, \Delta t \, g(y^{n+1},t_{n+1}).
\end{array}
\end{align}	

%	BDF schemes have some advantages over DIRK since a smaller number of numerical determination of the discrete Maxwellian is needed and fewer interpolations are required.
%	
BDF schemes have some advantages over DIRK since a smaller number of numerical determination
of the discrete Maxwellian and fluxes are needed and fewer interpolations are required. For BDF2 and BDF3, there is only one stage in which we have to compute the discrete Maxwellian and fluxes while two and three stages are required for DIRK2 and DIRK3 schemes respectively. 
%For BDF2 and BDF3, there are two steps in which we have to compute the discrete Maxwellian while four and six steps are required for DIRK2 and DIRK3 schemes respectively. 
Moreover, BDF2 and BDF3 schemes require two and three steps for interpolations whereas DIRK2 and DIRK3 schemes require three and six steps respectively. The price to pay is that BDF has more severe stability restriction than  DIRK (See Section 5).

\subsubsection{Algorithm BDF} Let $a_{k}$, and $\beta_s$ be the coefficients of a BDF method of order $s$. Given a discrete approximation $\{f^n_{ij}\}$ of the distribution function at time $t_n$, $\{f^{n+1}_{ij}\}$ is computed by the following steps
\begin{itemize}	
	%	\[
	%	BDF: y^{n+1}= \sum_{k=1}^{s}a_ky^{n+1-k} + \beta_{s+1} \Delta t g(y^{n+1},t_{n+1})\\
	%	\]
	%	\[
	%	(s=1) \quad BDF1 : a_1=1 , \beta_2=1\\
	%	\]
	%	\[
	%	(s=2) \quad BDF2 : a_1=\frac{4}{3}, a_2=-\frac{1}{3} , \beta_3=\frac{2}{3}\\
	%	\]
	%	\[
	%	(s=3) \quad BDF3 : a_1=\frac{18}{11}, a_2=-\frac{9}{11} ,a_3=\frac{2}{11}, \beta_4=\frac{6}{11}\\
	%	\]
	\item \emph{Non-conservative step}.
	\begin{enumerate}
		\item For $k=1,\dots,s$, interpolate $f_{ij}^{n,k}= f(x_i-k v_j \Delta{t}, v_j, t^{n+1-k})$ in $x_{ij}^{k}:= x_{i}-k v_j\Delta t$ from $\{f_{i,j}^{n+1-k}\}$ with a suitable \emph{generalized} WENO reconstruction in \cite{CFR}.
		\item Compute $f_{i,j}^{*}=\sum_{k = 1}^{s} a_{k}f_{ij}^{n,k}$ and
		$$
		f^{(1)}_{i,j} = f^{*}_{i,j}  + \beta_s\frac{\Delta t}{\kappa}\left( {d\mathcal{M}}_{i,j}^{(1)} - f_{i,j}^{(1)}\right)
		$$
		%$f^{(1),n+1}_{i,j} = \left( \kappa f_{i,j}^{*}  + \beta \Delta t {dM}_{i,j}^{(1),n+1} \right)\big/\big( \kappa  + \beta \Delta t \big)$
		
		where		${d\mathcal{M}}_{i,j}^{(1)}$ is computed imposing, within some tolerance, that
		\begin{align}\label{dM}
		\sum_{j} \phi_j \,d\mathcal{M}_{i,j}^{(1)} \Delta v= \sum_{j}\phi_j \,f^{*}_{i,j}\Delta v, \quad \phi_j = 1, v_j, v_j^2/2.
		\end{align}
	\end{enumerate}	
	\item	\emph{Conservative correction step}
	\begin{enumerate}
		\item Reconstruct $\widehat{F}^{(1)}_{i+1/2,j}$ from $\{ v_j f_{i,j}^{(1)}\}$ using WENO reconstruction  in the framework of the \emph{conservative finite difference formulation} (fd) \cite{Shu98}.
		\item Conservative convection: $f^{**}_{i,j} = \sum_{k = 1}^{s} a_k f_{i,j}^{n+1-k} -\beta_s \frac{\Delta t}{\Delta x}\left(\widehat{F}^{(1)}_{i+1/2,j} - \widehat{F}^{(1)}_{i-1/2,j} \right)$.
		\item Compute conservative solution: $$f^{n+1}_{i,j} = f^{**}_{i,j} + \beta_s \frac{\Delta t}{\kappa}\left( {d\mathcal{M}}_{i,j}^{n+1} - f_{i,j}^{n+1}\right).$$
		Note that ${d\mathcal{M}}_{i,j}^{n+1} = {d\mathcal{M}}(f_{i,j}^{**})$ as in (\ref{dM}).
	\end{enumerate}
\end{itemize}
%\item  Compute
%		\begin{align}
%			f_{i,j}^{n+1}= {f^*}_{i,j}^{n+1} + \beta_{s} \frac{\Delta t}{\kappa} \left({dM^*}_{i,j}^{n+1} - f_{i,j}^{n+1} \right).
%		\end{align}
%	\end{enumerate}
\begin{figure}[t]
	\centering
	\includegraphics[width=0.5\linewidth]{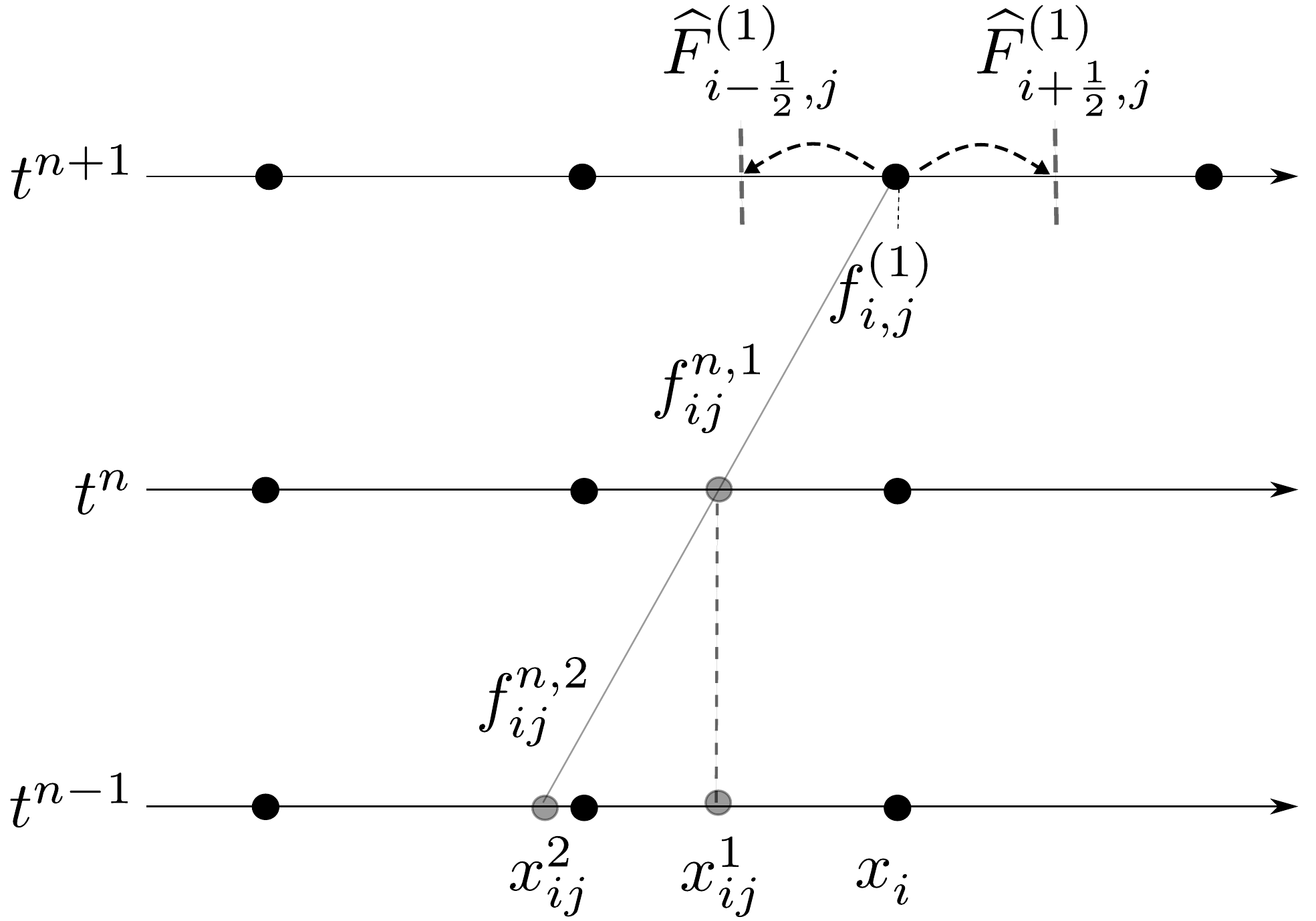}			\caption{Representation of BDF2 scheme with conservative correction. Black circles: grid nodes, grey circles: points where interpolation is needed.}	
	\label{Figbdf2}
\end{figure}
%
%\begin{figure}[t]
%	\centering
%	\includegraphics[width=0.5\linewidth]{Figures/BDF3}			\caption{Representation of BDF3 scheme with conservative correction. Black circles: grid nodes, grey circles: points where interpolation is needed.}	
% \label{Figbdf3}
%\end{figure}

Schematics of BDF2 is illustrated in Fig. \ref{Figbdf2}.

%As an example, In Fig. (\ref{Figbdf2}) and Fig. (\ref{Figbdf3})  we consider BDF scheme of order 2 (BDF2), and of order 3 (BDF3) so $s = 2, 3$. The coefficients in this two cases are given by formulas in (\ref{BDFschemes})

The conservative correction imposes some stability restrictions on the time step. The next section is devoted to the stability analysis of RK and BDF schemes applied to the linear advection equation. 

\section{Linear stability analysis}
%\blu{Boscarino: The first part has to be shortened since it is contained in th paper with Qiu }.

In this section we perform the stability analysis of conservative semi-Lagrangian
scheme for the 1D advection equation. Following  \cite{RQT} we consider the linear transport
equation
\begin{equation}\label{Equ}
u_t + v u_x = 0, \quad u(x,0) = u_0(x), \quad v \in \mathbb{R}.
\end{equation}
For simplicity, we assume a periodic boundary condition and $x \in [-\pi,\pi]$.

Algorithm \ref{A_DIRK} applied to (\ref{Equ}) gives
\begin{equation}\label{EquFq}
u_j^{n+1} = u_j^{n} -\frac{\Delta t}{\Delta x}\sum_{\ell = 1}^{s}b_{\ell} v \left( \widehat{u}^{(\ell)}_{j+1/2} - \widehat{u}^{(\ell)}_{j-1/2}\right),
\end{equation}
%Let us consider a conservative approximation of the equation (\ref{Equ}):
%\begin{equation}\label{EquInt}
%\frac{du_j}{dt} = -\frac{1}{\Delta x} \left( \widehat{f}_{j+1/2} - \widehat{f}_{j-1/2}\right),
%\end{equation}
%where, as usual, we denote $\widehat{f}_{j+1/2}$ as an approximation of $\widehat{f}(x_{j+1/2},t)$, and the hat denote the anti-average. i.e. for any space function
% $g(x)$, we donote by $\widehat{g}(x)$ a funciton such that its sliding average equals $g$,
% \[
% g(x) = \frac{1}{\Delta x}\int_{-\Delta x/2}^{\Delta x/2} \widehat{g}(x + \xi)d\xi.
% \]
% The conservative semilagrangian scheme is obtained by applying a quadrature rule to Eq. (\ref{EquInt}) and compute the stage values at the nodes of the quadrature rule by semilagrangian method. Then applying a quadrature formula with weights $b_1, ..., b_s$ and nodes $c_1, ..., c_s$ one obtains
% \begin{equation}\label{EquFq}
% u_j^{n+1} = u_j^{n} -\frac{\Delta t}{\Delta x}\sum_{\ell = 1}^{s}b_{\ell} \left( \widehat{f}^{(\ell)}_{j+1/2} - \widehat{f}^{(\ell)}_{j-1/2}\right),
% \end{equation}
% where the values $\widehat{f}^{(\ell)}_{j+1/2}$ can be computed, for example, by first computing $f^{(\ell)}_j = f(u^{(\ell)}_j)$ and then computing $\{ \widehat{f}^{(\ell)}_{j+1/2} \}$ from $f^{(\ell)}_j$ using some high order reconstruction from cell averages to point-wise values. In our case, the stage values are just
where $\widehat{u}^{(\ell)}_{j+1/2}$ are obtained from the stage values $u^{(\ell)}_j = u^n(x_j -v c_\ell \Delta t)$ by reconstruction, and $u^n(x)$ denotes a suitable interpolation from $\{u^{n}_j\}$. 

\subsection{Fourier interpolation}

We look for the evolution of a Fourier mode of the form
\[
u^n_j = \rho^n e^{ikj\Delta x} =\rho^n e^{ij\xi}, \quad \xi = k \Delta x, \quad i= 		\sqrt{-1}.
\]
In the analysis we first consider Fourier interpolation, so
%In this section our main goal is to investigate the linear stability due  to the approximation introduced by  using the
%quadrature formula in Eq. (\ref{EquFq}) in place of the exact integral in time of Eq. (\ref{EquInt}). In order to do that we
%assume that the interpolation and integration in space are performed exactly. Then we look for the evaluation of Fourier mode, identified by a Fourier variable
%$\xi \in [-\pi, \pi]$. Such a discrete Fourier mode at time $t^n = n \Delta t$ will be denoted by
\begin{equation}
u^n(x)=\rho^ne^{ikx} = \rho^ne^{i\xi x/\Delta x}, \quad \xi \in [-\pi, \pi], 
\label{eq:FI}
\end{equation}
where $\rho^n = \rho^n(\xi)$ is the amplification factor associated with $\xi$. Plugging such {\em ansatz} into the stage values, we get
\[
u^{(\ell)}_j =  \rho^n \exp{(i\xi(x_j -v\Delta t c_{\ell})}/\Delta x) = \rho^n e^{ij\xi}e^{-ic_\ell a \xi},
\]
where $a = v\Delta t/\Delta x$ denotes the CFL number.  
In \cite{Shu98}, the relation between $u(x)$ and $\hat{u}$ is given by
\[
\frac{\hat{u}(x+\Delta x/2)-\hat{u}(x-\Delta x/2)}{\Delta x} 
= \frac{\partial u}{\partial x}(x).
\]
Using this relation and (\ref{eq:FI}) one has
\begin{equation}\label{uhat}
\displaystyle\hat{u}^n(x) = \frac{u^n(x)}{\sinc(\xi/2)},
\end{equation}
where $\sinc(x) = \sin(x)/x$.

Making use (\ref{uhat}) and (\ref{eq:FI}) in Eq.(\ref{EquFq}), one obtains the following formula for the amplification factor:

% Because of the linearity, one can compute first
% \[
% F_j = \sum_{\ell = 1}^s b_{\ell}f_{j}^{(\ell)} = v \sum_{\ell = 1}^s b_{\ell}u_{j}^{(\ell)} = v \sum_{\ell = 1}^s b_{\ell} \exp{(-ic_\ell a \xi)}\rho^n \exp{(ij\xi)}
% \]
% where we let $F(x) := v\sum_{\ell = 1}^s u(x, t^n + c_{\ell}\Delta t) b_{\ell} \Delta t$ as a quadrature approximation of $\int_{t^n}^{t^{n+1}} v u(x,t) dt$ and
% $F_j = F(x_j)$.
%
% Make use of these results, one obtains an expression of the amplification factor:
\begin{equation}\label{rhoXi}
\rho(\xi) = 1 -i \xi a \sum_{\ell = 1}^s b_{\ell} \exp{(-i c_{\ell} a \xi)}.
\end{equation}
The scheme is stable if $|\rho(\xi)|\le 1$ for all $\xi \in [-\pi, \pi]$.

Such stability problem is closely related to the linear stability of the quadrature formula when applying to the approximation  of the
integral form of a scalar linear ODE,
\[
y' = \lambda y, \quad y(0) = 1, \quad \forall \lambda \in \mathbb{C}.
\]
In fact, the solution after one step of this ODE is: $y(\Delta t) = e^{\lambda \Delta t} = e^{z}$, where $z := \Delta t \lambda$. Such solution is stable iff
$\mathcal{R}(z) \le 0$, i.e. if $\mathcal{R}(\lambda) \le 0$. Considering the following identity
\[
e^{z} = 1 + z\int_{0}^1 e^{cz} dc
\]
and approximating the integral by a quadrature formula with nodes $c_{\ell}$ and weights $b_{\ell}$,
one obtains the approximation of the exact solution after one step:
\begin{equation}\label{StabR}
R(z) = 1 + z \sum_{\ell = 1}^s b_{\ell} e^{c_{\ell}z}
\end{equation}
with which the stability region can be drawn by the set
$\{z \in \mathbb{C} : |R(z)|\le 1\}$. Comparing equations (\ref{rhoXi}) with (\ref{StabR}), ones has $\rho(\xi) = R(-ia\xi)$ with $\xi \in [\pi, \pi]$. Thus the stability of a quadrature formula in a conservative semi-Lagrangian scheme for a linear advection equation is closely related to the stability on the imaginary axis.  Then in order to guarantee stability we look of the largest interval $I^* = [-y^*, y^*]$ of the imaginary axis such that $|R(iy)| \le 1$ $\forall y \in I^*$. Note that the bound $a^* = y^*/\pi$ quantifies the maximum CFL number for the semi-Lagrangian scheme that guarantees stability.

Now in order to maximize the stability interval on imaginary axis, we construct quadrature formulas that allow a wider stability region. Let us consider the expression $R(iy)$ and write it in the form
\begin{equation}\label{CsSs}
R(iy) = 1 + iy(C_s(y) + i S_s(y)) = 1 -yS_s(y) + iyC_s(y)
\end{equation}
where
\[
C_s(y) = \sum_{\ell = 1}^s b_{\ell}\cos(c_{\ell}y), \quad S_s(y) = \sum_{\ell = 1}^s b_{\ell}\sin(c_{\ell}y).
\]
The stability condition therefore bocomes
\[
|R(iy)|^2 = 1-2yS_y(y) + y^2(C^2_s(y) + S^2_s(y) )\le 1.
\]
Such condition can be written in the form
\begin{equation}\label{Fsy}
yF_s(y) \ge 0, \, \textrm{where} \quad F_s(y) := S_y(y) -\frac{1}{2}\left( C^2_s(y) + S^2_s(y)\right).
\end{equation}
Then the problem to find quadrature formulas with the widest stability region is connected to determine the coefficients ${\bf b} = (b_1, ..., b_s)$ and ${\bf c} = (c_1, ..., c_s)$ so that the interval in which (\ref{Fsy}) is satisfied is the widest. The analysis of quadrature formulas with even $s$ and symmetric distribution of nodes around the middle of the interval is performed in \cite{RQT}.

Here we numerically compute nodes and weights
for a particular class of third order DIRK schemes that satisfy the simplification conditions
\[
\sum_{j=1}^s a_{ij} = c_i,\quad i=1,\ldots,s
\]
and for which the last row of the $A$-matrix coincides with the weights, 
$a_{s,j} = b_j, \> j=1,\ldots,s$. This constraint is imposed in order to have $L$-stable schemes, in view of the AP property in the fluid dynamic regime.
Such schemes have the following structure:
%\begin{array}{ c | c  c  }
%\gamma_1 & \gamma_1 & 0  \\
%1 & a_{21} & \gamma_2  \\ \hline
%\end{array}
%\quad
\begin{equation}
\begin{array}{ c | c  c  c}
c_1 & c_1 & 0 & 0 \\
c_2 & c_2-\gamma_2 & \gamma_2 & 0 \\
1 & b_1 & b_2 & b_3 \\ \hline
& b_1 & b_2 & b_3 
\end{array}
\label{DIRK3g}
\end{equation}
The coefficients of the scheme are determined taking into account the following requirements:
\begin{itemize}
	\item the scheme has to be at least third order accurate;
	\item the scheme has to be $A$-stable (and therefore L-stable, because it is  Stiffly Accurate, (SA) i.e.  $a_{si} = b_i$ for $i = 1,2,3$, see \cite{HW});
	\item nodes and weight are selected in such a way that condition (\ref{Fsy}) is satisfied for a wide region.
\end{itemize}

Order conditions for scheme (\ref{DIRK3g}), up to third order accuracy, are:
\begin{equation}\label{Third_cond}
\sum_{i = 1}^s b_i = 1,\quad \sum_{i = 1}^s b_ic_i = 1/2, \quad \sum_{i = 1}^s b_ic^2_i = 1/3, \quad \sum_{i,j = 1}^s b_ia_{ij}c_j = 1/6.
\end{equation}

Solving these equations allows to express four parameters of the scheme as a function of $c_1$ and $c_2$:
\begin{equation}\label{Cond3}
b_2 = \frac{1}{6}\frac{3c_1-1}{(c_2-c_1)(c_2-1)}, \quad b_3 = \frac{1}{6}
\frac{6c_1c_2-3c_1-3c_2+2}{(c_2-1)(c_1-1)}, \quad
\gamma_2 = \frac12\frac{6c_1^2c_2-4c_1c_2-c_1+c_2}{(3c_1-1)(c_1-1)}
\end{equation}
and $b_1 = 1-b_2-b_3$.
This leaves two free parameters, which are chosen according to the two additional conditions. 

In order to impose $A$-stability,  from \cite{HW}, we recall the following result. 

An implicit R-K method is $A$-stable iff 
\begin{enumerate}
	\item the stability function $R(z) = P(z)/Q(z)$ is analytic in $\mathbb{C}$ for ${Re}(z)<0$;
	\item the method is $I$-stable, i.e. $|R(iy)|\le 1$ for all $y \in \mathbb{R}$ (stability on the imaginary axis).
\end{enumerate}
The $I$-stability is equivalent to the fact that the polynomial 
\begin{equation}\label{Ey}
E(y)= |Q(iy)|^2 - |P(iy)|^2 = \sum_{j = 0}^sE_{2j}y^{2j}
\end{equation} 
satisfies $E(y)\ge 0$ for all $y \in \mathbb{R}$ and $i = \sqrt{-1}$. 

Performing a detailed calculation (reported in Appendix \ref{StabilityCalculus}), the condition for $I$-stability (\ref{Ey}) and (\ref{Fsy}) becomes: either 
\begin{equation}
c_1<1/3, \> c_2>1
\label{eq:c1}
\end{equation}
or 
\begin{equation}
c_1>1/3, \> c_2<1.
\label{eq:c2} 
\end{equation}
The latter has to be excluded since it implies $b_3<0$ and condition (1) for the analyticity of the function $R(z)$ above is not satisfied. 

Then DIRK scheme (\ref{DIRK3g}) is $A$-stable and by the SA property, it's also $L$-stable.
%\end{proposition}

%%%%%%%%%%%%%%%%%5
%For any choice of the parameters $(\gamma_1,c_2)$, Fixed random nodes, we compute the weights by imposing the 
% and we look for nodes and weights such that $F_s(y)\ge 0$ and (\ref{q1q2}) are satisfied simultaneusly.  Such conditions are a balance between accuracy and stability.

%What we observed numerically is that condition (\ref{q1q2}) is satisfied if we choose $c_1\le1$ and $c_2\ge 1$.

%Then we get as a consequence of (\ref{q1q2}):
%\begin{proposition}
%For the scheme (\ref{DIRK3g}), if $\gamma_1\le1$ and $c_2\ge 1$, $E(y) \ge 0$ for all $y \in \mathbb{R}$ then the DIRK scheme (\ref{DIRK3g}) is $A$-stable then $L$-stable.
%\end{proposition}
%
%%%%%%%%%%%%%%%%%%%%%%%%%%%%%%%%%%%%
%%%%%%%%%%%%%%%%%%%%%%%%%%%%%%%%%%%%
% We random nodes numerically that fixing nodes such that  $\gamma_1 \le 1$ and $c_2\ge 1$ and compute weights from (\ref{Third_cond})
%.....

{\bf Remark} If we look for a third order Singly DIRK scheme i.e.\ with $\gamma_1 = \gamma_2 = \gamma_3 = \gamma$,
as in \cite{AUSR}, then there are no free parameters, and one obtains
%\begin{center}
%DIRK3 =
%\begin{tabular}{ c | c  c  c}
%	$\gamma$ & $\gamma$ & 0 & 0 \\
% $\frac{1 + \gamma}{2}$ & $\frac{1 - \gamma}{2}$ & $\gamma$ & 0 \\
%	1 & $1- \delta - \gamma$ & $\delta$ & $\gamma$ \\ \hline
%	& $1- \delta - \gamma$ & $\delta$ & $\gamma$ \\
%\end{tabular}
%\end{center}
$\gamma \simeq 0.4358665215$ and $\delta = \frac{3}{2}\gamma^2 - 5\gamma +\frac{5}{4} \simeq 644363171$. %This scheme is third order, i.e. it satisfies the order conditions (\ref{Third_cond}), and satisfy the condition (\ref{4.17}), 
This scheme is $A$-stable and $L$-stable,
but the weights and the nodes do not satisfy condition (\ref{Fsy}) for any $y >0$, i.e. scheme (\ref{EquFq}) is not stable.

This remark suggests to look for DIRK methods as (\ref{DIRK3g}) such that $\gamma_1 = \gamma_3 = \gamma$, i.e.
\begin{equation}\label{DIRK3bis}
\begin{array}{ c | c  c  c}
\gamma & \gamma & 0 & 0 \\
c_2 & c_2-\gamma_2 & \gamma_2 & 0 \\
1 & 1-b_2-\gamma & b_2 & \gamma \\ \hline
& 1-b_2 - \gamma & b_2 & \gamma
\end{array}
\end{equation}
%that define the quadrature formula for maximizing the  stability  interval  on imaginary  axis,  we proceed  as  follows.
From (\ref{Third_cond}), we have four equations with five unknowns $b_2$, $c_2$ $\gamma_2$, $\gamma$ and $b_1$ and from (\ref{Cond3}), with $c_1 = b_3 = \gamma$, we compute $b_2$, $c_2$ and $\gamma_2$ as functions of $\gamma$ and $b_1 = 1- b_2 -\gamma$.  

Performing a detailed calculation (reported in Appendix \ref{StabilityCalculus}),
we require to choose $\gamma$ in the following intervals  
%we impose that $\gamma_2$, computed by (\ref{Cond3}) with $c_1 = \gamma$, satisfies (\ref{g2}) as well. We obtain the following intervals:
\begin{equation}\label{intervals}
] 1-\sqrt{2}/2, 1/3[, \quad ] 1+\sqrt{2}/2, + \infty[.
\end{equation}
Note that the second interval can not be accepted because this implies values  of $\gamma$ such that $\gamma > 1 + \sqrt{2}/2 \approx 1.70710...$, and this is in contradiction with the hypothesis (\ref{eq:c1}).

%Now, we compute again weight and nodes such that (\ref{Fsy}) and (\ref{q1q2}) are satisfied simultaneously. In particular,
%for this scheme (\ref{q1q2}) becomes %$q_1 = (2\gamma + \gamma_2)$ and $q_2 = (2\gamma_2 \gamma + \gamma^2)$
%or
%\begin{equation}\label{g2}
%\gamma_2 \ge \frac{3-16\gamma + 12\gamma^2}{8-24\gamma}.
%\end{equation}
%with $q_1 = (2\gamma + \gamma_2)$ and $q_2 = (2\gamma_2 \gamma + \gamma^2)$.

A numerical experiment shows that the optimal value of $\gamma$  in the first interval in order to have stability, is approximately $\gamma = 0.3$, (see Fig. \ref{amp w.r.t gamma}). Then for this choice of $\gamma$, the coefficients of scheme (\ref{DIRK3bis}) are: $\gamma=0.3$, $\gamma_2={13}/{3}$, $b_2=-{3}/{710}$ and $c_2= {8}/{3}$.
This scheme is stable under the condition (\ref{Fsy}) for $y \le y^* = 4.715426442$ with $a^* \approx 1.5$ and is also $L$-stable.
\begin{figure}[h!]
	\centering
	%\begin{subfigure}[b]{0.46\linewidth}
	%	\includegraphics[width=1\linewidth]{Figures/optimalC2}
	%	\subcaption{$s=2$}
	%\end{subfigure}
	%\begin{subfigure}[b]{0.46\linewidth}
	\includegraphics[width=0.46\linewidth]{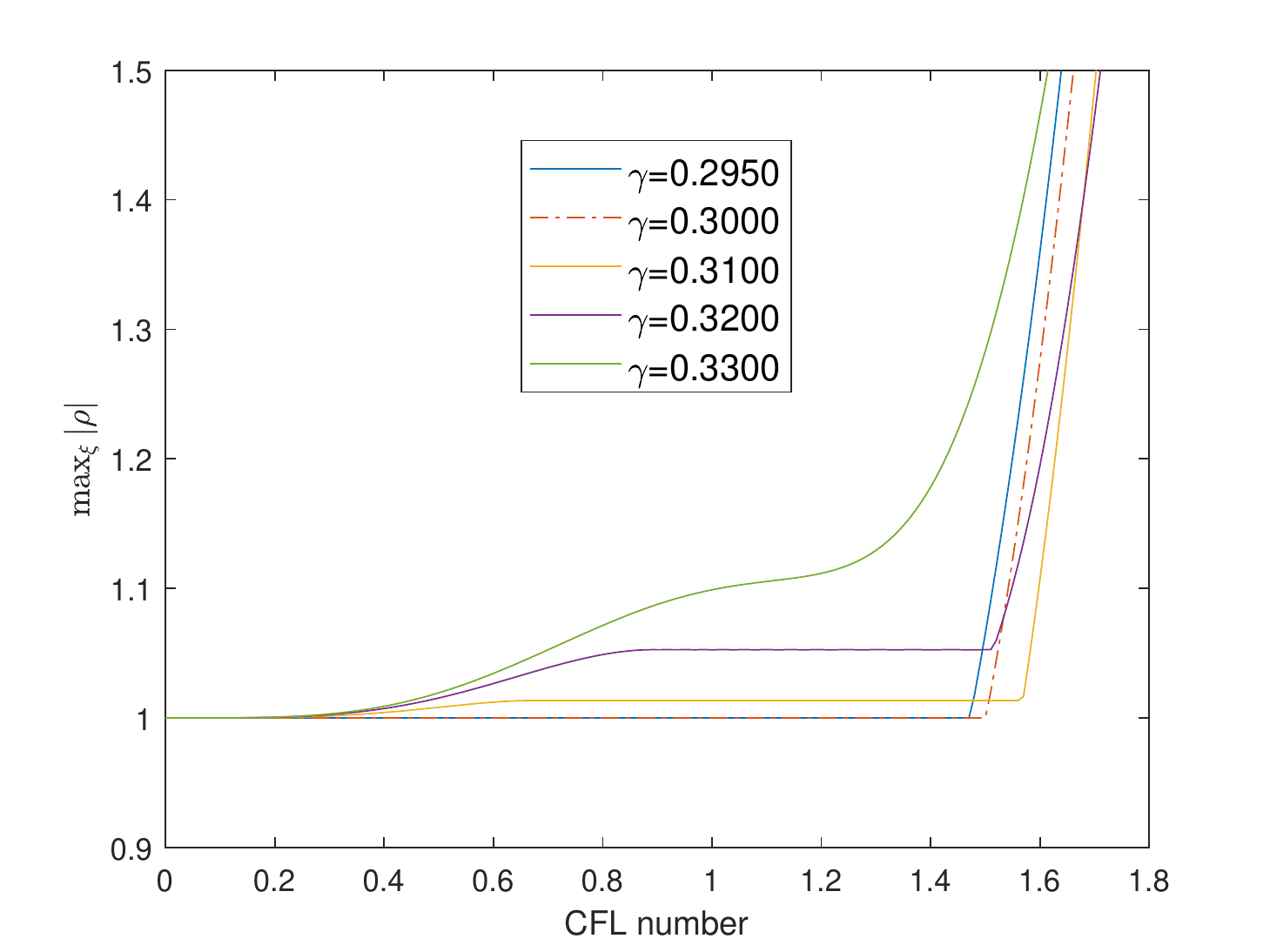}
	%	\subcaption{$s=3$}
	%\end{subfigure}
	\caption{Optimal choice of $\gamma$ for scheme (\ref{DIRK3bis}).
		%		\blu{Cho: For the figure in the right we should add some lines also for the case $\gamma \ge 0.3 $ for instance: $\gamma = 0.3, 0.32, 0.33, ...$}
	}
	\label{amp w.r.t gamma}
\end{figure}

%In this work, we use the following two types of DIRK methods and we call DIRK2 and N-DIRK3 respectively
%\begin{table}[h]\label{DIRK23}
%	\begin{center}
%		DIRK2 : \begin{tabular}{ c|cc}
%			$\alpha$ & $\alpha$ & 0  \\
%			$1$ & $1-\alpha$ & $\alpha$  \\
%			\hline
%			& $1-\alpha$ & $\alpha$  \\
%		\end{tabular},$\quad$
%		N-DIRK3 : \begin{tabular}{ c|ccc}
%			$\gamma$ & $\gamma$ & 0 & 0 \\
%			$c_2$ & $c_2 -\gamma_2$ & $\gamma_2$ & 0 \\
%			1 & $1-b_2 - \gamma$ & $b_2$ & $\gamma$ \\ \hline
%			& $1-b_2 - \gamma$ & $b_2$ & $\gamma$ \\
%		\end{tabular}
%	\end{center}
%\end{table}

%\noindent where $\alpha= 1-{\sqrt{2}}/{2}$ , $\gamma=0.3$, $\gamma_2={13}/{3}$, $b_2=-{3}/{710}$ and $c_2= {8}/{3}$.
%	The DIRK2 method is adopted from \cite{HWN} (see also \cite{GRS,RS,SP}), while N-DIRK3 method is newly introduced for this work. See Section 5 for details.

For the numerical experiments, we use the following two types of DIRK methods. The first is a second order DIRK scheme (DIRK2) \cite{AUSR}
\begin{center}
	DIRK2 =
	\begin{tabular}{ c | c  c }
		$\alpha$ & $\alpha$ & 0 \\
		1 & 1-$\alpha$ & $\alpha$ \\ \hline
		& 1-$\alpha$ & $\alpha$ \\
	\end{tabular}
\end{center}
where $\alpha = 1- \frac{\sqrt{2}}{2}$.% $\alpha$ is the middle root for $6x^3-18x^2+9x-1=0$.
This scheme is stable under the condition (\ref{Fsy}) for $y \le y^* = 4.586275880$ with $a^* \approx 1.46$ and is also $L$-stable.  %\cite{AUSR}.
The second one is the third order DIRK scheme (DIRK3)  (\ref{DIRK3bis}).

Now apply BDF schemes to system (\ref{Equ}) with $f = vu$, and we get
\begin{equation}\label{BDFadv}
u^{n+1}_{j} = \sum_{\ell = 1}^{k} {a}_{\ell} u^{n-\ell +1}_{j} - \beta_k v\frac{\Delta t}{\Delta x} \left( \widehat{u}^{n+1}_{j+1/2} - \widehat{u}^{n+1}_{j-1/2}\right),
\end{equation}
%with $\bar{a}_{\ell} = -a_{\ell}$
and by (\ref{uhat})
\begin{equation}\label{uhat2}
\widehat{u}^{n+1}_{j+1/2} =\frac{ \tilde{u}^{n+1}(x_{j + 1/2})}{\sinc(\xi/2)},
\end{equation}
with 
\begin{equation}\label{uhat3}
\tilde{u}^{n+1}(x_j) = \sum_{\ell = 1}^{k} {a}_\ell u^n(x_j - v\ell\Delta t).
\end{equation}

Then we look for the evolution of the Fourier mode identified by the parameter $\xi \in [-\pi, \pi]$. We set $u^(x) = \rho^n e^{ikx} = \rho^n e^{i\xi x/\Delta x}$ so that $u^n(x_j) = u^{n}_{j} = \rho^n e^{ij\xi}$. Then (\ref{uhat3}) becomes 
$$
\tilde{u}^{n+1}(x_j) = \sum_{\ell = 1}^k {a}_\ell\rho^{n - \ell + 1} e^{i \xi (x_j - \ell v\Delta t)/\Delta x},
$$ 
and (\ref{uhat2}) becomes 
$$
\hat{u}^{n+1}(x_{j + 1/2}) = \left(\sum_{\ell = 1}^k a_\ell\rho^{n-\ell +1}e^{ij\xi}e^{-i\xi a}e^{i\xi/2}\right)/\sinc(\xi/2),
$$
After some algebraic manipulation, we obtain for (\ref{BDFadv}):
\begin{equation}\label{BDFstab}
\rho^{n+1} = \sum_{\ell = 1}^{k}{a}_{\ell} \rho^{n-\ell +1}\left( 1- \beta_k a e^{-i\xi \ell a} i\xi \right)%2i \sin(\xi/2)\}
\end{equation}
where $a = v \Delta t/\Delta x$. Then the characteristic polynomial associated to (\ref{BDFstab}) is:
\begin{equation}
p(\rho) = \rho^{k} - \sum_{\ell = 1}^{k}{a}_{\ell} \rho^{k-\ell}\left( 1- \beta_k a e^{-i\xi \ell a}  i\xi \right).% \sin(\xi/2)\}.
\end{equation}
Now again we compute the maximum $a^*$ such that
\begin{equation}\label{astar}
\max_{\xi \in [-\pi, \pi]}|\rho(a,\xi)|\le 1 ,\quad  \forall a \in [0,a^*].
\end{equation}
Here $\rho(a,\xi)$ represents the largest root in absolute value of the polynomial $p(\rho)$. In particular, we consider the two BDF schemes BDF2 and BDF3 with $k = 2$ and $k = 3$, respectively. We compute numerically (\ref{astar}) and we get for BDF2  $a^* \approx 0.5678$, while BDF3  is unstable for each $a > 0$. %$a^* \approx 0.32$ for BDF3 respectively.

This analysis confirms that the conservative correction imposes stability restriction on the CFL number $a^*$ for the BDF methods.

We conclude that C-SL schemes based on RK framework have better stability properties that those based on BDF when applied to linear advection equation (\ref{Equ}).

Note that in our numerical tests in practise some schemes can be more stable when applied to BGK equation because the collision term has a  stabilizing effect.

\begin{figure}[t]
	\centering
	\begin{subfigure}[b]{0.48\linewidth}
		\includegraphics[width=1\linewidth]{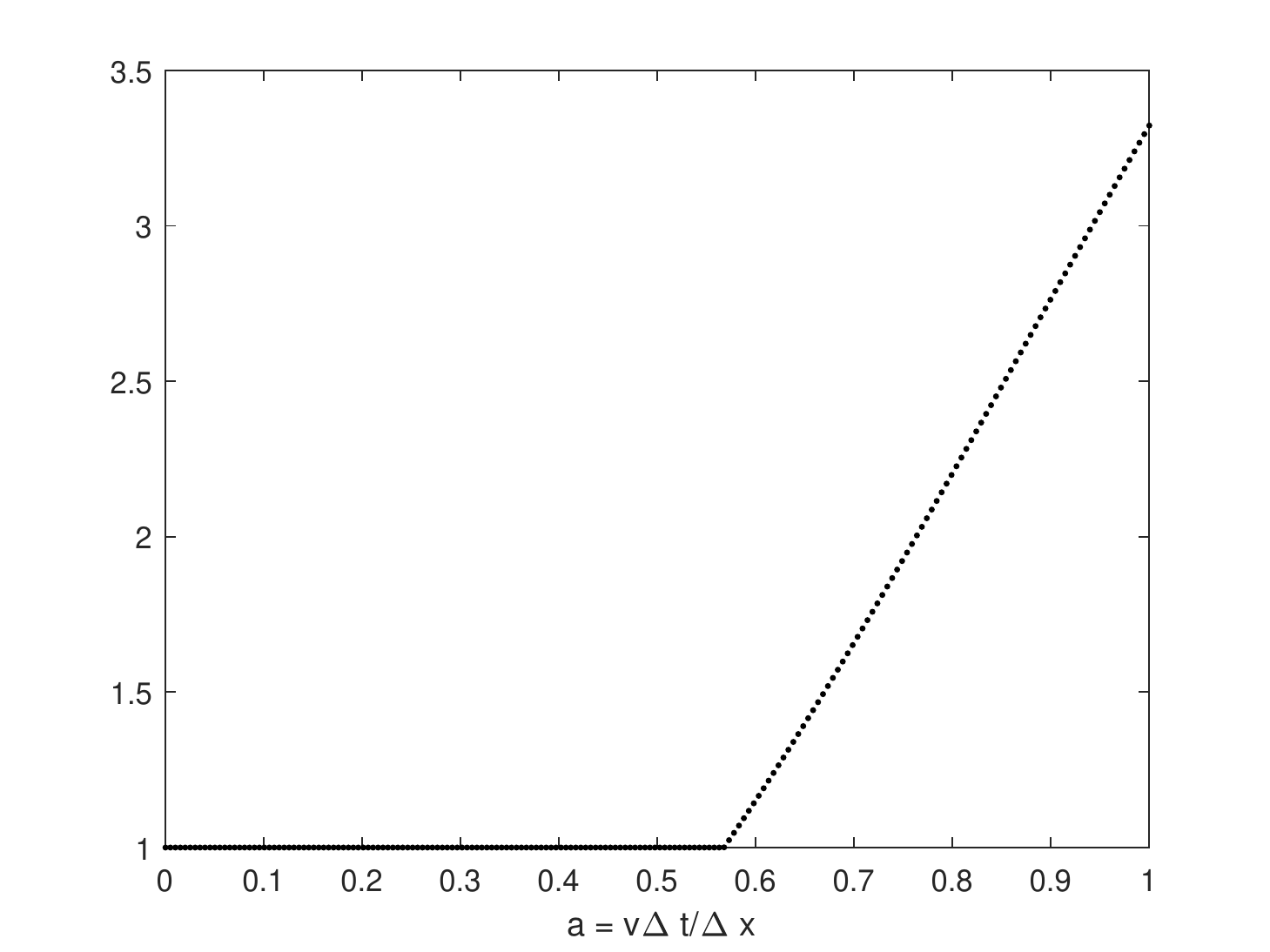}
		\subcaption{$k=2$}
	\end{subfigure}
	%\begin{subfigure}[b]{0.48\linewidth}
	%	\includegraphics[width=1\linewidth]{Figures/BDF3stab}
	%	\subcaption{$k=3$}
	%\end{subfigure}
	\caption{Optimal $a^*$ for BDF2}\label{amp BDF}% (left pannel)      and BDF3 (right pannel)}\label{amp BDF}
\end{figure}

\section{Numerical experiments}
In this section, we propose four tests to verify some properties of proposed schemes.

In test 1, we compute the conservation error of the schemes. In test 2, we check the correct order of accuracy for smooth solutions for various values of the Knudsen number $\kappa$. In test 2, we check the Asymptotic Preserving(AP) property. We end this section with test 3, in which we check shock capturing capability for the Euler limit. We only consider one-dimensional problems. For the time step, we use $\Delta{t}=CFL\times\Delta{x} /|v_{max}|$. For space and velocity grids, we discretize $\Delta{v}:=(v_{max}-v_{min})/N_v$ and $\Delta{x}:=(x_{max}-x_{min})/N_x$. To distinguish proposed conservative schemes from non-conservative schemes, we denote each scheme as follows:

%\begin{table}
\begin{center}
	\begin{tabular}{|l||c|c|c|l|} \hline
		Scheme name & Conservative & ODE solver & Reconstruction & Maxwellian \\ \hline
		RK2-W23-DM  & YES          & DIRK2      & WENO 2-3       & Discrete   \\
		RK3-W35-DM  & YES          & DIRK3      & WENO 3-5       & Discrete   \\
		RK2-W23     & NO           & DIRK2      & WENO 2-3       & Continuous \\
		RK3-W35     & NO           & DIRK3      & WENO 3-5       & Continuous \\ \hline
	\end{tabular}
\end{center}

A similar notation is used for the schemes based on BDF time integrator.

%\blu{In the table below: use scientific notation of the form x.xxe-p, with 3 digits. Use only one field, say density, and explain that a similar behavior is observed for the other fields (done)}

\subsection{Test 1}
\label{test1}
We consider the same single shock test adopted in Section \ref{sec:test1}, and apply the various schemes based on the conservative correction. The results are summarized in Tables 
\ref{tab3} and \ref{tab4}. 

%\begin{minipage}{\linewidth}

\begin{table}[h]
	\centering	
	{\begin{tabular}{|cccc|ccc|}
			\hline
			\multicolumn{1}{ |c| }{} & \multicolumn{3}{ c| }{RK3-W35-CM} & \multicolumn{3}{ c| }{RK3-W35-DM} \\ \hline
			\multicolumn{1}{ |c| }{$(N_x,N_v)$}& \multicolumn{1}{ c|  }{Mass} & \multicolumn{1}{ c| }{Momentum}& \multicolumn{1}{ c| }{Energy} & \multicolumn{1}{ c|  }{Mass} & \multicolumn{1}{ c| }{Momentum}& \multicolumn{1}{ c| }{Energy}\\ \hline	
			\multicolumn{1}{ |c|  }{$(100,40)$}& 3.74e-07 & 1.51e-05
			& 4.07e-06 & 5.47e-15 & 7.52e-14
			& 8.59e-14 \\
			\multicolumn{1}{ |c|  }{$(100,50)$}& 5.80e-12 & 3.58e-10
			& 9.66e-11 & 5.55e-14 & 4.04e-13
			& 5.22e-13 \\
			\multicolumn{1}{ |c|  }{$(100,60)$}& 1.70e-13 & 1.45e-13
			& 3.25e-13 & 1.92e-13 & 1.45e-13
			& 3.43e-13 \\
			\multicolumn{1}{ |c|  }{$(100,90)$}& 1.47e-13 & 8.85e-14
			& 1.85e-13 & 1.30e-13 & 8.51e-14
			& 2.27e-13 \\ \hline
			\multicolumn{1}{ |c|  }{$(200,40)$}& 7.70e-07 & 4.37e-06 & 8.38e-06  & 1.28e-14 & 2.39e-15 & 7.45e-14\\
			\multicolumn{1}{ |c|  }{$(200,50)$}& 1.21e-11 & 1.03e-10 & 1.99e-10  & 2.13e-13 & 3.04e-13 & 2.59e-13\\		
			\multicolumn{1}{ |c|  }{$(200,60)$}& 3.63e-13 & 1.74e-14 & 1.28e-13  & 3.73e-13 & 8.07e-14 & 1.79e-13\\
			\multicolumn{1}{ |c|  }{$(200,90)$}& 2.85e-13 & 1.47e-13 & 1.51e-13  & 2.83e-13 & 1.62e-13 & 1.75e-13\\
			\hline
	\end{tabular}}
	\caption{RK-based schemes, $CFL=2$. Conservation error of discrete moments in relative $L_1$ norm for Test 1. Comparison between continuous one (CM) and discrete Maxwellian (DM)}
	\label{tab3}
\end{table}

\begin{table}[h]
	\centering
	{\begin{tabular}{|cccc|ccc|}
			\hline
			\multicolumn{1}{ |c| }{} & \multicolumn{3}{ c| }{BDF3-W35-CM} & \multicolumn{3}{ c| }{BDF3-W35-DM} \\ \hline
			\multicolumn{1}{ |c| }{$(N_x,N_v)$}& \multicolumn{1}{ c|  }{Mass} & \multicolumn{1}{ c| }{Momentum}& \multicolumn{1}{ c| }{Energy} & \multicolumn{1}{ c|  }{Mass} & \multicolumn{1}{ c| }{Momentum}& \multicolumn{1}{ c| }{Energy}\\ \hline	
			\multicolumn{1}{ |c|  }{$(100,40)$}& 2.13e-07 & 6.30e-07
			& 2.31e-06 & 3.55e-15 & 6.50e-15
			& 3.40e-15 \\
			\multicolumn{1}{ |c|  }{$(100,50)$}& 3.33e-12 & 1.50e-11
			& 5.47e-11 & 6.08e-14 & 6.09e-14
			& 3.43e-14 \\
			\multicolumn{1}{ |c|  }{$(100,60)$}& 1.05e-13 & 1.95e-14
			& 1.61e-14 & 1.06e-13 & 3.42e-15
			& 6.49e-15 \\
			\multicolumn{1}{ |c|  }{$(100,90)$}& 8.12e-14 & 1.16e-14
			& 1.17e-14 & 8.39e-14 & 3.25e-14
			& 1.54e-14 \\ \hline
			\multicolumn{1}{ |c|  }{$(200,40)$}& 4.31e-07 & 2.19e-06
			& 4.68e-06  & 1.49e-14 & 3.08e-14 & 3.80e-14\\
			\multicolumn{1}{ |c|  }{$(200,50)$}& 6.78e-12 & 5.22e-11 & 1.11e-10 & 1.60e-13 & 8.65e-14 & 1.51e-14\\		
			\multicolumn{1}{ |c|  }{$(200,60)$}& 1.80e-13 & 4.75e-14 & 3.68e-14  & 1.82e-13 & 5.61e-14 & 4.48e-14\\
			\multicolumn{1}{ |c|  }{$(200,90)$}& 1.34e-13 & 8.03e-14 & 3.77e-14  & 1.36e-13 & 8.99e-14 & 4.42e-14\\
			\hline
	\end{tabular}}
	
	\caption{BDF-based schemes, $CFL=2$. Conservation error of discrete moments in relative $L_1$ norm for Test 1. Comparison between continuous one (CM) and discrete Maxwellian (DM)}
	\label{tab4}
\end{table}

%\end{minipage}

When using with the continuous Maxwellian with the conservative correction,  a negligible conservation error can be achieved, but only using a large enough number of velocity grid points. In contrast, conservation error can be suppressed to a negligible level with relatively small number of velocity grid points when the discrete Maxwellian is used (See Table \ref{tab3} and \ref{tab4} with $N_v = 40$). We present the conservation error estimates in Appendix \ref{Conservation error Estimates}.

It appears that the combined use of discrete Maxwellian and conservative correction provides a scheme which maintains conservation within round-off error.
In particular, the use of discrete Maxwellian allows to maintain conservation with a small number of velocity nodes, which is particularly useful when adopting the method to capture the fluid dynamic limit for small Knudsen number.

\subsection{Test 2}
%\blu{CHO: put the reference where we take each numerical test, for example Test 1 I guess from Russo-Groppi-Straq ...}
%\blu{CHO: we just discussed of this Test 1, you try to increase the final time to T=1 or less. : done}

This test is proposed in \cite{GRS} to check the accuracy of the scheme.
The initial condition for the distribution function is the Maxwellian 
$$
f_0(x,v)=\frac{\rho_0}{\sqrt{2 \pi T_0}}\exp\left(-\frac{|v-u_0(x)|^2}{2T_0}\right),
$$ 
where initial velocity profile is given by
\[
u_0(x) = 0.1 \exp\left(-(10x - 1)^2\right) - 2 \exp\left(-(10x + 3)^2\right).
\]
Initial density and temperature are uniform, with constant value $\rho_0(x) = 1 $ and $T_0(x) = 1$. We use the periodic boudndary condition. The computation is performed on $ (x, v) \in [-1,1] \times [-10,10]$. Since shock appears for $\kappa=10^{-6}$ at $t = 0.35$, the final time is taken $T_f=0.32$.  We take $N_x = 160, 320, 640, 1280, 2560$, and $5120$ and uniform grid points in $x $ direction and
$N_v = 20$ uniform grid points in $v$ direction. We have used differnt CFL based on the stability analysis. For RK schemes, we choose CFL $= 2$. For BDF2 and BDF3, we set CFL$=0.5$. We compute the relative $L^1$ norm to check the accuracy.

\begin{table}[h]
	\begin{center}
		
		{\begin{tabular}{|ccccccccc|}
				\hline
				\multicolumn{9}{ |c| }{Test 2 RK2-W23-DM Mass, CFL=2} \\ \hline
				\multicolumn{1}{ |c }{}& \multicolumn{2}{ |c  }{$\kappa=10^{-6}$} & \multicolumn{2}{ |c }{$\kappa=10^{-4}$}& \multicolumn{2}{ |c| }{$\kappa=10^{-2}$} &
				\multicolumn{2}{ |c| }{$\kappa=10^{-0}$} \\ \hline
				\multicolumn{1}{ |c  }{$N_x$} &
				\multicolumn{1}{ |c  }{error} &
				\multicolumn{1}{ c|  }{rate} &
				\multicolumn{1}{ |c  }{error} &
				\multicolumn{1}{ c|  }{rate} &
				\multicolumn{1}{ |c  }{error} &
				\multicolumn{1}{ c|  }{rate} &
				\multicolumn{1}{ |c  }{error} &
				\multicolumn{1}{ c|  }{rate}     \\ \hline	
				\multicolumn{1}{ |c|  }{160-320}& 1.01e-03
				&1.74
				&9.80e-04
				&1.79
				&1.79e-04&2.17
				&7.13e-04
				&1.69
				
				\\
				\multicolumn{1}{ |c|  }{320-640}& 3.02e-04
				&1.78
				&
				2.84e-04&1.86
				&
				3.97e-05&2.06
				&2.21e-04
				&1.84
				\\
				\multicolumn{1}{ |c|  }{640-1280}& 8.83e-05
				&2.17
				&
				7.82e-05&2.29
				&
				9.54e-06&2.01
				&6.19e-05
				&2.12
				\\
				\multicolumn{1}{ |c|  }{1280-2560}& 
				1.96e-05
				&2.38
				&
				1.60e-05&2.37
				&
				2.37e-06&1.99
				&1.42e-05
				&2.53
				\\
				\multicolumn{1}{ |c|  }{2560-5120}& 
				3.76e-06
				&&
				3.09e-06&&
				5.95e-07&&2.47e-06
				& 
				\\
				\hline
		\end{tabular}}\\
		
		{\begin{tabular}{|ccccccccc|}
				\hline
				\multicolumn{9}{ |c| }{Test 2 BDF2-W23-DM Mass, CFL=0.5} \\ \hline
				\multicolumn{1}{ |c }{}& \multicolumn{2}{ |c  }{$\kappa=10^{-6}$} & \multicolumn{2}{ |c }{$\kappa=10^{-4}$}& \multicolumn{2}{ |c| }{$\kappa=10^{-2}$} &
				\multicolumn{2}{ |c| }{$\kappa=10^{-0}$} \\ \hline
				\multicolumn{1}{ |c  }{$N_x$} &
				\multicolumn{1}{ |c  }{error} &
				\multicolumn{1}{ c|  }{rate} &
				\multicolumn{1}{ |c  }{error} &
				\multicolumn{1}{ c|  }{rate} &
				\multicolumn{1}{ |c  }{error} &
				\multicolumn{1}{ c|  }{rate} &
				\multicolumn{1}{ |c  }{error} &
				\multicolumn{1}{ c|  }{rate}     \\ \hline	
				\multicolumn{1}{ |c|  }{160-320}& 
				1.01e-03
				&1.74
				&9.80e-04
				&1.78
				&1.77e-04
				&2.17
				&7.09e-04
				&1.68
				\\
				\multicolumn{1}{ |c|  }{320-640}& 
				3.03e-04
				&1.77
				&2.85e-04
				&1.86
				&3.93e-05
				&2.05
				&2.21e-04
				&1.84
				\\
				\multicolumn{1}{ |c|  }{640-1280}& 
				8.86e-05
				&2.18
				&7.84e-05
				&2.29
				&9.46e-06
				&2.01
				&6.19e-05
				&2.13
				\\
				\multicolumn{1}{ |c|  }{1280-2560}&
				1.96e-05&2.38&1.60e-05
				&2.37&2.35e-06
				&1.99&1.42e-05
				&2.53
				\\
				\multicolumn{1}{ |c|  }{2560-5120}&
				3.75e-06&&3.09e-06&&5.93e-07&&2.45e-06& 
				\\
				\hline
		\end{tabular}}\\
		
		{\begin{tabular}{|ccccccccc|}
				\hline
				\multicolumn{9}{ |c| }{Test 2 RK3-W35-DM Mass, CFL=2} \\ \hline
				\multicolumn{1}{ |c }{}& \multicolumn{2}{ |c  }{$\kappa=10^{-6}$} & \multicolumn{2}{ |c }{$\kappa=10^{-4}$}& \multicolumn{2}{ |c| }{$\kappa=10^{-2}$} &
				\multicolumn{2}{ |c| }{$\kappa=10^{-0}$} \\ \hline
				\multicolumn{1}{ |c  }{$N_x$} &
				\multicolumn{1}{ |c  }{error} &
				\multicolumn{1}{ c|  }{rate} &
				\multicolumn{1}{ |c  }{error} &
				\multicolumn{1}{ c|  }{rate} &
				\multicolumn{1}{ |c  }{error} &
				\multicolumn{1}{ c|  }{rate} &
				\multicolumn{1}{ |c  }{error} &
				\multicolumn{1}{ c|  }{rate}     \\ \hline	
				\multicolumn{1}{ |c|  }{160-320}& 5.74e-05&3.31
				&5.07e-05
				&3.47
				&2.28e-06
				&4.34
				&1.28e-05
				&4.74
				\\
				\multicolumn{1}{ |c|  }{320-640}& 
				5.77e-06&4.23
				&4.58e-06
				&4.39
				&1.12e-07
				&3.59
				&4.80e-07
				&4.88
				\\
				\multicolumn{1}{ |c|  }{640-1280}& 
				3.08e-07&4.61
				&2.19e-07
				&4.43
				&9.31e-09
				&3.09
				&1.63e-08
				&4.66
				\\
				\multicolumn{1}{ |c|  }{1280-2560}& 
				1.26e-08&4.28
				&1.02e-08
				&3.58&1.09e-09
				&2.98
				&6.43e-10
				&4.06
				\\
				\multicolumn{1}{ |c|  }{2560-5120}& 
				6.49e-10&&8.50e-10&&1.38e-10&&3.84e-11& 
				\\
				\hline
		\end{tabular}}\\
		
		{\begin{tabular}{|ccccccccc|}
				\hline
				\multicolumn{9}{ |c| }{Test 2 BDF3-W35-DM Mass, CFL=0.5} \\ \hline
				\multicolumn{1}{ |c }{}& \multicolumn{2}{ |c  }{$\kappa=10^{-6}$} & \multicolumn{2}{ |c }{$\kappa=10^{-4}$}& \multicolumn{2}{ |c| }{$\kappa=10^{-2}$} &
				\multicolumn{2}{ |c| }{$\kappa=10^{-0}$} \\ \hline
				\multicolumn{1}{ |c  }{$N_x$} &
				\multicolumn{1}{ |c  }{error} &
				\multicolumn{1}{ c|  }{rate} &
				\multicolumn{1}{ |c  }{error} &
				\multicolumn{1}{ c|  }{rate} &
				\multicolumn{1}{ |c  }{error} &
				\multicolumn{1}{ c|  }{rate} &
				\multicolumn{1}{ |c  }{error} &
				\multicolumn{1}{ c|  }{rate}     \\ \hline	
				\multicolumn{1}{ |c|  }{160-320}& 
				5.59e-05
				&3.30
				&4.93e-05
				&3.47
				&1.93e-06
				&4.99
				&1.26e-05
				&4.81
				\\
				\multicolumn{1}{ |c|  }{320-640}& 
				5.69e-06
				&4.28
				&4.44e-06
				&4.47
				&6.07e-08
				&5.31
				&4.47e-07
				&5.18
				\\
				\multicolumn{1}{ |c|  }{640-1280}& 
				2.93e-07
				&4.77
				&2.01e-07
				&4.90
				&1.53e-09
				&4.61
				&1.24e-08
				&5.00
				\\
				\multicolumn{1}{ |c|  }{1280-2560}&
				1.07e-08
				&4.96&6.72e-09
				&4.98&6.27e-11
				&2.38&3.85e-10
				&2.95
				\\
				\multicolumn{1}{ |c|  }{2560-5120}&
				3.45e-10&&2.13e-10&&1.20e-11&&5.00e-11& 
				\\
				\hline
		\end{tabular}}
	\end{center}
	\label{tab:test2}
	\caption{Test 2: convergence rate for second and third order RK and BDF schemes. A final time $T_f = 0.32$ is selected such that the solution is still smooth even in the limit of vanishing Knudsen number. For such small time, space error appears to be dominant, and this explains the order of accuracy higher than expected from the order of the RK or BDF schemes. Some order reduction is observed in intermediate regimes.}
\end{table}

%	{\begin{tabular}{|ccccccccc|}
%			\hline
%			\multicolumn{9}{ |c| }{Test 2 Mass, $\kappa=10^{-6}$} \\ \hline
%			\multicolumn{1}{ |c }{}& \multicolumn{2}{ |c  }{RK2-W23-DM} & \multicolumn{2}{ |c }{BDF2-W23-DM}& \multicolumn{2}{ |c| }{RK3-W35-DM} &
%			\multicolumn{2}{ |c| }{BDF3-W35-DM} \\ \hline
%			\multicolumn{1}{ |c }{}& \multicolumn{2}{ |c  }{$CFL=6$} & \multicolumn{2}{ |c }{$CFL=2$}& \multicolumn{2}{ |c| }{$CFL=8$} &
%			\multicolumn{2}{ |c| }{$CFL=2$} \\ \hline		
%			\multicolumn{1}{ |c  }{$N_x$} &
%			\multicolumn{1}{ |c  }{error} &
%			\multicolumn{1}{ c|  }{rate} &
%			\multicolumn{1}{ |c  }{error} &
%			\multicolumn{1}{ c|  }{rate} &
%			\multicolumn{1}{ |c  }{error} &
%			\multicolumn{1}{ c|  }{rate} &
%			\multicolumn{1}{ |c  }{error} &
%			\multicolumn{1}{ c|  }{rate}     \\ \hline	
%			\multicolumn{1}{ |c|  }{160-320}&1.00e-03
%			&1.78
%			&9.76e-04
%			&1.82
%			&1.47e-04&2.77&7.22e-05&3.49       
%			\\
%			\multicolumn{1}{ |c|  }{320-640}& 
%			2.93e-04
%			&1.83
%			&2.76e-04
%			&1.88&2.15e-05&3.26&6.44e-06&3.38      
%			\\
%			\multicolumn{1}{ |c|  }{640-1280}&    
%			8.22e-05
%			&2.19
%			&7.52e-05&2.22&2.25e-06&3.42&6.18e-07&2.91         
%			\\
%			\multicolumn{1}{ |c|  }{1280-2560}&   
%			1.80e-05
%			&2.40&1.61e-05&2.44&2.10e-07&3.15&8.23e-08&2.95         
%			\\
%			\multicolumn{1}{ |c|  }{2560-5120}&   
%			3.41e-06&&2.97e-06
%			&&2.37e-08&&1.06e-08&    
%			\\
%			\hline
%	\end{tabular}}
%\end{center}
%

\subsection{Test 3}

%\blu{CHO: put here the reference where we take this test.(done)}

%\blu{CHO: we should explain the meaning of these acronym at the beginning of this section: RK3fdDM+W35 and BDF3fdDM+W35}
To check the AP property of the C-SL scheme, we take a similar test as in \cite{BS}. We check numerically $\|f-M\|_1 = \mathcal{O}(\kappa)$ for different values of $\kappa = 10^{-4},10^{-5},10^{-6},10^{-7}$.

We take the following non-equilibrium initial data
$$f_0(x,v)=0.5 \left(\frac{\rho_0(x)}{\sqrt{2\pi T_0(x)}}\exp \left(-\frac{(v-u_0(x))^2}{2T_0(x)}\right) + \frac{\rho_0(x)}{\sqrt{2\pi T_0(x)}}\exp \left(-\frac{(v+u_0(x))^2}{2T_0(x)}\right)\right),
$$
where initial density, velocity and temperature are given by
$$\rho_0(x) = \frac{2+\sin2\pi x}{3}, \quad
u_0(x)=\frac{\cos2\pi x}{5},\quad
T_0(x) = \frac{3+\cos2\pi x}{4}. $$

We use the periodic boundary condition. The computation is performed
on $  x \in [-1, 1]$, $ v \in [-8, 8]$. The final time is taken 0.02.

We implemented RK3-W35-DM and BDF3-W35-DM with $N_x = 100$ and CFL$=1$. In Figures \ref{Ap fig1}--\ref{Ap fig4}, we show the time evolution of $\|f-\mathcal{M}\|_1$ for our C-SL scheme for different values of $\kappa$ and different values for the number of grid points in velocity space, i.e., $N_v = 20, 32$.

From the figures it appears that the norm of the difference between $f$ and  the Maxwellian is roughly proportional to the Knudsen number $\kappa$, es expected. If a continuous Maxwellian, such a norm depends also on the number of velocity grid points, as appears in see Fig.~\ref{Ap fig1} (A) and in Fig.~\ref{Ap fig3} (A), where with $N_v=20$ the difference does not decrease significantly when going from $\kappa=10^{-7}$ to $\kappa = 10^{-8}$. On the contrary, when using a discrete Mawellian, the discrepancy between $f$ and the Maxwellian only depends on the Knudsen number: $\|f-\mathcal{M} \|_1 = \mathcal{O}(\kappa)$.

%We conclude that the C-SL scheme coupled with the Discrete Maxwellian shows exactly the so called \emph{strong} AP property introduced in \cite{Jin2}.

The proposed C-SL scheme (\ref{Conservative Stable DM}) is an asymptotic preserving (AP) scheme for the kinetic equation (\ref{bgk}), that is, it becomes a consistent scheme for the underlying hydrodynamic limit. %The proof of the AP property of the scheme (\ref{Conservative Stable DM}) in presented in the Appendix \ref{App1}.
Note that in a recent review  on AP schemes for kinetic and
hyperbolic equations \cite{Jin2}, a necessary condition to be AP for a scheme for BGK model (\ref{bgk}) is that the solution $f^n$ must be driven to the local equilibrium $\mathcal{M}^{n}$ when $\kappa \to 0$
\begin{equation}
f^n - \mathcal{M}(f^n) = \mathcal{O}(\kappa), \quad \textrm{for} \quad n \ge 1
\end{equation}
for any initial data $f^0$, namely, the numerical solution projects any data into the local
equilibrium $\mathcal{M}^{n}$, with an accuracy of $\mathcal{O}(\kappa)$, in one step.
%In general this can usually be achieved by a
%backward Euler or any imlicit L-stable ODE solvers for the collision term [61].
Such AP schemes are referred to as \emph{strongly} AP.

\begin{figure}[]
	\centering
	%	\begin{subfigure}[b]{0.45\linewidth}
	%		\includegraphics[width=1\linewidth]{Figures/RK3fdW35%AptestNv20CMCM}	
	%		\subcaption{RK3+W35+CM, $\|f-CM \|_1$}
	%	\end{subfigure}	
	%	\begin{subfigure}[b]{0.45\linewidth}
	%		\includegraphics[width=1\linewidth]{Figures/RK3fdW35%AptestNv20CMDM}	
	%		\subcaption{RK3+W35+CM, $\|f-DM \|_1$}
	%\end{subfigure}	
	\begin{subfigure}[b]{0.45\linewidth}
		\includegraphics[width=1\linewidth]{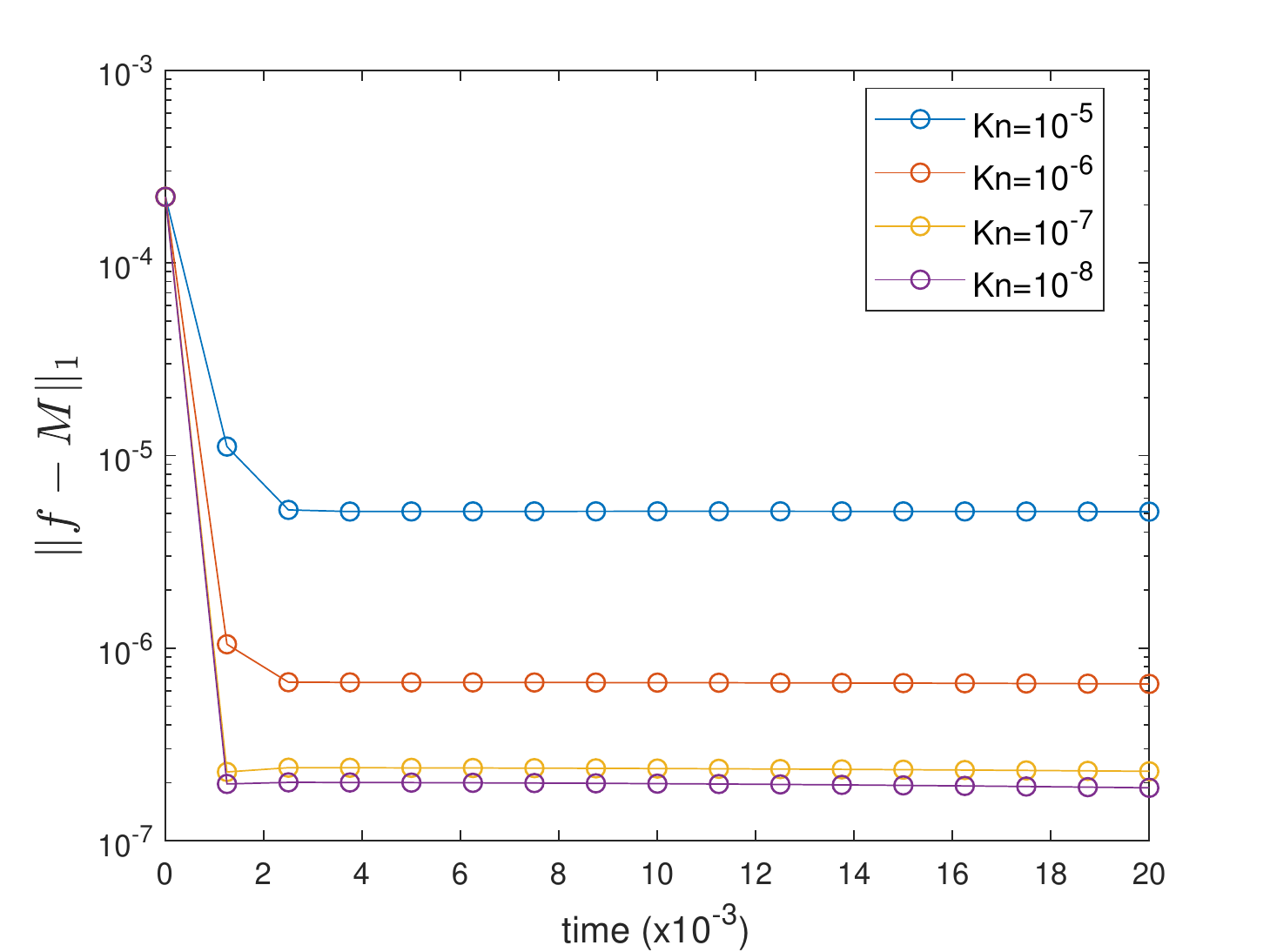}	
		\subcaption{RK3-W35-CM, $\|f-CM \|_1$}
	\end{subfigure}	
	\begin{subfigure}[b]{0.45\linewidth}
		\includegraphics[width=1\linewidth]{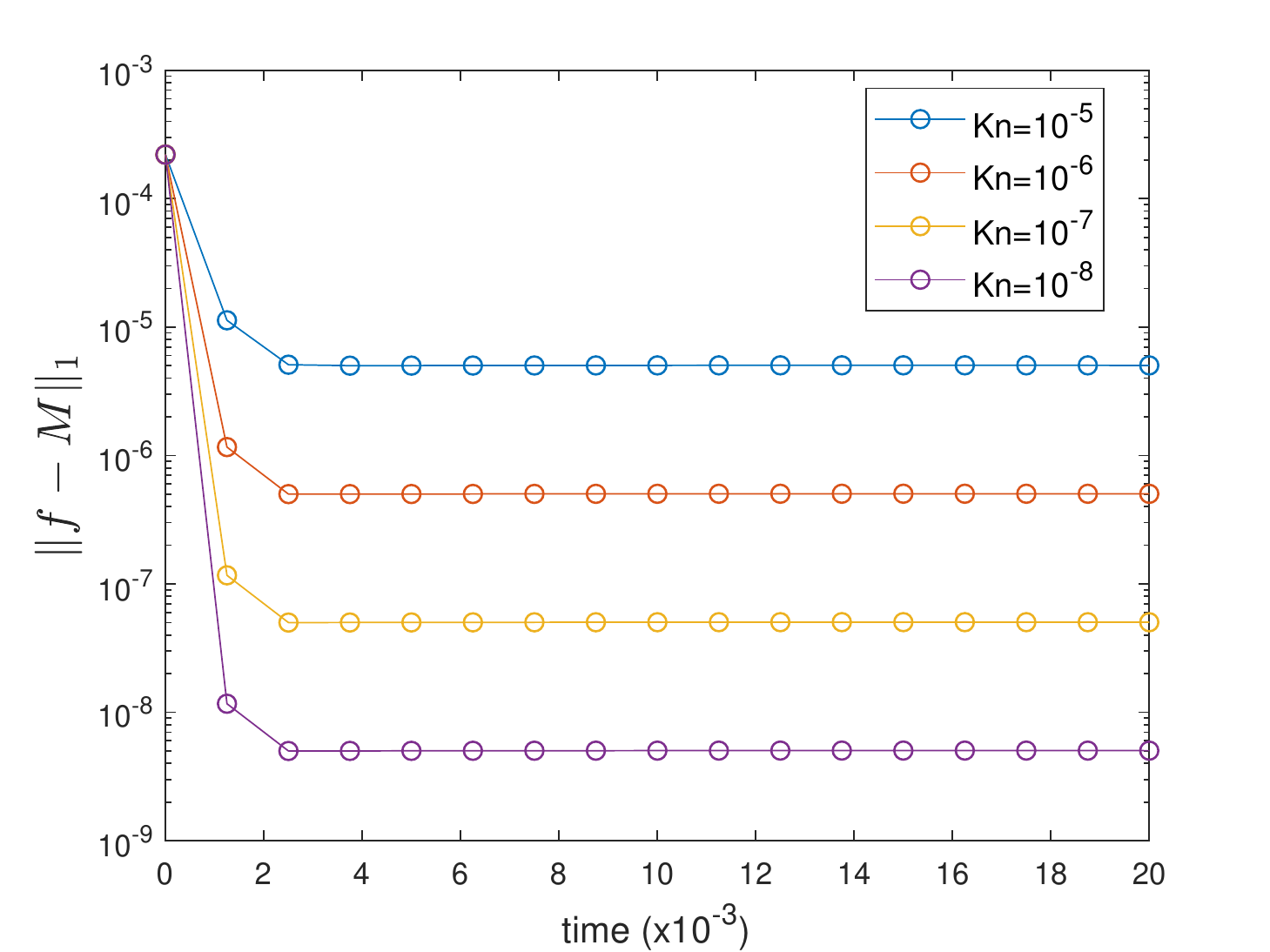}	
		\subcaption{RK3-W35-DM, $\|f-d\mathcal{M} \|_1$}
	\end{subfigure}	
	\caption{Time evolution of $\|f-M \|_1$ for high order methods for the BGK model. $N_v=20$. When using the continuous Maxwellian (left panel), the discrepancy between the distribution function and the Maxwellian saturates for small values of the Knudsen number, while the method based on the discrete Maxwellian (right panel) shows the expected behaviour $\|f-M \|_1={\mathcal O}(\kappa)$. } \label{Ap fig1}	
	%	\begin{subfigure}[b]{0.45\linewidth}
	%		\includegraphics[width=1\linewidth]{Figures/RK3fdW35AptestNv32CMCM}	
	%		\subcaption{RK3+W35+CM, $\|f-CM \|_1$}
	%	\end{subfigure}	
	%	\begin{subfigure}[b]{0.45\linewidth}
	%		\includegraphics[width=1\linewidth]{Figures/RK3fdW35AptestNv32CMDM}	
	%		\subcaption{RK3+W35+CM, $\|f-DM \|_1$}
	%	\end{subfigure}	
	\begin{subfigure}[b]{0.45\linewidth}
		\includegraphics[width=1\linewidth]{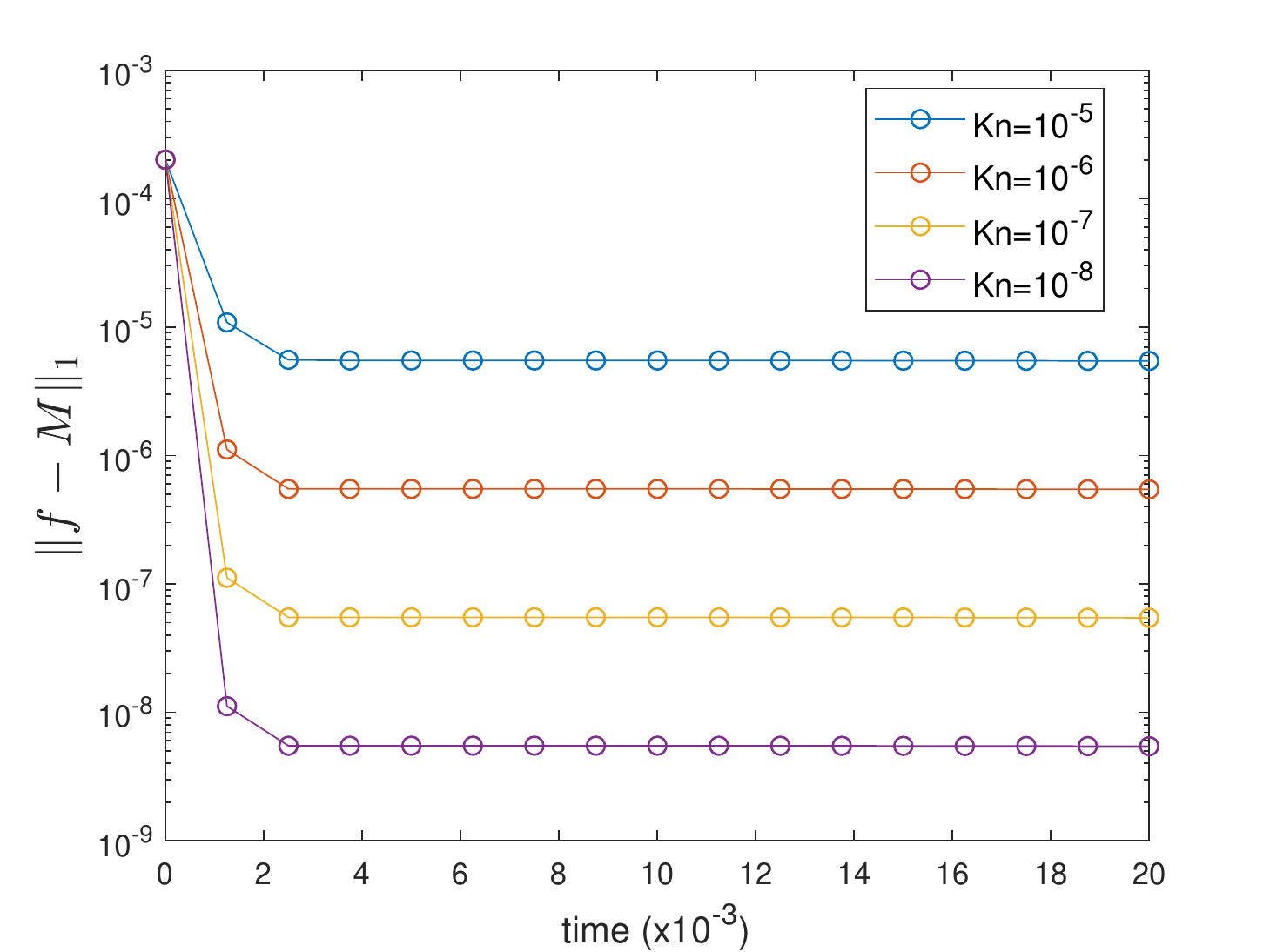}	
		\subcaption{RK3-W35-CM, $\|f-CM \|_1$}
	\end{subfigure}	
	\begin{subfigure}[b]{0.45\linewidth}
		\includegraphics[width=1\linewidth]{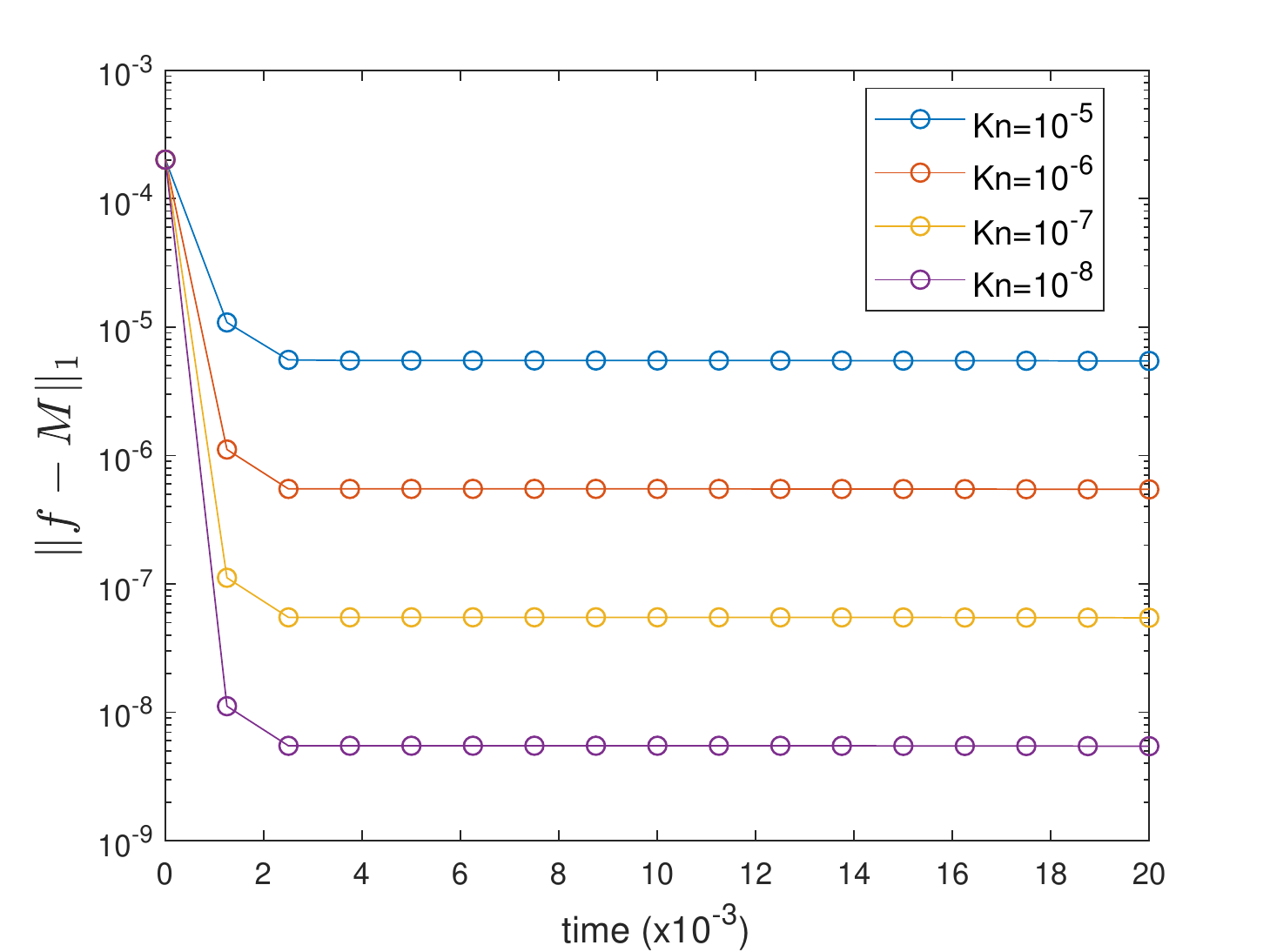}	
		\subcaption{RK3-W35-DM, $\|f-d\mathcal{M} \|_1$}
	\end{subfigure}	
	\caption{Time evolution of $\|f-M \|_1$ for high order methods for the BGK model. $N_v=32$. For large enough number of velocity grid points the discrepancy between the 
		function and the Maxwellian appears to be proportional to $\kappa$, up to $\kappa = 10^{-8}$, for both schemes.}	
	\label{Ap fig2}
\end{figure}

\begin{figure}[]
	\centering
	\begin{subfigure}[b]{0.45\linewidth}
		\includegraphics[width=1\linewidth]{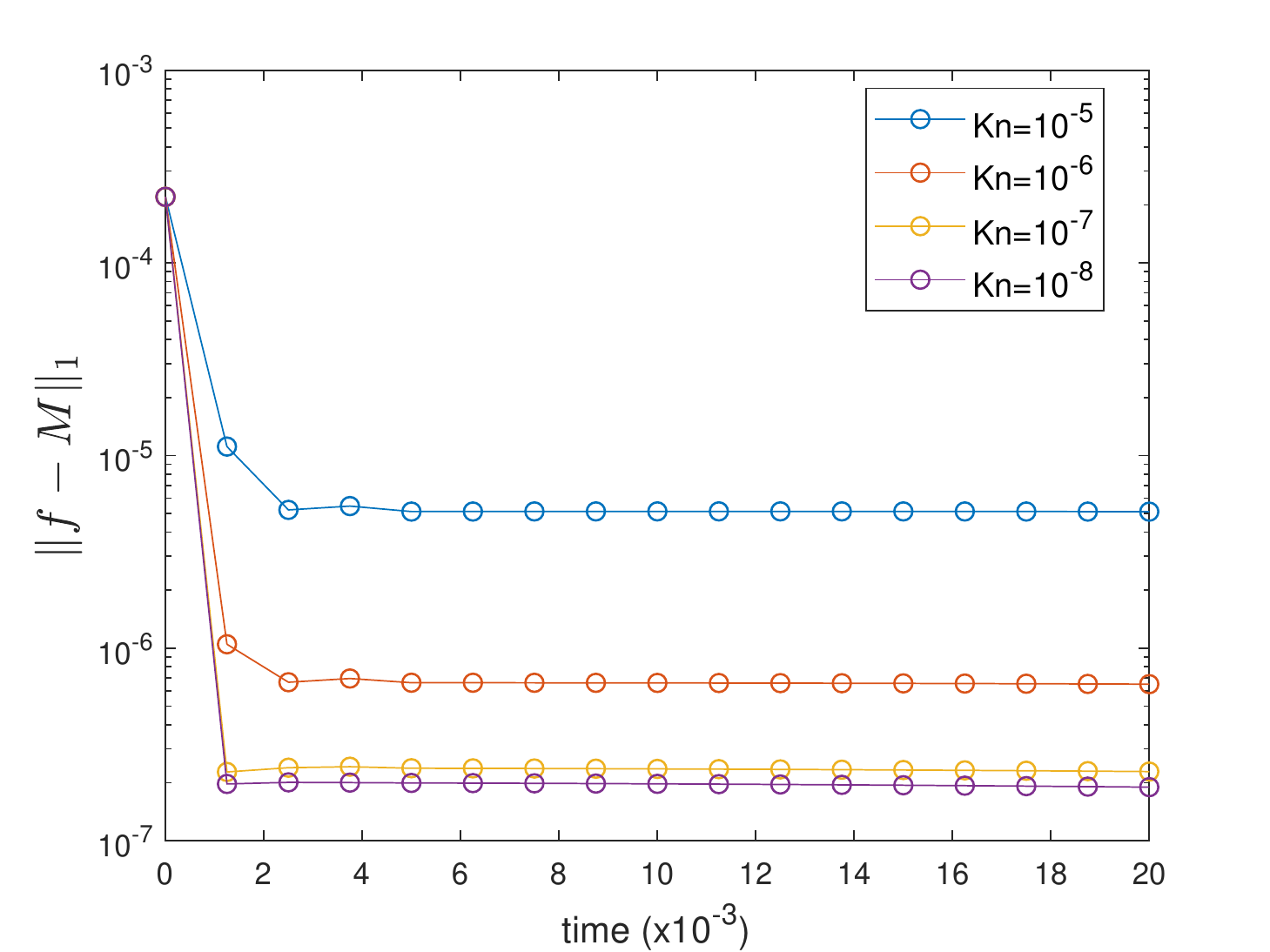}	
		\subcaption{BDF3-W35-CM, $\|f-CM \|_1$}
	\end{subfigure}	
	\begin{subfigure}[b]{0.45\linewidth}
		\includegraphics[width=1\linewidth]{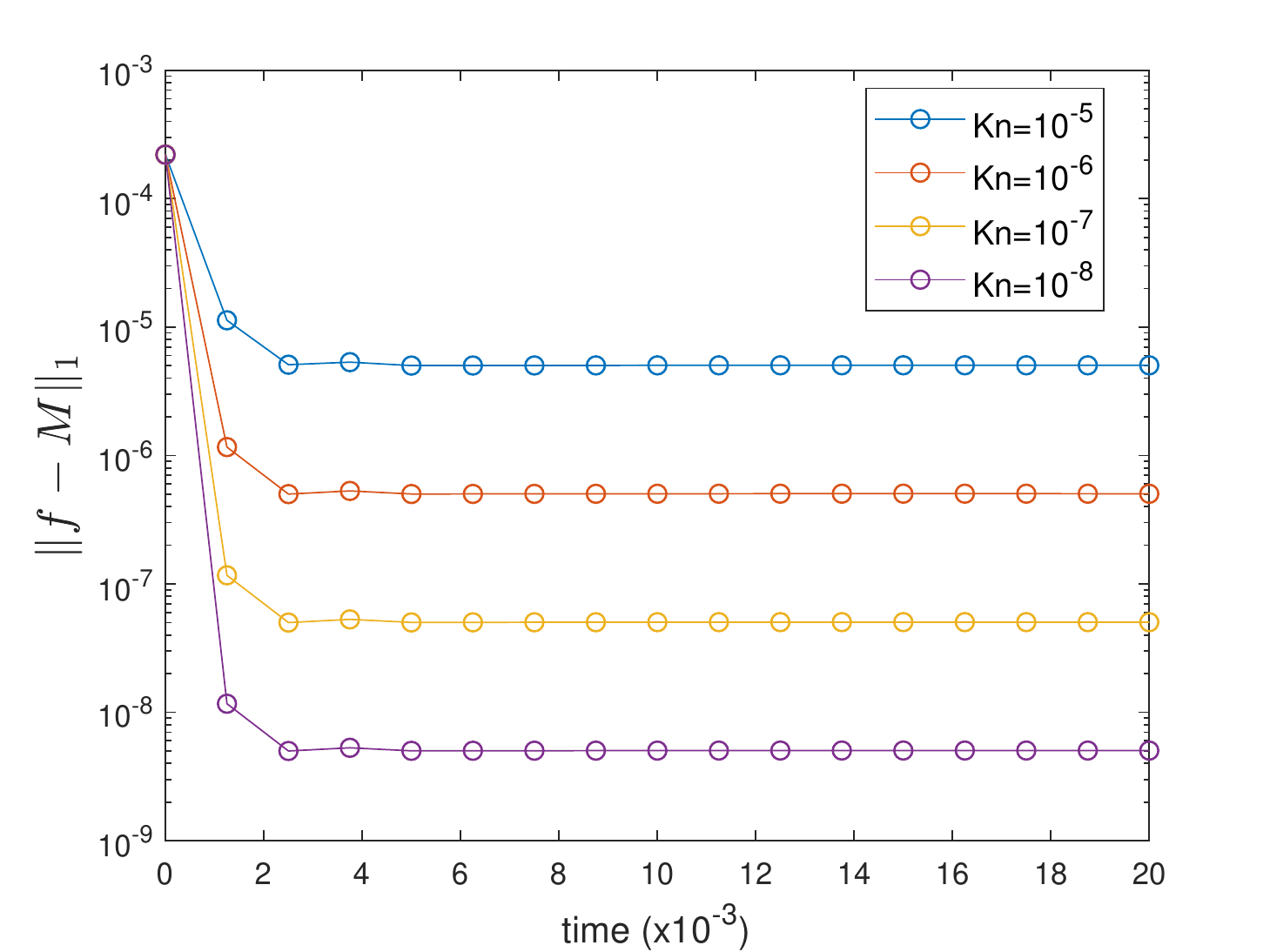}	
		\subcaption{BDF3-W35-DM, $\|f-d\mathcal{M} \|_1$}
	\end{subfigure}	
	\caption{Same as Figure \ref{Ap fig1}, but with BDF based schemes}
	\label{Ap fig3}	
	\begin{subfigure}[b]{0.45\linewidth}
		\includegraphics[width=1\linewidth]{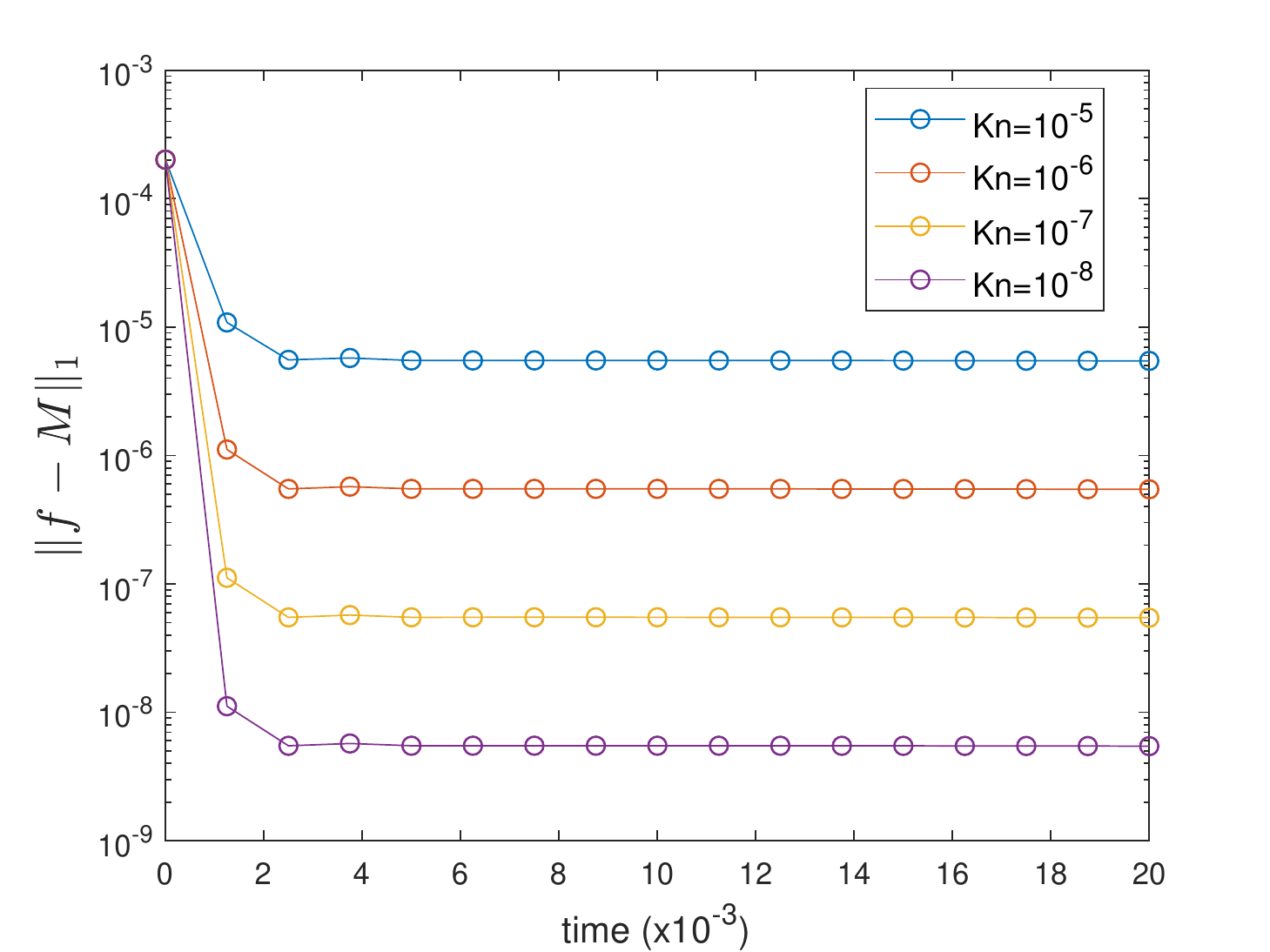}	
		\subcaption{BDF3-W35-CM, $\|f-CM \|_1$}
	\end{subfigure}	
	\begin{subfigure}[b]{0.45\linewidth}
		\includegraphics[width=1\linewidth]{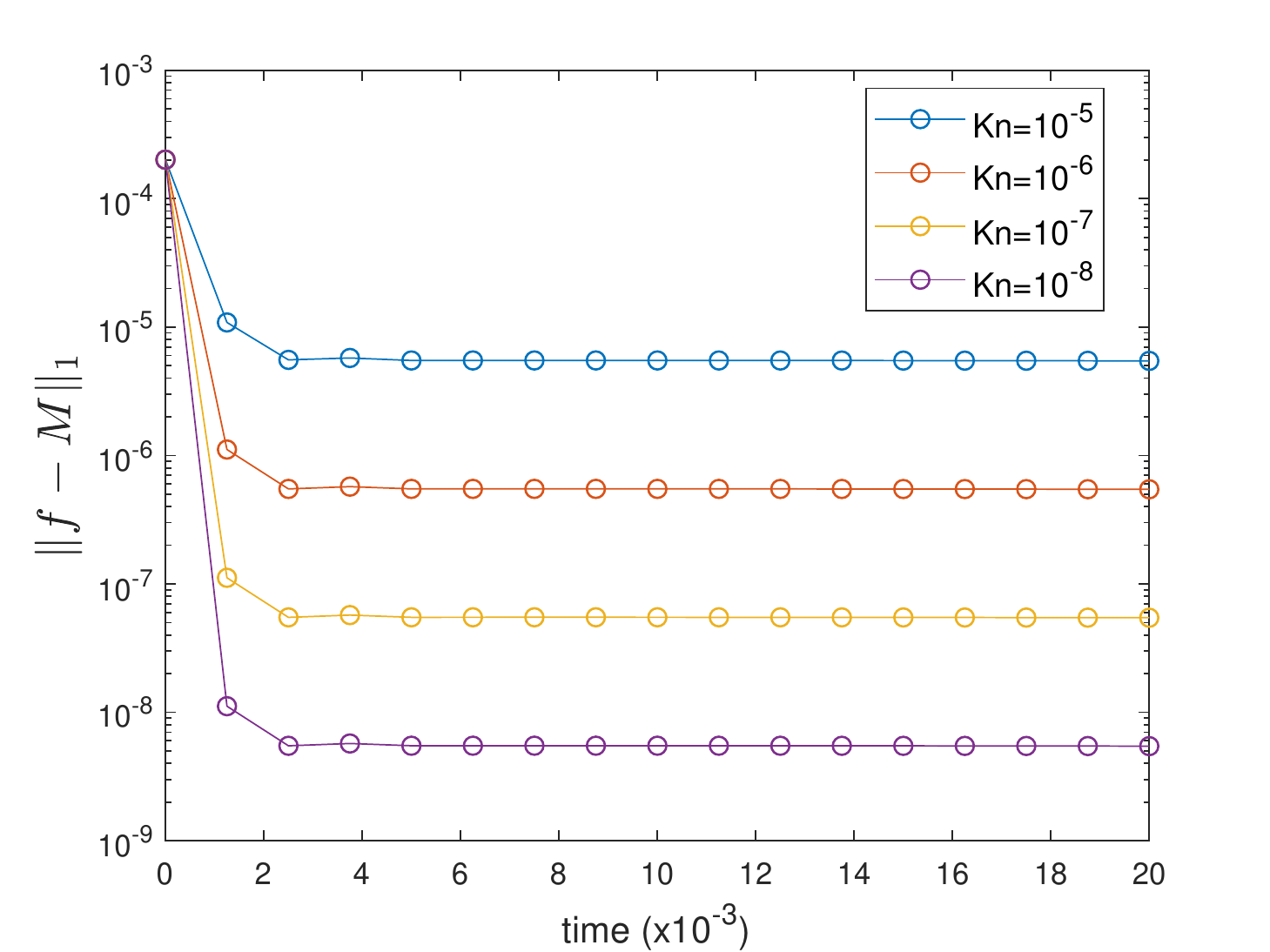}	
		\subcaption{BDF3-W35-DM, $\|f-d\mathcal{M} \|_1$}
	\end{subfigure}	
	\caption{Same as Figure \ref{Ap fig2}, but with BDF based schemes}	
	\label{Ap fig4}
\end{figure}

\subsection{Test 4}
The final test is the classical Riemann problem. To observe the Euler limit, we take $\kappa=10^{-6}$. Moreover, to see the  influence of the conservative correction, we compare our scheme with the standard non-conservative semi-Lagrangian scheme. Initial condition is given by the Maxwellian computed from
\[
(\rho_0,u_0,p_0) = \bigg\{\begin{array}{ll}
(2.25,0,1.125), & \text{for } x\le 0.5\\
(3/7,0,1/6), & \text{for } x>0.5
\end{array}
\]
We use freeflow boundary condition. Computations are performed on $ x \in [0, 1]$ , $ v \in [-10, 10]$ upto final time $T_f=0.16$.
With small Knudsen number, both conservative and non-conservative schemes are stable and they enable us to use CFL = 2.
We take $N_x=200$ for both schemes. For the proposed schemes, RK3-W35-DM and BDF3-W35-DM, we take $N_v=30$. For non-conservative schemes, RK3W35 and BDF3W35, we take $N_v=60$. The larger number of velocity grid points ensures that the conservation error due to the use of continuous Maxwellian is negligible with respect to the one due to lack of conservation of the transport term. 
The results are shown in Figure \ref{fig:Riemann}. It appears that conservative schemes capture shocks correctly and are in perfect agreement with the exact reference solution.

\begin{figure}[]
	\centering
	\begin{subfigure}[b]{0.42\linewidth}
		\includegraphics[width=1\linewidth]{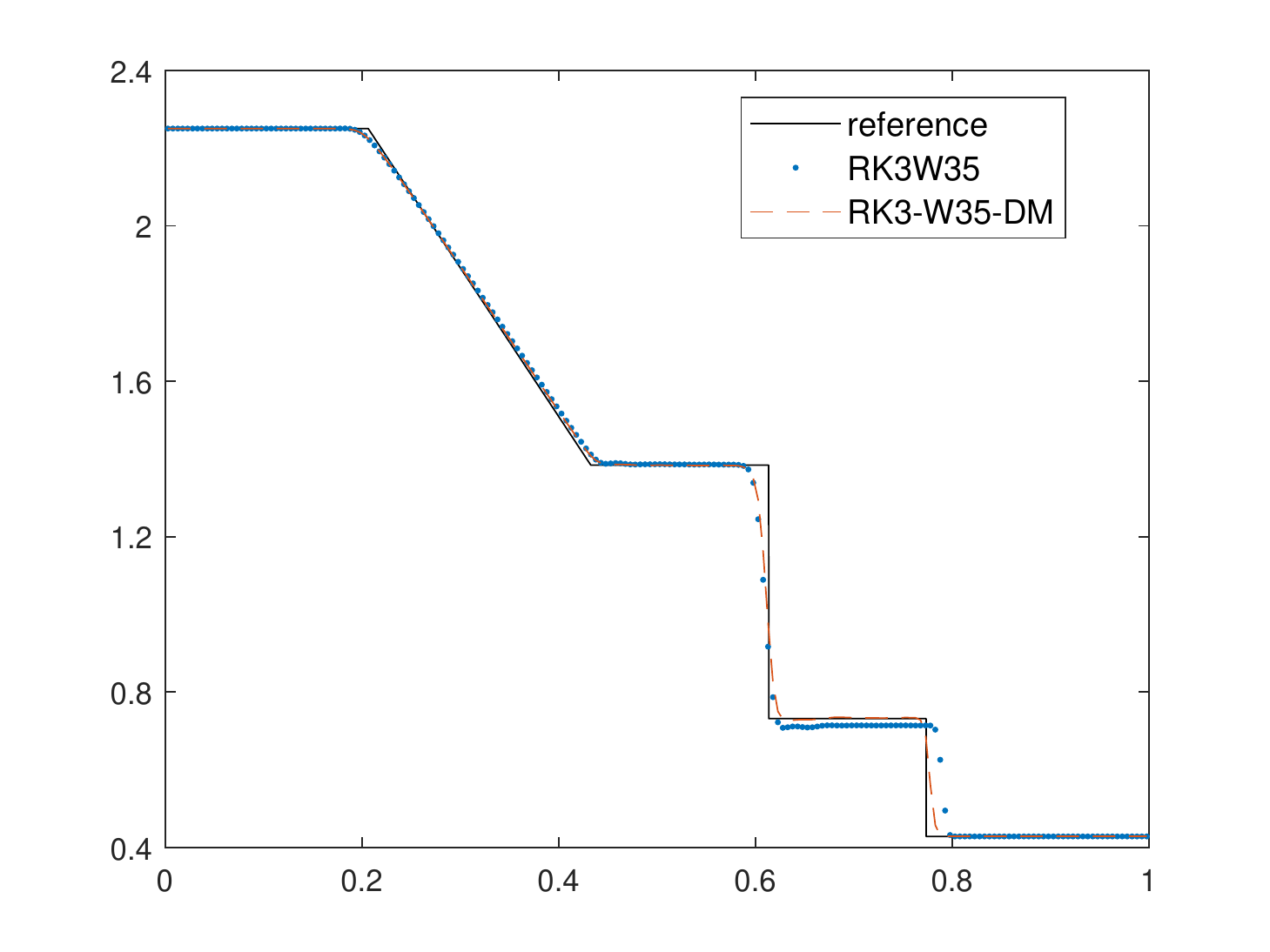}	
		\subcaption{Density}
	\end{subfigure}	
	\begin{subfigure}[b]{0.42\linewidth}
		\includegraphics[width=1\linewidth]{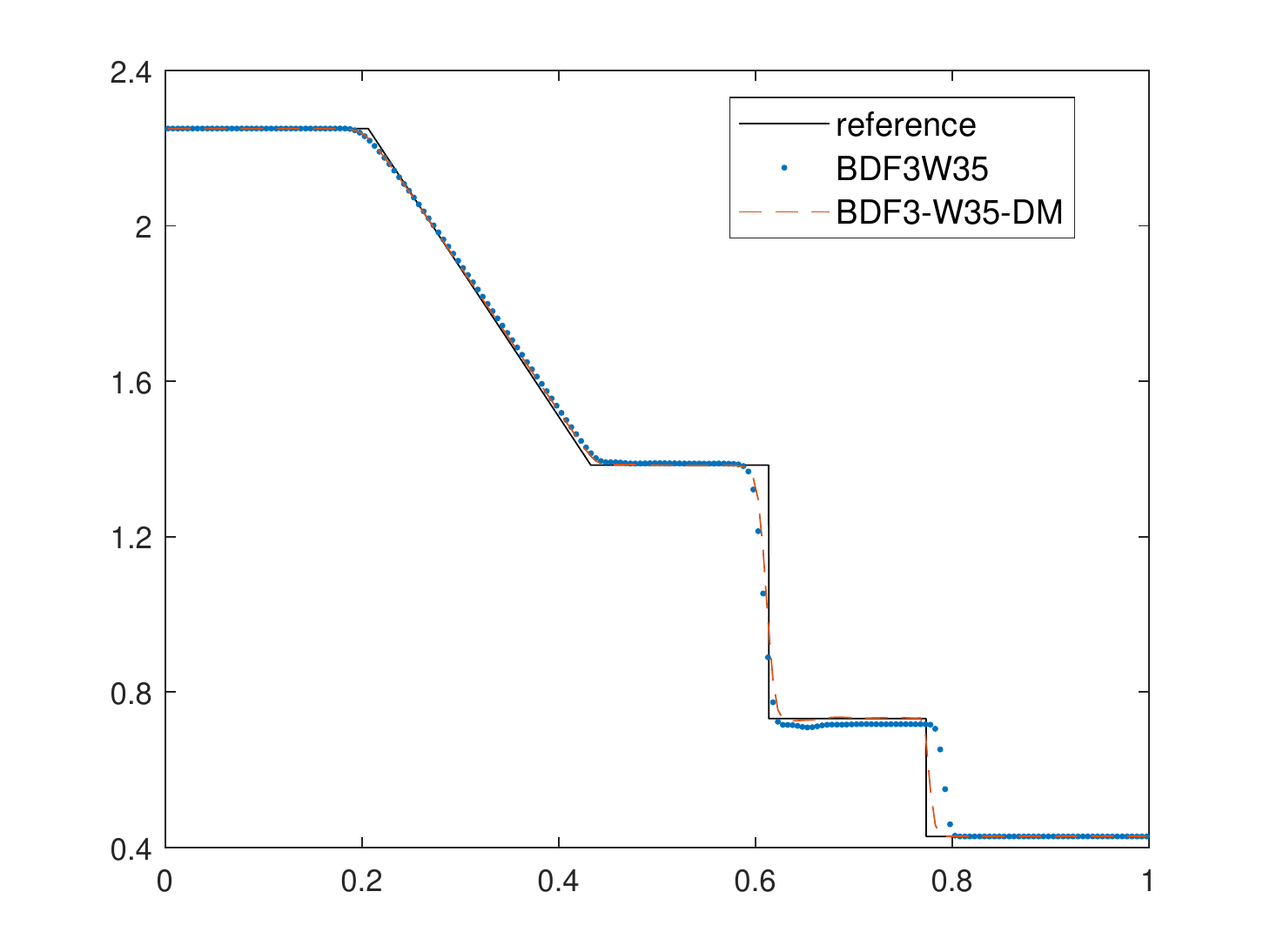}	
		\subcaption{Density}
	\end{subfigure}	
	
	\begin{subfigure}[b]{0.42\linewidth}
		\includegraphics[width=1\linewidth]{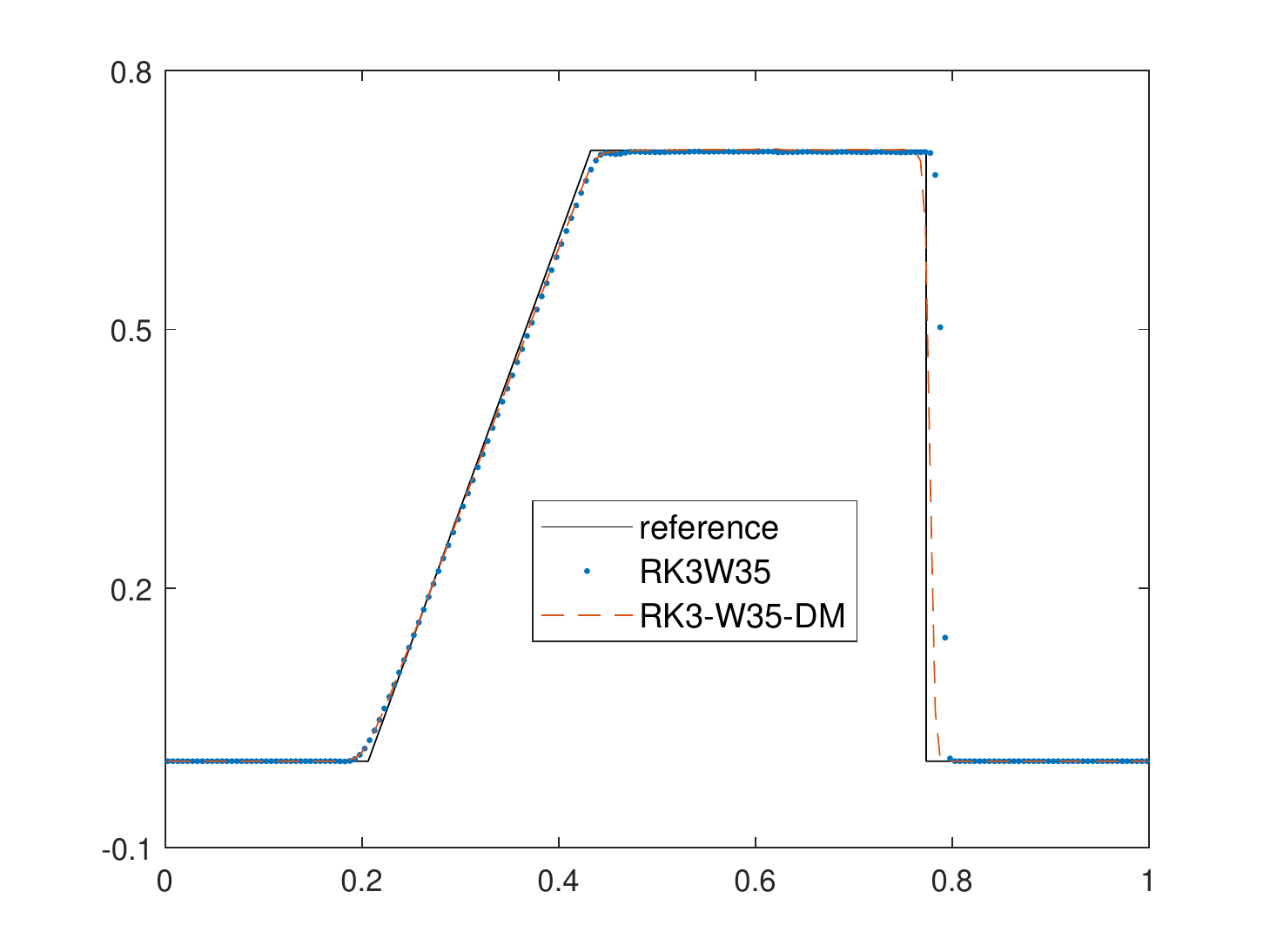}	
		\subcaption{Velocity}
	\end{subfigure}	
	\begin{subfigure}[b]{0.42\linewidth}
		\includegraphics[width=1\linewidth]{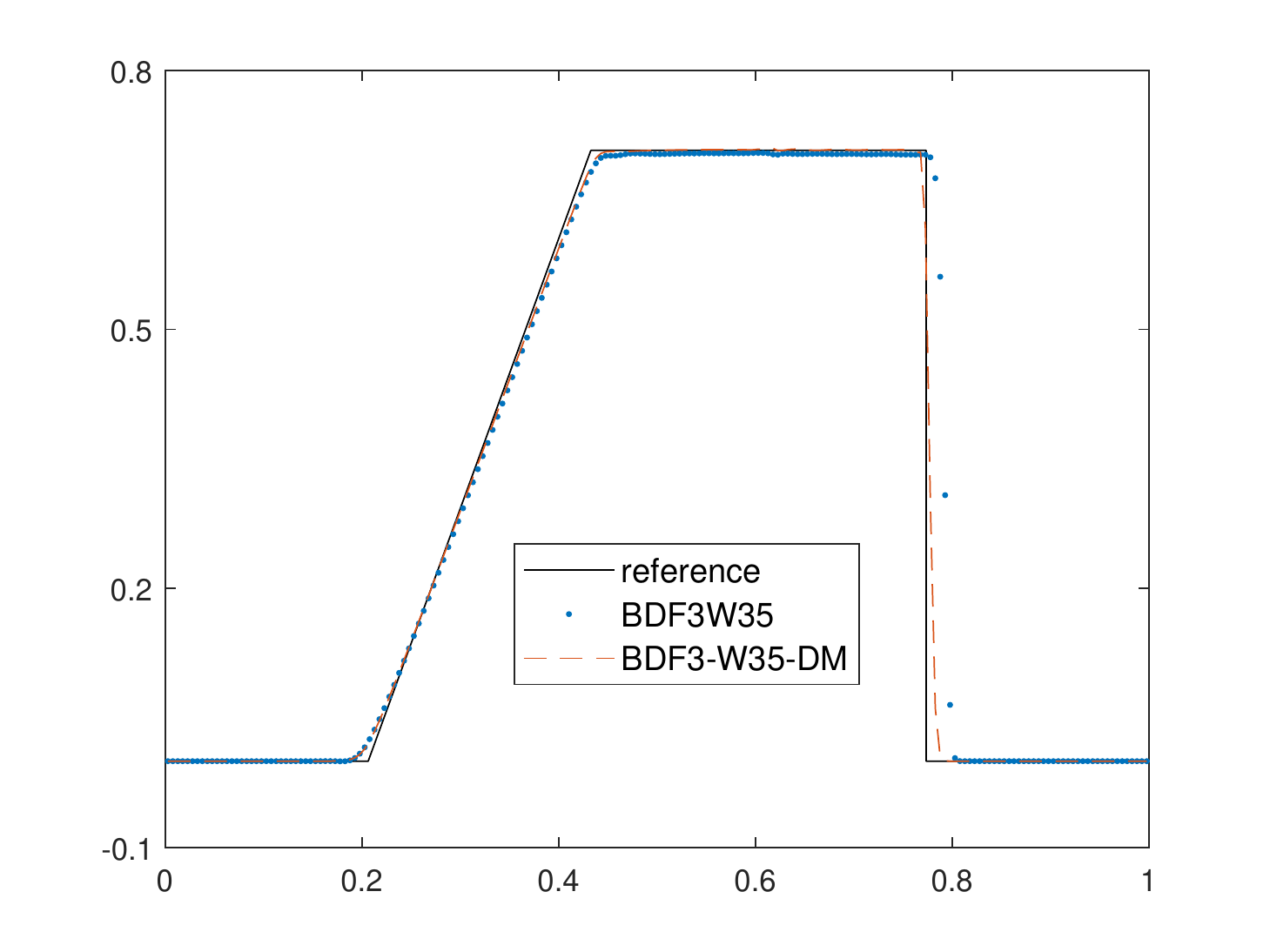}	
		\subcaption{Velocity}
	\end{subfigure}
	
	\begin{subfigure}[b]{0.42\linewidth}
		\includegraphics[width=1\linewidth]{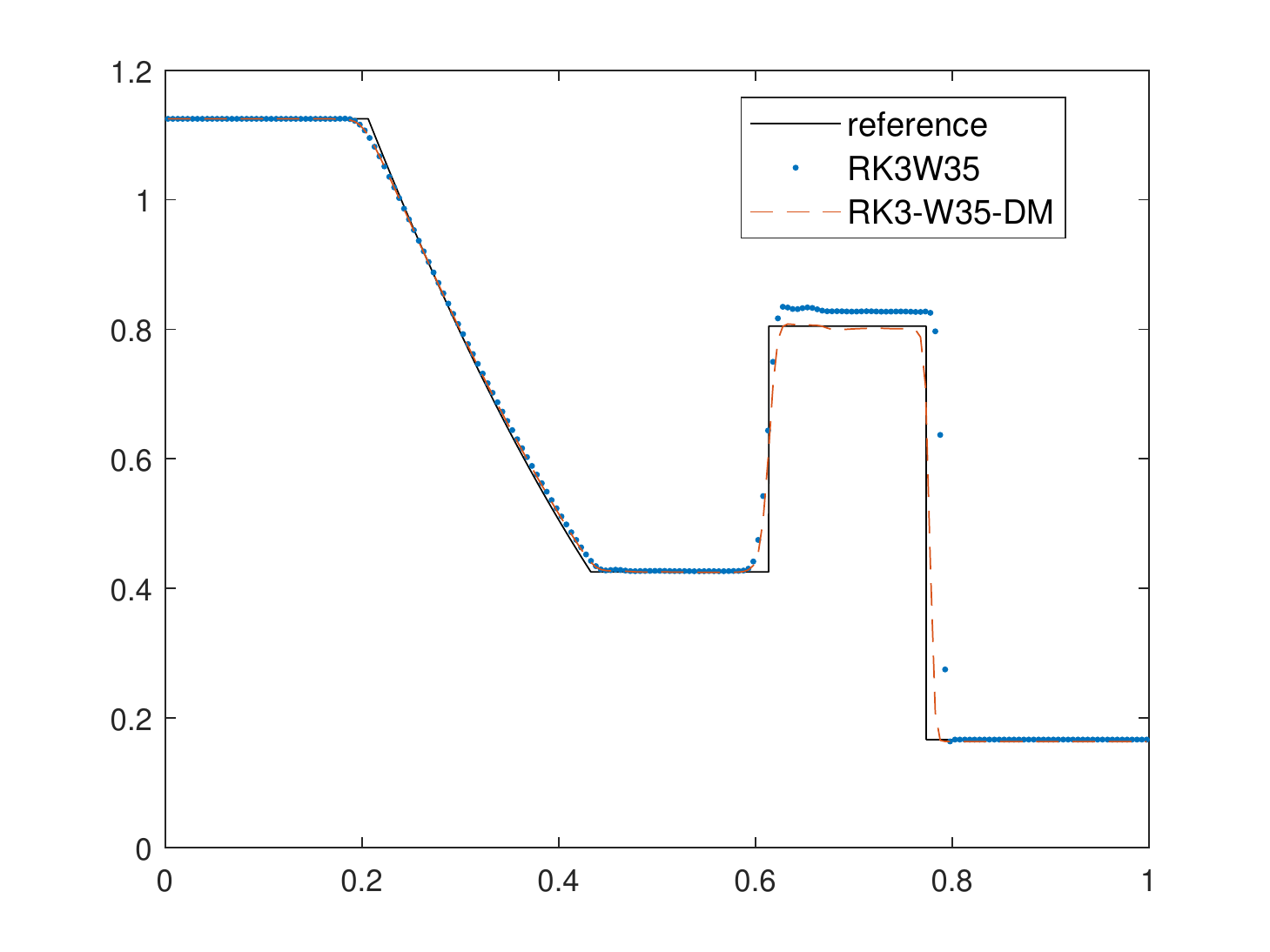}	
		\subcaption{Temperature}
	\end{subfigure}	
	\begin{subfigure}[b]{0.42\linewidth}
		\includegraphics[width=1\linewidth]{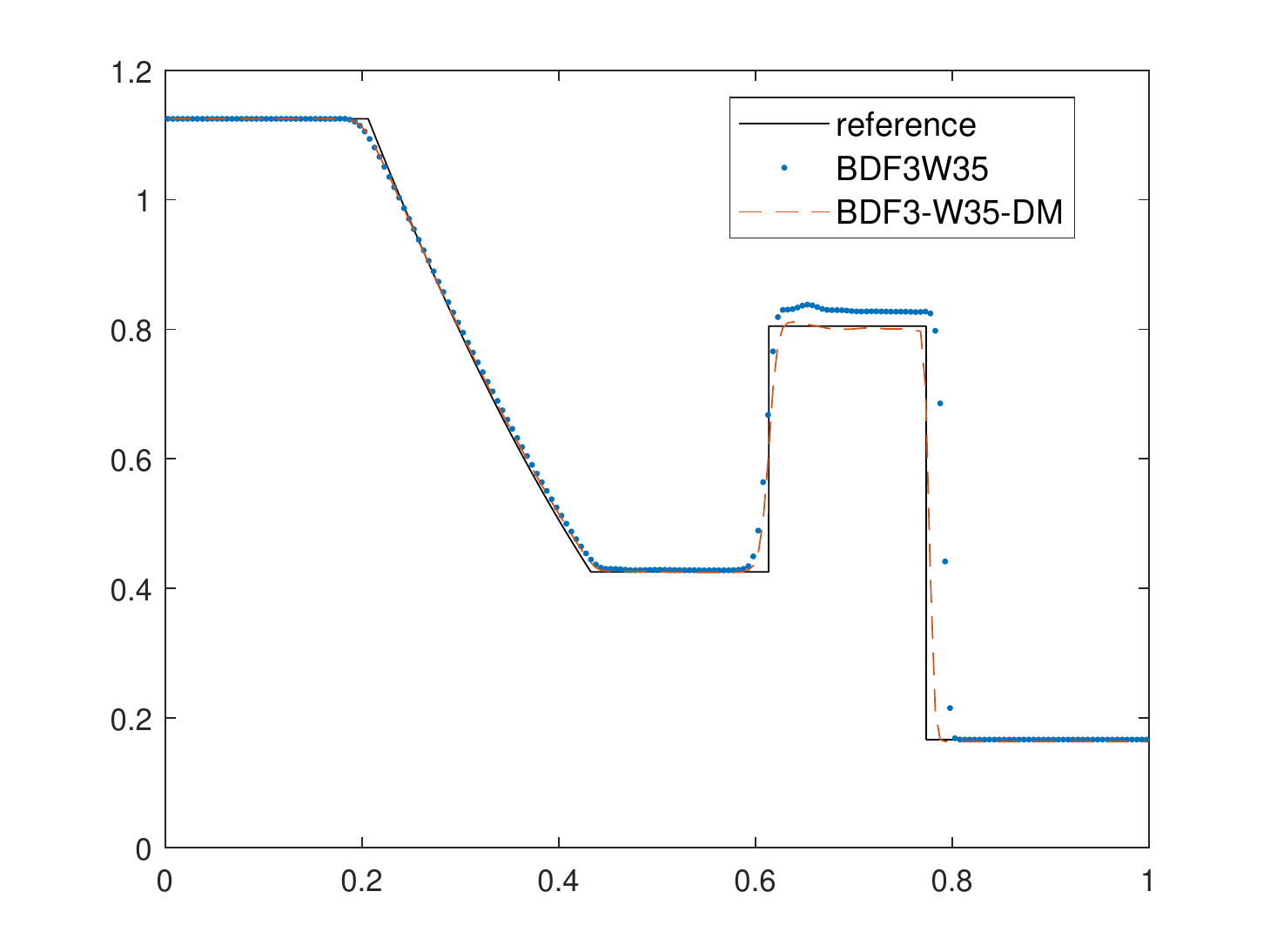}	
		\subcaption{Temperature}
	\end{subfigure}	
	
	\begin{subfigure}[b]{0.42\linewidth}
		\includegraphics[width=1\linewidth]{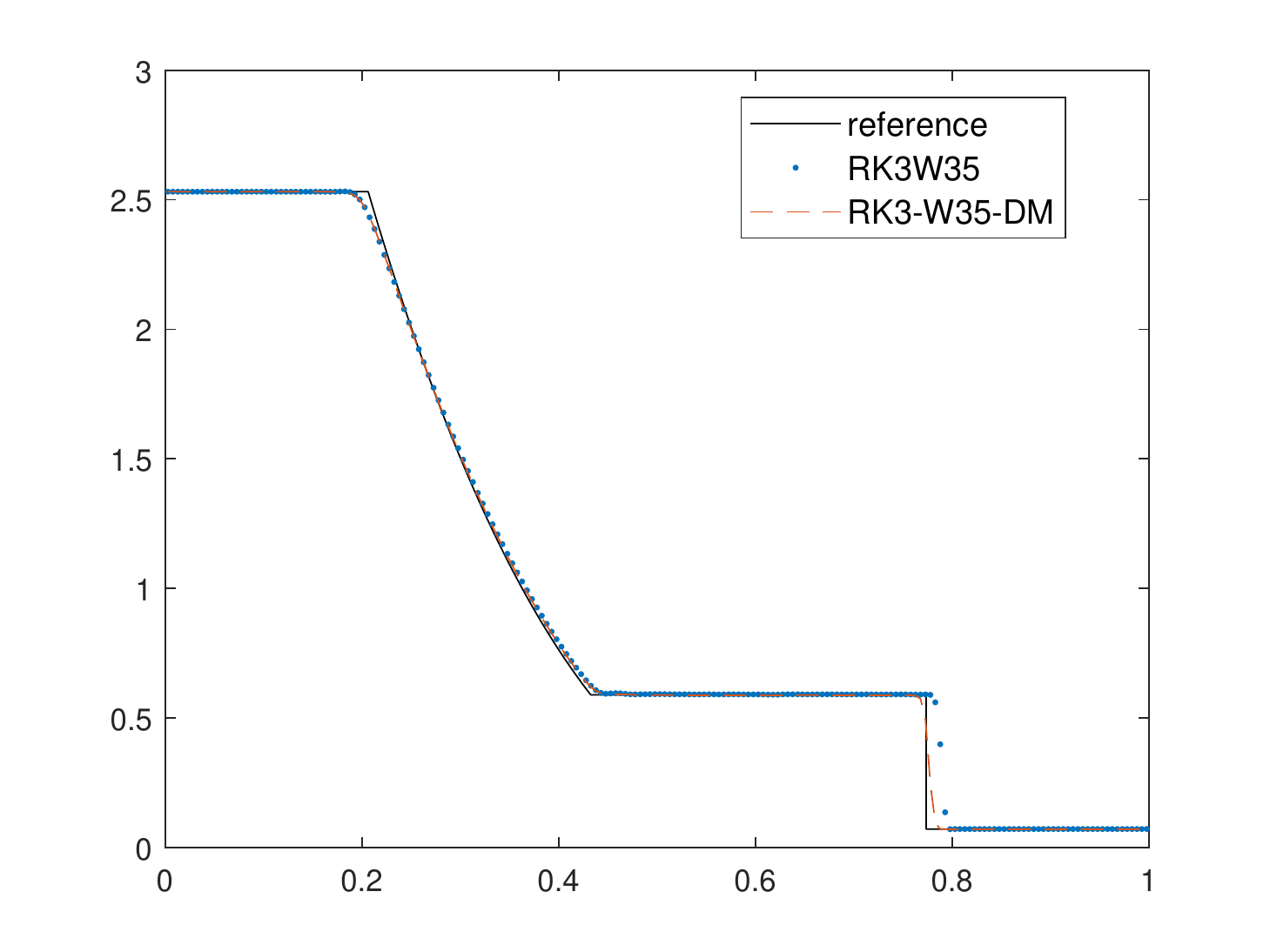}	
		\subcaption{Pressure}
	\end{subfigure}	
	\begin{subfigure}[b]{0.42\linewidth}
		\includegraphics[width=1\linewidth]{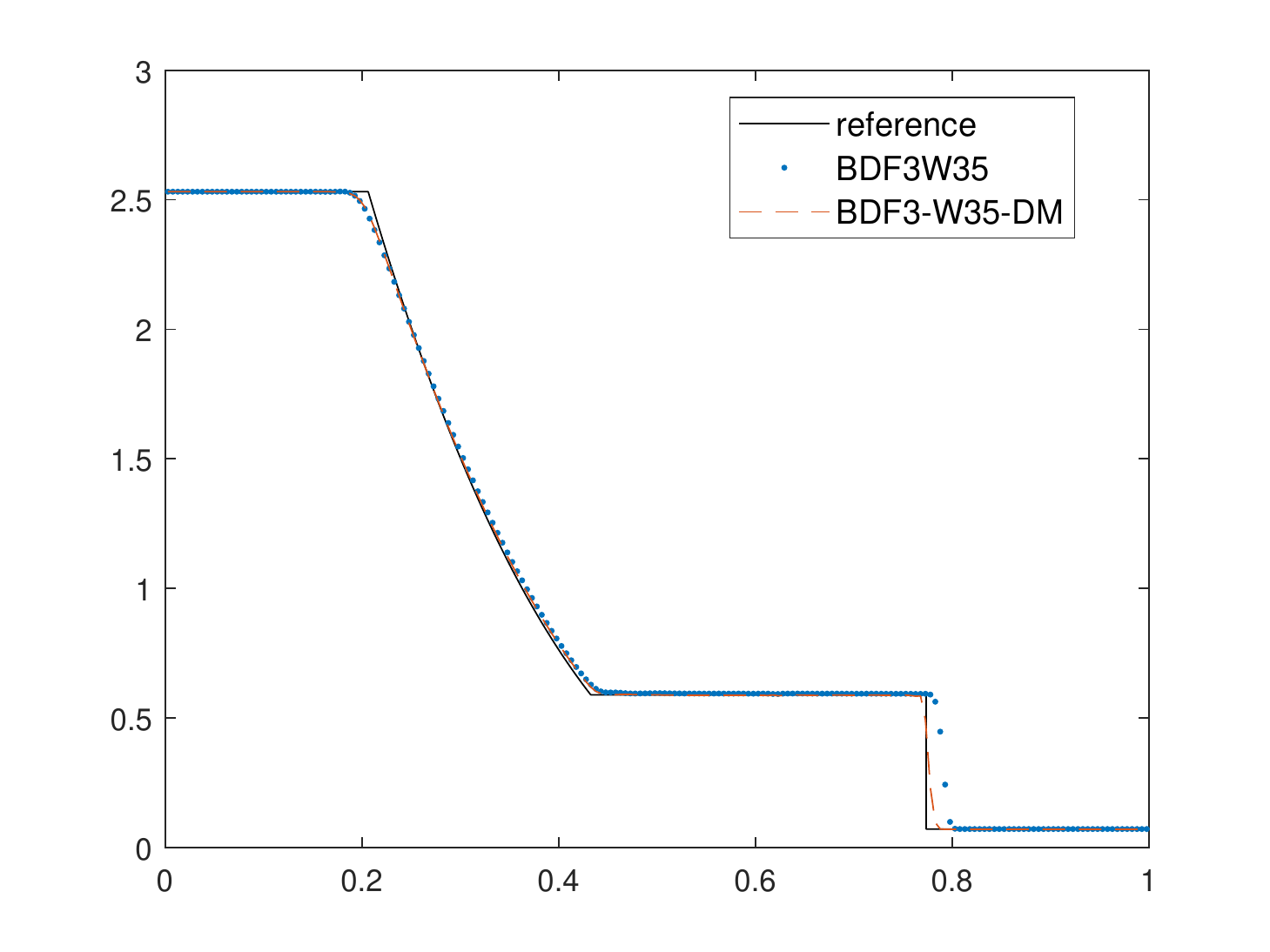}	
		\subcaption{Pressure}
	\end{subfigure}	
	
	\caption{Riemann problem in 1D space and velocity with $\kappa=10^{-6}$. From top to bottom: Density, Velocity, Temperature and Pressure. Red crosses: standard SL schemes, blue squares: new conservative schemes, black line: exact solution.}  	\label{fig:Riemann}
\end{figure}

\section{Conclusions}
In this paper we present high order conservative semi-lagrangian schemes for the numerical solution of the BGK model of the Boltzmann equation. Conservation properties are obtained by using a discrete Maxwellian in the collision operator, and by a conservative correction of the advection term. Exact conservation can be reached up to round-off errors. Together with L-stable treatment of the collisions, exact conservation allows the construction of schemes which become consistent {\em shock-capturing} schemes for the underlying Euler limit, as the Knudsen number $\kappa$ vanishes (AP property), even when using a relatively small number of grid points in velocity.

The conservation properties, and the consequent AP property, have been proven mathematically and verified in several numerical tests. A drawback of the conservative correction procedure is the limitation it imposes on the stability of the schemes. A stability analysis has been performed to understand the reason of such limitation. It is observed that Runge-Kutta based schemes have a wider stability region than multistep-bases ones, with a net improvement over Eulerian based schemes for all Knudsen numbers. Stability restrictions become less severe for small Knudsen numbers, making the schemes competitive in such regimes. 

In this paper, we only consider 1D case in both space and velocity. However, the technique is quite general and can be applied to the multi-dimensional case.

As a work in progress, we are developing new conservative semi-lagrangian schemes that do not suffer from such a CFL limitation, which will be the subject of a forthcoming paper.

\appendix
\section{Proof of the estimate on the conservation error for IE-SL scheme with the discrete Maxwellian}\label{SM proof of conservation error}

%We replace $\mathcal{M}(\tilde{f}_{ij}^n)$ with the discrete Maxwellian ${d\mathcal{M}}(\tilde{f}_{i,j}^{n})$ constructed from $\{\tilde{f}_{i,j}^{n}\}$ to obtain
%\begin{align}\label{numSol dm}
%f_{i,j}^{n+1} = \frac{\kappa \tilde{f}_{ij}^n + \Delta t d\mathcal{\mathcal{M}}(\tilde{f}_{ij}^n)}{\kappa + \Delta t}.
%\end{align}
Let us check in what sense the scheme (\ref{numSol dm}) has a better conservative nature compared to that of (\ref{numSol}).
For this, we first rewrite \eqref{numSol dm} as
\begin{equation}\label{step1}
f_{i,j}^{n+1} - \tilde{f}_{ij}^n=\frac{ \Delta{t} }{\kappa + \Delta t }\left(d\mathcal{M}(\tilde{f}_{ij}^n) - \tilde{f}_{ij}^n \right),
\end{equation}
where $\tilde{f}^n_{ij} = \theta_j f^n_{i^*+1,j}+(1-\theta_j)f^n_{i^*,j}$ with
$i^* = \lfloor i-v_j \Delta t/\Delta x \rfloor$, $\theta_j = (x_i-\tilde{x}_j)/\Delta x$.

Since $\theta_j$ does not depend on $i$, we find that the following telescoping cancellation holds so that $f^n_{i,j}$ and $\tilde{f}^n_{ij}$ share the first moment:
\begin{equation}\label{telescoping}
\sum_{i = 1}^{N_x}  \tilde{f}_{ij}^n=\sum_{i = 1}^{N_x} \big(\theta_j f_{i^* +1 ,j}^n + (1-\theta_j) f_{i^*,j}^n\big)=\sum_{i = 1}^{N_x}  f_{i,j}^n.
\end{equation}
Multiplying (\ref{step1}) by $\phi(v_j) = (1; v_j; v_j^2/2)$, taking summation on $i,j$ and inserting (\ref{telescoping}), one gets
\begin{equation*}
\sum_{i = 1}^{N_x}\sum_{j=0}^{N_v}\left(f_{i,j}^{n+1} - f_{i,j}^n\right)\phi(v_j)\Delta v\Delta x=\frac{ \Delta{t} }{\kappa + \Delta t }\sum_{i = 1}^{N_x}\sum_{j=0}^{N_v}\left(d\mathcal{M}(\tilde{f}_{ij}^n) - \tilde{f}_{ij}^n \right)\phi(v_j)\Delta v\Delta x.
\end{equation*}
Summing further in time step, we have
\begin{equation}\label{nonconservative error}
\begin{split}
&\sum_{i = 1}^{N_x}\sum_{j=0}^{N_v}\left(f_{i,j}^{Nt} - f_{i,j}^0\right)\phi(v_j)\Delta v\Delta x\\
&\qquad=\frac{ \Delta{t} }{\kappa + \Delta t }\sum_{k = 0}^{N_t-1}\sum_{i = 1}^{N_x}\sum_{j=0}^{N_v}\left(d\mathcal{M}(\tilde{f}_{ij}^k) - \tilde{f}_{ij}^k \right)\phi(v_j)\Delta v\Delta x,
\end{split}
\end{equation}
then, denoting the $\ell$th component of $\phi(v_j)$ by $\phi_\ell(v_j)$, $\ell=1,2,3$, and using a variant of (\ref{dm newton}):
\begin{align}\label{dm tol}
\max_{1 \leq \ell \leq 3}\left|\sum_{j=0}^{N_v} \left[  d\mathcal{M}(\tilde{f}_{ij}^k) -\tilde{f}_{ij}^k  \right]  \phi_\ell(v_j)\Delta v\right|<tol,
\end{align}
we obtain the following estimate:
\begin{align*}
&\max_{1 \leq \ell \leq 3}\left|\sum_{i = 1}^{N_x}\sum_{j=0}^{N_v}\left(f_{i,j}^{Nt} - f_{i,j}^0\right)\phi_\ell(v_j) \Delta v\Delta x\right| \\
%&=\| \frac{ \Delta{t} }{\kappa + \Delta t }\sum_{k = 1}^{Nt}\sum_{i = 1}^{N_x}\sum_{j=0}^{N_v}\left(dM(\tilde{f}_{ij}^k) - \tilde{f}_{ij}^k \right)\phi(v_j)\Delta v\Delta x\|_\infty \cr
&\qquad\qquad \leq  \frac{ \Delta{t} }{\kappa + \Delta t }\sum_{k = 0}^{Nt-1}\sum_{i = 1}^{N_x}\max_{1 \leq \ell \leq 3}\left| \sum_{j=0}^{N_v}\left(d\mathcal{M}(\tilde{f}_{ij}^k) - \tilde{f}_{ij}^k \right)\phi_\ell(v_j)\Delta v \right| \Delta x \cr
&\qquad\qquad \leq  \frac{N_t \Delta{t} }{\kappa + \Delta t }(x_{max}-x_{min}) tol.
\end{align*}
%In the last inequality, we used \eqref{dm tol}.
%	\begin{align} \label{discrete tol}
%		\|\sum_{j=0}^{N_v} \left[  d\mathcal{M}(\tilde{f}_{ij}^k) -\tilde{f}_{ij}^k  \right]  \phi(v_j)\Delta v\|_\infty<tol.
%	\end{align}
This estimate tells us that error is stacked in each time step by $tol$. So, the total conservation error in the end essentially depends on $N_t \times  tol$ uniformly in $\kappa$. Therefore, $tol$ should be taken small enough to attain a machine precision conservation error.

\section{General framework of G-WENO interpolation}\label{G-Weno}
In this section, we illustrate the G-WENO interpolation of degree $2n-1$. Let $U=\{u_j\}, j \in I$ be a set of given values of a function $u$ on a space grid $x_j$, $j \in I$.

We start with the Lagrange polynomial $Q(x)$ built on the stencil $S =\{x_{j-n+1},...,x_{j+n}\}$:

\begin{equation}\label{Qp}
Q(x)=\sum^n_
{k=1}
C_k(x)P_k(x),
\end{equation}
where the ``linear weights" $C_k(x)$ are polynomials of degree $n-1$ and $P_k$ are polynomials of degree $n$ interpolating $U$ on the stencil $S_k =\{x_{j-n+k},...,x_{j+k}\}$, $k=1,...,n$. The linear
weights $C_k(x)$ ($k=1,...,n$) are determined to satisfy the following two properties \cite{CFR}:
\begin{enumerate}
	\item $\displaystyle C_k(x_i)=0 \text{ for } x_i \in S-S_k$.
	\item $\displaystyle \sum_k C_k(x_i)=1 \text{ for } x_i \in S.$
\end{enumerate}

To guarantee non-oscillatory property, we replace the linear weights $C_k(x)$ by the non-linear weights $\omega_k(x)$:
\begin{equation}
\omega_k(x)= \frac{\alpha_k(x)}{\sum_l\alpha^l(x)},
\end{equation}
where $\alpha_k(x)$ is defined by
\begin{equation}\label{alpha}
\alpha_k(x)=\frac{ C_k(x)}{(\beta_k+\epsilon)^2},
\end{equation}
with the choice of $\epsilon = 10^{-6}$.  The smoothness indicators $\beta_k$ in (\ref{alpha}) is defined by
\begin{equation}
\beta_k =
\sum^n_{l=1}
\int^{x_{j+1}}_{x_j}\Delta x^{2l-1}(P_k^{(l)} )^2 dx.
\end{equation}

The nonlinear weights $\omega_k(x)$ are designed to put more weights on the smooth part of $u$ and less weights on the discontinuous part of $u$.

Finally, the G-WENO reconstruction of the values $U =\{u_j\}_{j\in I}$ reads $$I[U](x)=\sum^n_{k = 1} \omega_k(x)P_k(x)$$.

In the following, we explicitly construct the G-WENO interpolations of order 3 and 5.
\subsection{G-WENO of order 3 (WENO23)}
The G-WENO interpolation of order 3 can be represented with two second order polynomials $P_L$ and $P_R$ built respectively on stencils $\{x_{j-1},x_j,x_{j+1}\}$ and $\{x_j,x_{j+1},x_{j+2}\}$:
$$P(x)=\omega_L P_L(x) + \omega_R P_R(x),$$
where the non-linear weights $\omega_L$ and $\omega_R$ are given by
\begin{align}\label{weight}
\omega_\ell=\frac{\alpha_\ell}{\sum_\ell\alpha_\ell},\quad  \alpha_\ell=\frac{C_\ell}{(\epsilon+\beta_\ell)^2}, \quad \ell=L,R
\end{align}
with
\begin{align*}
C_L=\frac{x_{j+2}-x}{3\Delta x}, \quad C_R=\frac{x-x_{j-1}}{3\Delta x},
\end{align*}

and
\begin{align*}
\beta_L&= \frac{13}{12}v_{j-1}^2 + \frac{16}{3}v_{j}^2 +\frac{25}{12}v_{j+1}^2 - \frac{13}{3}v_{j-1}v_{j} + \frac{13}{6}v_{j-1}v_{j+1} - \frac{19}{3}v_{j}v_{j+1}\cr
\beta_R&= \frac{13}{12}v_{j}^2 + \frac{16}{3}v_{j+1}^2 +\frac{25}{12}v_{j+2}^2 - \frac{13}{3}v_{j}v_{j+1} + \frac{13}{6}v_{j}v_{j+2} - \frac{19}{3}v_{j+1}v_{j+2}.
\end{align*}
\subsection{G-WENO of order 5 (WENO35)}
For the G-WENO interpolation of order 5, we use
third order polynomials $P_L$, $P_C$ and $P_R$ built respectively on stencils $\{x_{j-2},x_{j-1},x_j,x_{j+1}\}$, $\{x_{j-1},x_{j},x_{j+1},x_{j+2}\}$ and $\{x_{j},x_{j+1},x_{j+2},x_{j+3}\}$:
$$P(x)=\omega_L P_L(x) + \omega_C P_C(x) +\omega_R P_R(x),$$
where the non-linear weights $\omega_L$, $\omega_C$ and $\omega_R$ are given by
\begin{align*}
\omega_\ell=\frac{\alpha_\ell}{\sum_\ell\alpha_\ell},\quad  \alpha_\ell=\frac{C_\ell}{(\epsilon+\beta_\ell)^2}, \quad \ell=L,C,R
\end{align*}
with
\begin{align*}
C_L=\frac{(x-x_{j+2})(x-x_{j+3})}{20{\Delta x}^2}, \quad C_C=-\frac{(x-x_{j-2})(x-x_{j+3})}{10{\Delta x}^2}, \quad C_R=\frac{(x-x_{j-2})(x-x_{j-1})}{20{\Delta x}^2}
\end{align*}
and
\begin{align*}
\beta_L=& \frac{407}{90}v_{j+1}^2 + \frac{721}{30}v_{j}^2 +\frac{248}{15}v_{j-1}^2 +\frac{61}{45}v_{j-2}^2 - \frac{1193}{60}v_{j+1}v_{j-2}+ \frac{439}{30}v_{j+1}v_{j-1}\cr   &- \frac{683}{180}v_{j+1}v_{j-2}- \frac{2309}{60}v_{j}v_{j-1}+ \frac{309}{30}v_{j}v_{j-2}-\frac{553}{60}v_{j-1}v_{j-2}\cr
\beta_C=& \frac{61}{45}v_{j-1}^2 + \frac{331}{30}v_{j}^2 +\frac{331}{30}v_{j+1}^2 +\frac{61}{45}v_{j+2}^2 - \frac{141}{20}v_{j-1}v_{j} + \frac{179}{30}v_{j-1}v_{j+1} \cr
&- \frac{293}{180}v_{j-1}v_{j+2}- \frac{1259}{60}v_{j}v_{j+1}+ \frac{179}{30}v_{j}v_{j+2}- \frac{141}{20}v_{j+1}v_{j+2}\cr
\beta_R=& \frac{407}{90}v_{j}^2 + \frac{721}{30}v_{j+1}^2 +\frac{248}{15}v_{j+2}^2 +\frac{61}{45}v_{j+3}^2 - \frac{1193}{60}v_{j}v_{j+3}+ \frac{439}{30}v_{j}v_{j+2}\cr   &- \frac{683}{180}v_{j}v_{j+3}- \frac{2309}{60}v_{j+1}v_{j+2}+ \frac{309}{30}v_{j+1}v_{j+3}-\frac{553}{60}v_{j+2}v_{j+3}.
\end{align*}

%\begin{theorem}
%  An example theorem.
%\end{theorem}

%\lipsum[102]

%\begin{lemma}
%  An example lemma.
%\end{lemma}

%\lipsum[103-105]

%Here is an example citation: \cite{KoMa14}.

%\section[Proof of Thm]{Proof of \cref{thm:bigthm}}
%\label{sec:proof}
%
%\lipsum[10-100]

\section{Details on the stability analysis}\label{StabilityCalculus} 
Here we find conditions such that Eq. (\ref{Ey}) is satisfied. In order to make clear the calculation in the method (\ref{DIRK3g}) we take $c_1 = \gamma_1$ and $b_3 = \gamma_2$ then we get for the stability function (see \cite{HW}). Note that it is sufficient to have all $E_{2j}\ge 0$ for the $I$-stability. These are the conditions that we actually use, in order to simplify the analysis.

We consider
\begin{equation}\label{SFs}
R(z) = \frac{P(z)}{Q(z)} = \frac{p_0 + p_1 z + p_2z^2 + p_3z^3}{q_0 - q_1 z + q_2z^2 - q_3z^3}
\end{equation} 
with with the following quantities:
\begin{align}\label{pp}
p_0 &= 1, \quad p_1 = \left(\frac{q_0}{1!} -\frac{q_1}{0!} \right), \quad p_2 = \left(\frac{q_0}{2!} -\frac{q_1}{1!} +  \frac{q_2}{0!} \right), \quad p_3 = \left(\frac{q_0}{3!} -\frac{q_1}{2!}  + \frac{q_2}{1!} - \frac{q_3}{0!} \right)
\end{align}
and
\begin{align}\label{qq}
q_0 &= 1, \quad q_1 = \gamma_1 + \gamma_2 + \gamma_3, \quad q_2 = \gamma_1\gamma_2 + \gamma_1\gamma_3 + \gamma_2\gamma_3, \quad q_3 = \gamma_1\gamma_2\gamma_3
\end{align}
and from (\ref{Ey}) we have
\begin{align}
E_{2} &= (q_1^2 -p_1^2) - 2(q_2q_0 - p_2p_0)\ge 0,\\
E_{4} &= (q_2^2 -p_2^2) - 2(q_3q_1 - p_3p_1) \ge 0,\\
E_{6} &= q_3^2 -p_3^2 \ge 0.
\end{align}

By (\ref{qq}) it follows $E_2 = 0$, and by SA we have $R(\infty) = 0$, and from (\ref{SFs}) we get $p_3 = 0$ and then $E_{6} = q_3^2 \ge 0$.  
Now we compute $E_4$, and by (\ref{pp}-\ref{qq}) we get $E_{4} = (q_2^2 -p_2^2) - 2q_3q_1$. From $p_3 = 0$, it follows 
\[
p_3 = \frac{1}{6} - \frac{q_1}{2} + q_2
\]
and substituting we obtain $E_4 = 8q_1  - 12q_2 -3\ge 0$.

Now if we substitute the quantities (\ref{qq}) in $E_4$, and using (\ref{Cond3}) we get a function that depends on $\gamma_1$ and $c_2$
\[
-	\frac{S}{(3\gamma_1 - 1)(\gamma_1-1)^2(c_2-1)}\ge 0
\]
with 
\[
S = (108c_2^2-72c_2+18)\gamma_1^4+(-144c_2^2+105c_2-33)\gamma_1^3+(84c_2^2-69c_2+24)\gamma_1^2+(-24c_2^2+21c_2-7)\gamma_1+3c_2^2-3c_2+1.
\]
The funciton $S$ is always positive for $\gamma_1 = c_2 \ge 0$, then we have that $E_4 \ge 0$, if 
\[
\gamma_1 \le 1/3, \> c_2\ge1, \> \textrm{or} \>
\gamma_1 \ge 1/3, \> c_2\le1 .
\]

Now in order to justify the requirement to choose $\gamma$ in the intervals (\ref{intervals}) we consider again $E_4 = 8q_1  - 12q_2 -3\ge 0$.

In (\ref{DIRK3bis}) with $\gamma_1 = \gamma_3 = \gamma$, we compute from (\ref{Cond3}) $b_2$, $c_2$ and $\gamma_2$
as functions of  $\gamma$:
\begin{equation}\label{Cond3_bis}
b_2 = 
-\frac{3}{4}\frac{(2\gamma^2-4\gamma + 1)^2}{3\gamma^3 - 9\gamma^2 + 6 \gamma -1}, \quad c2=\frac{1}{3}\frac{(6\gamma^2-9\gamma+2}{(2\gamma^2-4\gamma+1)}, \quad
\gamma_2 =\frac{1}{3}\frac{(6\gamma^2-6\gamma+1}{(2\gamma^2-4\gamma+1)}
\end{equation}
and $b_1 = 1-b_2-\gamma$.
Furthermore it follows: $q_1 = (2\gamma + \gamma_2)$ and $q_2 = (2\gamma_2 \gamma + \gamma^2)$ and substituting this values in $E_4$ we get
\begin{equation}\label{g2}
\gamma_2 \ge \frac{3-16\gamma + 12\gamma^2}{8-24\gamma}.
\end{equation}

Now substituting $\gamma_2$ from (\ref{Cond3_bis}) in (\ref{g2}) and solving this inequality for $\gamma$ we get (\ref{intervals}).

\section{Conservation Error Estimates for discrete moments}\label{Conservation error Estimates}

In this section, we carry out some elementary conservation error estimate for each of the schemes derived so far. For simplicity we denote macroscopic moments
$m_i^n:=(\rho_{i}^{n},\rho_{i}^{n}U_{i}^{n},E_{i}^{n})^T$, final time $T^f$ and tolerance $tol$.

Now, we give discrete conservation error estimates with high order C-SL schemes with the discrete Maxwellian.
\begin{proposition}\label{proposition dirk}
	In the periodic boundary condition, conservation error estimates for the DIRK scheme of order $s=1,2,3$ in mass, momentum and energy are given by
	%\begin{enumerate}
	%	\item
	\begin{align*}
	\left\| \sum_{i=1}^{N_x} (m_i^{N_t}
	- m_i^{0})\Delta x \right\|_\infty
	&\leq \left(\sum_{k=1}^{s-1} |b_{k}| +|b_s|\right)\frac{ N_t\Delta{t}}{ \kappa+ b_s\Delta{t} } (x_{max}-x_{min}) tol.
	\end{align*}	
	where $b_k$, $k=1,...,s$ are determined for each $s=1,2,3$
	%	\item In the freeflow boundary condition
	%	\begin{align*}
	%		&\| \sum_{i=1}^{N_x} (m_i^{N_t}
	%		- m_i^{0})\Delta x +  \sum_{r=1}^{N_t} \sum_{j=0}^{N_v} \left(\sum_{k=1}^{s} b_{k} \left(\widehat{F}_{N_x+\frac{1}{2},j}^{(k),r} -\widehat{F}_{\frac{1}{2},j}^{(k),r}\right)  \right) \phi(v_j) \Delta v \Delta t \|_\infty \\
	%		&\qquad \qquad \qquad \qquad \ \qquad \leq \left(\sum_{k=1}^{s-1} |b_{k}| +|b_s|\right)\frac{N_t \Delta t}{ \kappa+ b_s\Delta{t} } (x_{max}-x_{min}) tol.
	%	\end{align*}	
	%\end{enumerate}	
\end{proposition}
\begin{proof}
	The DIRK scheme of order $s$ is given by
	\begin{align}\label{dirk prop}
	f_{i,j}^{n+1}= {f^*}_{i,j}^{n+1} +  \frac{\Delta t}{\kappa} \sum_{k=1}^{s} b_{k} \left({d\mathcal{M}}_{i,j}^{(k),n+1} - f_{i,j}^{(k),n+1} \right),
	\end{align}
	where ${f^*}_{i,j}^{n+1}={f}_{i,j}^{n} - \sum_{k=1}^{s} b_{k} \frac{\Delta t}{\Delta x}\left(\widehat{F}_{i+\frac{1}{2},j}^{(k),n} -\widehat{F}_{i-\frac{1}{2},j}^{(k),n}\right)$ , ${d\mathcal{M}}_{i,j}^{(s),n+1}={d\mathcal{M}^*}_{i,j}^{(s),n+1}$ and ${f}_{i,j}^{(s),n+1}$ is the precomputed value at $t=t^{n+1}$ for each grid points using classical schemes. From (\ref{dirk prop}), we have
	\begin{equation}\label{dirk prop sum}
	\begin{split}
	&\left(1+ b_s\frac{\Delta{t}}{\kappa}\right)\sum_{i=1}^{N_x}\sum_{j=0}^{N_v}\left(f_{i,j}^{n+1} - {f^*}_{i,j}^{n+1} \right) \phi(v_j)\Delta v \Delta x\\
	&=\sum_{i=1}^{N_x}\sum_{j=0}^{N_v} \left[\frac{\Delta t}{\kappa} \sum_{k=1}^{s-1} b_{k} \left({d\mathcal{M}}_{i,j}^{(k),n+1} - f_{i,j}^{(k),n+1} \right) + \frac{\Delta t}{\kappa}b_s \left({d\mathcal{M}}_{i,j}^{(s),n+1} - {f^*}_{i,j}^{n+1} \right) \right]  \phi(v_j)\Delta v \Delta x.
	\end{split}
	\end{equation}
	Using
	\begin{align*}
	\max_{i} \left\|\sum_{j=0}^{N_v} \left[ {f}_{i,j}^{(k),n+1}-{d\mathcal{M}}_{i,j}^{(k),n+1}  \right]  \phi(v_j) \Delta v \right\|_\infty
	&\leq  tol,
	\end{align*}
	we can get from (\ref{dirk prop sum})
	\begin{align*}
	&\left\| \left(1+ b_s\frac{\Delta{t}}{\kappa}\right)  \sum_{i=1}^{N_x}\sum_{j=0}^{N_v}\left(f_{i,j}^{n+1} - {f^*}_{i,j}^{n+1} \right) \phi(v_j)\Delta v \Delta x \right\|_\infty \cr
	&\leq \frac{\Delta t}{\kappa}\sum_{i=1}^{N_x}\left(\sum_{k=1}^{s-1} |b_{k}| +|b_s|\right) tol  \Delta x\cr
	&= \left(\sum_{k=1}^{s-1} |b_{k}| +|b_s|\right) \frac{\Delta{t}}{\kappa} (x_{max}-x_{min}) tol.
	\end{align*}
	For any dirk scheme of order $s$, we have $b_s>0$, which implies
	\begin{align*}
	& \left\|\sum_{i=1}^{N_x}\left(f_{i,j}^{n+1} - {f^*}_{i,j}^{n+1} \right) \phi(v_j)\Delta v \Delta x \right\|_\infty \leq  \left(\sum_{k=1}^{s-1} |b_{k}| +|b_s|\right)\frac{ \Delta{t}}{ \kappa+ b_s\Delta{t} } (x_{max}-x_{min}) tol.
	\end{align*}
	%For the estimate of conservation error, we divide the proof into two cases:
	%\begin{enumerate}
	%	\item
	Moreover, the periodic boundary condition gives
	\begin{align*}
	\sum_{i=1}^{N_x}\sum_{j=0}^{N_v}{f^*}_{i,j}^{n+1} \phi(v_j)\Delta v \Delta x
	&=\sum_{i=1}^{N_x}\sum_{j=0}^{N_v} f_{i,j}^n  \phi(v_j)\Delta v \Delta x,
	\end{align*}	
	and hence
	\begin{align*}
	& \left\| \sum_{i=1}^{N_x} (m_i^{n+1}
	- m_i^{n})\Delta x \right\|_\infty \leq \left(\sum_{k=1}^{s-1} |b_{k}| +|b_s|\right)  \frac{\Delta{t}}{\kappa + b_s\Delta{t}} (x_{max}-x_{min}) tol.
	\end{align*}
	Finally, we can conclude that
	\begin{align*}
	\left\| \sum_{i=1}^{N_x} (m_i^{N_t}
	- m_i^{0})\Delta x \right\|_\infty &\leq  \sum_{i=1}^{N_x}\sum_{r=1}^{N_t}  \left\| (m_i^{r}
	- m_i^{r-1})\Delta x \right\|_\infty   \cr
	&\leq \left(\sum_{k=1}^{s-1} |b_{k}| +|b_s|\right)\frac{ N_t\Delta{t}}{ \kappa+ b_s\Delta{t} } (x_{max}-x_{min}) tol.
	\end{align*}	
	%	\item In the freeflow boundary condition, we have
	%	\begin{align*}
	%	\sum_{i=1}^{N_x}\sum_{j=0}^{N_v}{f^*}_{i,j}^{n+1} \phi(v_j)\Delta v \Delta x
	%	&=\sum_{i=1}^{N_x}\sum_{j=0}^{N_v} f_{i,j}^n  \phi(v_j)\Delta v \Delta x\\
	%	&\quad  - \sum_{j=0}^{N_v} \left(\sum_{k=1}^{s} b_{k} \frac{\Delta t}{\Delta x}\left(\widehat{F}_{N_x+\frac{1}{2},j}^{(k),n+1} -\widehat{F}_{\frac{1}{2},j}^{(k),n+1}\right)  \right) \phi(v_j)\Delta v \Delta x.
	%	\end{align*}	
	%	Similarly, we can see the desired result
	%	\begin{align*}
	%	&\| \sum_{i=1}^{N_x} (m_i^{N_t}
	%	- m_i^{0})\Delta x +  \sum_{r=1}^{N_t} \sum_{j=0}^{N_v} \left(\sum_{k=1}^{s} b_{k} \left(\widehat{F}_{N_x+\frac{1}{2},j}^{(k),r} -\widehat{F}_{\frac{1}{2},j}^{(k),r}\right)  \right) \phi(v_j) \Delta v \Delta t \|_\infty \\
	%	&\qquad \qquad \qquad \qquad \qquad \leq \left(\sum_{k=1}^{s-1} |b_{k}| +|b_s|\right)\frac{N_t \Delta t}{ \kappa+ b_s\Delta{t} } (x_{max}-x_{min}) tol.
	%	\end{align*}
	%\end{enumerate}	
\end{proof}

Similar results hold for the BDF methods, which can be derived from similar (but more tedius) argument. We present it without proof.
\begin{proposition}\label{proposition bdf}
	In the periodic boundary condition, conservation error estimates for the BDF scheme of order $s=2,3$ in mass, momentum and energy are given by	
	\begin{align}\label{bdf estimate}
	\begin{split}
	&\left\| \sum_{i=1}^{N_x} (m_i^{N_t}
	- m_i^{0})\Delta x \right\|_\infty \\
	&\qquad \qquad \leq   \gamma_s \bigg( \frac{ (N_t-s)\beta_s \Delta{t}}{ \kappa+ \beta_s\Delta{t} } +  \left(\sum_{k=1}^{s-1} |b_{k}| +|b_s|\right)\frac{s\Delta{t}}{ \kappa+ b_s\Delta{t} }\bigg)    (x_{max}-x_{min}) tol ,			
	\end{split}	
	\end{align}	
	where $b_k$, $k=1,...,s$ and $\beta_s$ are determined for each $s=2,3$ and $\gamma_2=\frac{3}{2}$ and $\gamma_3=\frac{146}{11}$.
\end{proposition}

\bibliographystyle{amsplain}
\bibliography{references}

\end{document}